\documentclass{article}

\usepackage[utf8]{inputenc}
\usepackage{amsmath, amssymb, amsthm, bm, mathrsfs, dutchcal}
\usepackage[top=1.5cm, bottom=1.5cm, left=2cm, right=2cm]{geometry}
\usepackage{csquotes, authblk, cases, enumitem}
\usepackage{graphicx, float, subcaption, adjustbox}
\usepackage[hidelinks]{hyperref}
\usepackage[numbers]{natbib}
\usepackage[section, plain]{algorithm}
\usepackage{algpseudocode}
\usepackage{array, tabularx, colortbl}
\usepackage{color, xcolor}
\usepackage[para]{footmisc}
\usepackage[normalem]{ulem}
\counterwithin{figure}{section}
\counterwithin{table}{section}
\numberwithin{equation}{section}
\graphicspath{{Figures/}}
\definecolor{darkred}{rgb}{0.8,0,0}
\DeclareMathOperator{\vectorization}{\operatorname{vec}}

\DeclareMathOperator{\Span}{\operatorname{span}}

\DeclareMathOperator{\diag}{\operatorname{diag}}

\DeclareMathOperator{\arccot}{arccot}
\newenvironment{frcseries}{\fontfamily{frc}\selectfont}{}

\theoremstyle{definition}
\newtheorem{definition}{Definition}[section]
\newtheorem{remark}[definition]{Remark}
\newtheorem{example}[definition]{Example}
\newtheorem{theorem}[definition]{Theorem}
\newtheorem{lemma}[definition]{Lemma}
\newtheorem{corollary}[definition]{Corollary}

\title{A mathematical theory for mass lumping and its generalization with applications to isogeometric analysis}
\author[1]{Yannis Voet \thanks{yannis.voet@epfl.ch}}
\author[1]{Espen Sande \thanks{espen.sande@epfl.ch}}
\author[1]{Annalisa Buffa \thanks{annalisa.buffa@epfl.ch}}
\affil[1]{\small MNS, Institute of Mathematics, École polytechnique fédérale de Lausanne, Station 8, CH-1015 Lausanne, Switzerland}
\date{\today}

\begin{document}

\maketitle

\begin{abstract}
Explicit time integration schemes coupled with Galerkin discretizations of time-dependent partial differential equations require solving a linear system with the mass matrix at each time step. For applications in structural dynamics, the solution of the linear system is frequently approximated through so-called mass lumping, which consists in replacing the mass matrix by some diagonal approximation. Mass lumping has been widely used in engineering practice for decades already and has a sound mathematical theory supporting it for finite element methods using the classical Lagrange basis. However, the theory for more general basis functions is still missing. Our paper partly addresses this shortcoming. Some special and practically relevant properties of lumped mass matrices are proved and we discuss how these properties naturally extend to banded and Kronecker product matrices whose structure allows to solve linear systems very efficiently. Our theoretical results are applied to isogeometric discretizations but are not restricted to them. \\

\end{abstract}

\section{Introduction and background}
\label{se: introduction}
Mass lumping dates from the early days of the finite element method, in the 1970s, when practitioners were facing the burden of inverting non-diagonal mass matrices in explicit time integration schemes for solving wave propagation problems. They quickly discovered that reasonably good solutions could still be recovered after replacing the mass matrix by some easily invertible (typically diagonal) ad hoc approximation, a procedure which became known as mass lumping. The mass lumping approach relied on physical intuition of mass conservation \citep{zienkiewicz2005finite, hughes2012finite}. Numerous strategies were proposed but all later attention in modern literature has focused on only three of them: the row-sum technique, the nodal quadrature method and a so-called ``special mass lumping'' scheme \cite{hughes2012finite}, more commonly referred to as the diagonal scaling procedure \cite{duczek2019mass}. 

As the name suggests, the row-sum technique consists in constructing a diagonal matrix where the $i$th diagonal entry is simply the sum of all entries on the $i$th row. Apart from being straightforward to implement, it is widely used and appreciated for increasing the critical time step of explicit time integration schemes, a property we will formally prove in this article. However, the row-sum technique has major limitations. First of all, it is prone to failure by generating singular or indefinite lumped mass matrices. While the former are simply not applicable in standard explicit time stepping schemes, the latter might lead to erroneous dynamical behaviors. Nonnegative partition of unity methods, such as isogeometric analysis \citep{hughes2005isogeometric, cottrell2009isogeometric}, are exempt from such breakdowns and have sparked renewed interest in the row-sum technique \citep{cottrell2006isogeometric, leidinger2020explicit}. Unfortunately, the row-sum technique also has detrimental effects on the convergence rate of the smallest generalized eigenvalues by downgrading it to second order. This phenomenon has been consistently observed and proved for isogeometric analysis in \cite{cottrell2006isogeometric} for 1D model problems and low order B-splines. To this date, high order mass lumping in isogeometric analysis is an open research problem and has only been achieved (to our knowledge) in \cite{anitescu2019isogeometric} using dual basis functions.

Another mass lumping technique, albeit specific to classical finite element methods, is the nodal quadrature method. This method consists in choosing as finite element nodes the quadrature nodes of a (near) optimal quadrature rule and cleverly using the interpolation properties of the Lagrange basis. Although this technique might preserve the higher order convergence of the consistent mass \citep{fried1975finite, cohen1994higher}, it also breaks down in case nonpositive quadrature weights are used. Positive weights can be guaranteed by either choosing specific quadrature rules and sacrificing some accuracy \citep{duczek2019mass, durufle2009influence} or enlarging the finite element space by introducing additional degrees of freedom \cite{cohen2001higher}. Both solutions come with their own drawbacks: the first one does not preserve the convergence rate while the second one requires solving a larger system of equations. In either case, research is hampered by tedious element specific analysis. Other authors, instead, have suggested modifying the time integration scheme to accommodate nonpositive weights. In \cite{malkus1986zero}, a mixed implicit-explicit time integration scheme is designed to cope with positive semidefinite (potentially singular) lumped mass matrices.

Finally, the diagonal scaling method \cite{hinton1976note} was designed to circumvent breakdowns in the row-sum and nodal quadrature methods. Positive definite lumped mass matrices are indeed guaranteed by construction. The method scales the diagonal entries of the consistent mass matrix, which are necessarily positive, by positive factors that depend on the mass of the elements.

More recently, selective mass scaling techniques have also been proposed to mitigate the effect of high frequencies on the critical step size \citep{macek1995mass, olovsson2005selective, stoter2022variationally, tkachuk2013variational}. These technique typically consist in adding a positive semidefinite perturbation term to the mass matrix and are designed to preserve the rigid body motion. Although they are used in conjunction with mass lumping, they generally do not preserve the diagonal structure of the lumped mass matrix and practitioners then resort to iterative schemes for solving linear systems \cite{olovsson2006iterative}. Although very few guidelines are available for the choice of scaling parameters, mass scaling techniques have already been implemented in finite element software packages and are especially popular in the automobile industry.

Despite appealing numerical results, the mathematical theory supporting most of these methods is rather limited and usually restricted to specific discretization methods, orders and element types. Although a completely unified analysis is indeed probably intractable, a framework with the least restrictive assumptions is sought. For this reason, our analysis will be restricted to the row-sum technique because it allows us to draw very general conclusions. It is based on purely algebraic considerations and is therefore independent of the discretization method used. Moreover, although the three aforementioned mass lumping strategies generally yield different lumped mass matrices, they might be identical in some cases, as already proved in \cite{duczek2019mass} for the Lagrange basis. Thus, although our conclusions will generally not apply to other mass lumping strategies, they might still hold under certain circumstances. 

The rest of the paper is structured as follows: in Section \ref{se: explicit_dynamics}, we briefly recall the Galerkin discretization of a model wave equation followed by the direct time integration of the semi-discrete equations using the Newmark family of methods and in particular the central difference method. The success of mass lumping techniques depends on their ability to recover good approximate solutions in a computationally efficient manner. Overall, the problem amounts to studying the generalized eigenvalues of perturbed matrix pairs. Therefore, Section \ref{se: main_results} first brings the reader up to date by recalling well-known theoretical results for generalized eigenproblems. The rest of Section \ref{se: main_results} is devoted to the theoretical analysis of perturbed generalized eigenproblems and conclusions are drawn regarding the stability of explicit time integration schemes in connection with mass lumping. For high dimensional discretizations, we first assume that the mass matrix is expressed as a Kronecker product, a structure inherited by tensor product basis functions. Unfortunately, this is generally not the case in practice for a variety of reasons. Thus, in Section \ref{se: NKP}, a two-level approximation is devised by first computing the nearest Kronecker product of a matrix and then applying our results to the approximation itself. Finally, Section \ref{se: conclusion} concludes our article and summarizes our findings. Numerical experiments are presented throughout the article and closely follow the theory. They serve multiple purposes by illustrating the findings, highlighting their limitations and drawing additional insight. In order to foster applications to other problems, theoretical results are stated for general matrices satisfying certain properties. 

\section{Mass lumping in explicit dynamics}
\label{se: explicit_dynamics}
We consider as model problem the classical wave equation. Let $\Omega \subset \mathbb{R}^d$ be an open connected domain with Lipschitz boundary and let $I=[0,T]$ be the time domain with $T>0$ denoting the final time. We look for $u \colon \Omega \times [0,T] \to \mathbb{R}$ such that 
\begin{align}
 \partial_{tt} u(\mathbf{x},t)-\Delta u(\mathbf{x},t) &=f(\mathbf{x},t) & &\text{ in } \Omega \times (0,T], \label{eq: wave_equation} \\
 u(\mathbf{x},t)&=0 & &\text{ on } \partial \Omega \times (0,T], \nonumber\\
 u(\mathbf{x},0)&=u_0(\mathbf{x}) & &\text{ in } \Omega,  \nonumber\\
 \partial_t u(\mathbf{x},0)&=v_0(\mathbf{x}) & &\text{ in } \Omega, \nonumber
\end{align}
where $u_0$ and $v_0$ are some initial displacement and velocity, respectively, and we prescribe homogeneous Dirichlet boundary conditions for simplicity. A Galerkin discretization of the spatial variables (see for instance \citep{hughes2012finite, quarteroni2009numerical}) leads to solving the semi-discrete problem 
\begin{align}
\label{eq: semi_discrete_pb}
\begin{split}
M\ddot{\mathbf{u}}(t) + K\mathbf{u}(t) &= \mathbf{f}(t) \qquad \text{for } t \in [0,T], \\
\mathbf{u}(0) &= \mathbf{u}_0,\\
\dot{\mathbf{u}}(0) &= \mathbf{v}_0.
\end{split}
\end{align}
where $M, K \in \mathbb{R}^{n \times n}$ are the mass and stiffness matrices, respectively, and $\mathbf{u}(t) \in \mathbb{R}^{n}$ is a vector containing the coefficients of the approximate solution $u_h(\mathbf{x}, t)$ in the basis of some finite dimensional subspace. The time-dependent right-hand side $\mathbf{f}(t) \in \mathbb{R}^{n}$ accounts for the forcing function $f$ and potential non-homogeneous Neumann and Dirichlet boundary conditions. Naturally, the properties of $M$ and $K$ depend on the discretization method. Nevertheless, they are always symmetric with $M$ being positive definite for the standard Galerkin method. Provided Dirichlet boundary conditions are prescribed on some portion of the boundary, which is often the case in applications, $K$ is also positive definite. Therefore, we will assume throughout the paper that both $M$ and $K$ are symmetric positive definite.

There exists a host of numerical methods for approximating the solution of the semi-discrete problem \eqref{eq: semi_discrete_pb}. The Newmark method \cite{newmark1959method}, proposed in 1959 and summarized in Algorithm \ref{algo: Newmark}, has been historically used and still remains very popular today. The time domain $[0,T]$ is first discretized using a uniform time step $\Delta t=\frac{T}{N}$ such that $t_s=s\Delta t, \quad s=0,1,\dots,N$ with $N \in \mathbb{N}^*$. The method then consists in approximating the solution $\mathbf{u}(t)$ as well as its first and second time derivative (interpreted as the displacement, velocity and acceleration, respectively) at the discrete times $t_s$. These quantities are denoted $\mathbf{u}_s$, $\mathbf{v}_s$ and $\mathbf{a}_s$ in Algorithm \ref{algo: Newmark}. 

\begin{algorithm}[ht]
\begin{algorithmic}[1]
\caption{Newmark method}
\label{algo: Newmark}
\Statex \textbf{Input}: Initial conditions $\mathbf{u}_0, \mathbf{v}_0$, matrices $M,K$ and parameters $\beta$ and $\gamma$
\Statex \textbf{Output}: Approximate displacement $\{\mathbf{u}_s\}_{s=0}^N$, velocity $\{\mathbf{v}_s\}_{s=0}^N$ and acceleration $\{\mathbf{a}_s\}_{s=0}^N$
\State Compute $\mathbf{a}_0=M^{-1}(\mathbf{f}(t_0)-K\mathbf{u}_0)$ \label{line: initialization}
\For{$s=0,1,\cdots, N-1$}
    \State Compute $\Tilde{\mathbf{u}}_{s+1}=\mathbf{u}_s+\Delta t \mathbf{v}_s+(\frac{1}{2}-\beta)\Delta t^2 \mathbf{a}_s$
    \State Compute $\Tilde{\mathbf{v}}_{s+1}= \mathbf{v}_s+(1-\gamma)\Delta t \mathbf{a}_s$
    \State Compute $\mathbf{a}_{s+1}=(M + \beta \Delta t^2 K)^{-1}(\mathbf{f}(t_{s+1})-K\Tilde{\mathbf{u}}_{s+1})$ \label{line: linear_system}
    \State Compute $\mathbf{v}_{s+1}=\Tilde{\mathbf{v}}_{s+1}+\gamma \Delta t \mathbf{a}_{s+1}$
    \State Compute $\mathbf{u}_{s+1}=\Tilde{\mathbf{u}}_{s+1}+\beta \Delta t^2 \mathbf{a}_{s+1}$
\EndFor
\end{algorithmic}
\end{algorithm}

The accuracy and stability properties of the Newmark method depend on its two parameters $\beta, \gamma \in [0,1]$. Explicit methods are recovered for $\beta=0$ while implicit unconditionally stable ones arise for $\beta \geq \frac{\gamma}{2} \geq \frac{1}{4}$. The central difference method, obtained for $\beta=0$ and $\gamma=\frac{1}{2}$, is a popular explicit second order method and its stability analysis reveals that the critical time step is given by 
\begin{equation}
\label{eq: delta_t_crit_central_diff}
    \Delta t_c=\frac{2}{\sqrt{\lambda_n(K,M)}}
\end{equation}
where $\lambda_n(K,M)$ is the largest eigenvalue of the generalized eigenvalue problem $K\mathbf{u}=\lambda M \mathbf{u}$. For fast dynamic loadings such as blasts or impacts, the central difference method is usually favored over implicit unconditionally stable versions of the Newmark method. Indeed, rapidly changing loading conditions require a sufficiently small time step in order to resolve the correct physical behavior. Thus, the unconditional stability of some implicit methods cannot be fully exploited while the computational burden prevails in Line 5. Indeed, for very large problems, solving the linear system in Line 5 at each time step becomes a major computational bottleneck. However, for explicit methods ($\beta=0$), this linear system only involves the mass matrix. Therefore, substituting the consistent mass matrix in Lines 1 and 5 with some kind of easily invertible approximation $\tilde{M}$ becomes very appealing. Mass lumping stems from this observation. In fact, practitioners might rely on all kinds of other ad hoc substitutions, including the stiffness matrix and right-hand side vector, but we will not consider those here.

While the computational benefits of mass lumping are obvious, the substitution of $M$ with $\tilde{M}$ both effects the stability of the time integration method and the accuracy of the approximate solution. On the one hand, strong numerical evidence suggests that mass lumping improves the stability by increasing the critical time step \citep{leidinger2020explicit, anitescu2019isogeometric}. Verifying this claim amounts to guaranteeing that $\lambda_n(K,\tilde{M}) \leq \lambda_n(K,M)$. On the other hand, the accuracy of the smallest eigenpairs of $(K,M)$ is critical for the accuracy of the finite element solution, although the largest ones might also sometimes play a role \cite{hughes2014finite}. Therefore, the smallest eigenpairs of $(K,\tilde{M})$ must approximate those of $(K,M)$ very well. Consequently, both questions of stability and accuracy can be tackled by comparing the eigenvalues and eigenspaces of some matrix pencils. The expressions ``matrix pencils'' and ``matrix pairs'' will be used interchangeably in this work. The upcoming analysis requires some theoretical understanding of generalized eigenvalue problems. All necessary background knowledge will be covered in the beginning of the next section.

\section{Mathematical properties of mass lumping and its generalization}
\label{se: main_results}
\subsection{Generalized eigenvalue problems}
\label{se: generalized_eigenproblems}
Given the important role generalized eigenproblems play in this work, this section provides a brief reminder of some important properties. Since all matrices we will be considering for our applications are symmetric, we will restrict the discussion to symmetric generalized eigenvalue problems. Let $A,B \in \mathbb{R}^{n \times n}$ be symmetric matrices. We look for pairs $(\mathbf{u},\lambda) \in \mathbb{C}^n \times \mathbb{C}$ such that $A\mathbf{u}=\lambda B\mathbf{u}$. In analogy to standard eigenproblems, the generalized eigenvalues are the roots of the characteristic polynomial $p(\lambda)=\det(A-\lambda B)$. However, contrary to standard eigenvalue problems, symmetry of the coefficient matrices is not sufficient to guarantee real eigenvalues. This point must be emphasized to avoid common misconceptions. A sufficient condition for real finite eigenvalues, which will cover all our applications, is given in the following classical lemma (see for instance \citep[][Theorem VI.1.15]{stewart1990matrix} for the proof).

\begin{lemma}
\label{lem:classical_problem}
Let $(A,B)$ be a symmetric pencil with $B$ positive definite. Then, all generalized eigenvalues are real and there exists an invertible matrix $U \in \mathbb{R}^{n \times n}$ such that
\begin{equation*}
    U^TAU=D, \qquad U^TBU=I,
\end{equation*}
where $D=\diag(\lambda_1, \dots, \lambda_n)$ is a real diagonal matrix containing the eigenvalues.
\end{lemma}

More can be said if both $A$ and $B$ are symmetric positive definite. The following implications are a straightforward consequence of Lemma \ref{lem:classical_problem}.
\begin{itemize}[noitemsep]
    \item If $A$ is symmetric positive semidefinite and $B$ is symmetric positive definite, then $\Lambda(A,B) \subset [0, +\infty)$,
    \item If $A$ and $B$ are symmetric positive definite, then $\Lambda(A,B) \subset (0, +\infty)$.
\end{itemize}
We conclude this section by introducing a few notations. $\Lambda(A,B)$ will denote the spectrum (i.e. the set of all eigenvalues) of the matrix pencil $(A,B)$ whereas $\lambda(A,B)$ will refer to individual eigenvalues. For standard eigenvalue problems (when $B=I$), we will simply write $\lambda(A)$ instead of $\lambda(A,I)$. The matrices will be explicitly specified when referring to individual eigenvalues or to the entire spectrum if there is potential confusion. The eigenvalues will always be numbered in ascending algebraic order such that
\begin{equation*}
    \lambda_1(A,B) \leq \lambda_2(A,B) \leq \dots \leq \lambda_n(A,B).
\end{equation*}

\subsection{Mathematical properties of mass lumping}
An important concept used throughout the analysis is the notion of equivalent pencils, which is motivated through the following lemma. Its proof is a simple exercise in \cite[][Chapter 15]{parlett1998symmetric}.

\begin{lemma}
\label{lem:equivalent_pencils}
Let $A,B,E,F \in \mathbb{R}^{n \times n}$ and assume $E$ and $F$ are invertible. Then
\begin{equation*}
    \Lambda(A,B)=\Lambda(EAF, EBF)
\end{equation*}
\end{lemma}
\begin{proof}
Indeed,
\begin{align*}
    \det(EAF-\lambda EBF)&=0 \\
    \iff \det(E)\det(A-\lambda B)\det(F)&=0
\end{align*}
Since we have assumed $E$ and $F$ are invertible ($\det(E) \neq 0$ and $\det(F) \neq 0$), then $\det(EAF-\lambda EBF)=0 \iff \det(A-\lambda B)=0$. Consequently, $\Lambda(A,B)=\Lambda(EAF, EBF)$.
\end{proof}

\begin{definition}[Equivalent pencils]
For matrices $A,B,E,F \in \mathbb{R}^{n \times n}$ with $E,F$ invertible, the matrix pencils $(A,B)$ and $(EAF, EBF)$ are said to be \textit{equivalent}. They are called \textit{congruent} if $E=F^T$.
\end{definition}

The next theorem will play a central role in our analysis. Its proof is a direct consequence of the celebrated Courant-Fischer theorem. The result can also be alternatively deduced starting from Ostrowski's theorem (see for instance \citep[][Theorem 4.5.9]{horn2012matrix}).

\begin{theorem}
\label{th:eig_bounds}
Let $A,B,C \in \mathbb{R}^{n \times n}$ be symmetric positive definite matrices and let all eigenvalues be numbered in ascending algebraic order. Then
\begin{subequations}
\begin{align}
    \lambda_k(A,C)\lambda_1(C,B) &\leq \lambda_k(A,B) \leq \lambda_k(A,C)\lambda_n(C,B) \qquad 1 \leq k \leq n, \label{eq: ineq1}\\
    \lambda_1(A,C)\lambda_k(C,B) &\leq \lambda_k(A,B) \leq \lambda_n(A,C)\lambda_k(C,B) \qquad 1 \leq k \leq n. \label{eq: ineq2}
\end{align}
\end{subequations}
\end{theorem}
\begin{proof}
Let $\mathbf{u}_i$, $\mathbf{v}_i$ and $\mathbf{w}_i$ denote the eigenvectors associated to the eigenvalues $\lambda_i(A,B)$, $\lambda_i(A,C)$ and $\lambda_i(C,B)$, respectively. We define the subspaces
\begin{align*}
    \mathcal{U}_k &= \Span \{\mathbf{u}_1, \dots, \mathbf{u}_k \}, & \mathcal{U}'_k &= \Span \{\mathbf{u}_k, \dots, \mathbf{u}_n \}, \\
    \mathcal{V}_k &= \Span \{\mathbf{v}_1, \dots, \mathbf{v}_k \}, & \mathcal{V}'_k &= \Span \{\mathbf{v}_k, \dots, \mathbf{v}_n \}, \\
    \mathcal{W}_k &= \Span \{\mathbf{w}_1, \dots, \mathbf{w}_k \}, & \mathcal{W}'_k &= \Span \{\mathbf{w}_k, \dots, \mathbf{w}_n \}.
\end{align*}
Since all matrices are assumed symmetric positive definite, the generalized eigenvectors of all matrix pencils are linearly independent (see Lemma \ref{lem:classical_problem}). Consequently, $\dim(\mathcal{U}_k)=\dim(\mathcal{V}_k)=\dim(\mathcal{W}_k)=k$ and $\dim(\mathcal{U}'_k)=\dim(\mathcal{V}'_k)=\dim(\mathcal{W}'_k)=n-k+1$. By the Courant-Fischer theorem (see for instance \cite[][Theorem 4.2.6]{horn2012matrix})
\begin{equation*}
    \lambda_k(A,B) = \min_{\substack{\mathcal{S} \subseteq \mathbb{C}^n \\ \dim(\mathcal{S})=k}} \max_{\substack{ \mathbf{x} \in \mathcal{S} \\ \mathbf{x} \neq \mathbf{0}}} \frac{\mathbf{x}^T A\mathbf{x}}{\mathbf{x}^T B \mathbf{x}} = \max_{\substack{\mathcal{S} \subseteq \mathbb{C}^n \\ \dim(\mathcal{S})=n-k+1}} \min_{\substack{ \mathbf{x} \in \mathcal{S} \\ \mathbf{x} \neq \mathbf{0}}} \frac{\mathbf{x}^T A\mathbf{x}}{\mathbf{x}^T B \mathbf{x}} \qquad 1 \leq k \leq n.
\end{equation*}
We already know that the minimum in the second expression is attained for $\mathcal{S}=\mathcal{U}_k$ and the maximum in the third expression is attained for $\mathcal{S}=\mathcal{U}'_k$. Moreover, we note that
\begin{equation*}
    \lambda_k(A,B) = \min_{\substack{\mathcal{S} \subseteq \mathbb{C}^n \\ \dim(\mathcal{S})=k}} \max_{\substack{ \mathbf{x} \in \mathcal{S} \\ \mathbf{x} \neq \mathbf{0}}} \frac{\mathbf{x}^T A\mathbf{x}}{\mathbf{x}^T C \mathbf{x}}\frac{\mathbf{x}^T C\mathbf{x}}{\mathbf{x}^T B \mathbf{x}}
    \leq \max_{\substack{ \mathbf{x} \in \mathcal{X}_k \\ \mathbf{x} \neq \mathbf{0}}} \frac{\mathbf{x}^T A\mathbf{x}}{\mathbf{x}^T C \mathbf{x}}\frac{\mathbf{x}^T C\mathbf{x}}{\mathbf{x}^T B \mathbf{x}} 
    \leq \max_{\substack{ \mathbf{x} \in \mathcal{X}_k \\ \mathbf{x} \neq \mathbf{0}}} \frac{\mathbf{x}^T A\mathbf{x}}{\mathbf{x}^T C \mathbf{x}} \max_{\substack{ \mathbf{x} \in \mathcal{X}_k \\ \mathbf{x} \neq \mathbf{0}}} \frac{\mathbf{x}^T C\mathbf{x}}{\mathbf{x}^T B \mathbf{x}},
\end{equation*}
and
\begin{equation*}
    \lambda_k(A,B) = \max_{\substack{\mathcal{S} \subseteq \mathbb{C}^n \\ \dim(\mathcal{S})=n-k+1}} \min_{\substack{ \mathbf{x} \in \mathcal{S} \\ \mathbf{x} \neq \mathbf{0}}} \frac{\mathbf{x}^T A\mathbf{x}}{\mathbf{x}^T C \mathbf{x}}\frac{\mathbf{x}^T C \mathbf{x}}{\mathbf{x}^T B \mathbf{x}}
    \geq \min_{\substack{ \mathbf{x} \in \mathcal{X}'_k \\ \mathbf{x} \neq \mathbf{0}}} \frac{\mathbf{x}^T A\mathbf{x}}{\mathbf{x}^T C \mathbf{x}}\frac{\mathbf{x}^T C\mathbf{x}}{\mathbf{x}^T B \mathbf{x}} 
    \geq \min_{\substack{ \mathbf{x} \in \mathcal{X}'_k \\ \mathbf{x} \neq \mathbf{0}}} \frac{\mathbf{x}^T A\mathbf{x}}{\mathbf{x}^T C \mathbf{x}} \min_{\substack{ \mathbf{x} \in \mathcal{X}'_k \\ \mathbf{x} \neq \mathbf{0}}} \frac{\mathbf{x}^T C\mathbf{x}}{\mathbf{x}^T B \mathbf{x}},
\end{equation*}
for any given subspaces $\mathcal{X}_k \subseteq \mathbb{C}^n$ and $\mathcal{X}'_k \subseteq \mathbb{C}^n$ with $\dim(\mathcal{X}_k)=k$ and $\dim(\mathcal{X}'_k)=n-k+1$. The strategy now consists in cleverly choosing these subspaces. In particular, for the upper bound, we obtain
\begin{align*}
    \lambda_k(A,B) &\leq \lambda_k(A,C) \lambda_n(C,B) \qquad \text{for } \mathcal{X}_k=\mathcal{V}_k, \\
    \lambda_k(A,B) &\leq \lambda_n(A,C) \lambda_k(C,B) \qquad \text{for } \mathcal{X}_k=\mathcal{W}_k,
\end{align*}
and for the lower bound
\begin{align*}
    \lambda_k(A,B) &\geq \lambda_k(A,C) \lambda_1(C,B) \qquad \text{for } \mathcal{X}'_k=\mathcal{V}'_k, \\
    \lambda_k(A,B) &\geq \lambda_1(A,C) \lambda_k(C,B) \qquad \text{for } \mathcal{X}'_k=\mathcal{W}'_k.
\end{align*}
\end{proof}

The last theorem is useful in understanding the relationship between the eigenvalues of $(K,M)$ and those of $(K,\tilde{M})$ following a worst case analysis. Provided $\tilde{M}$ is symmetric positive definite, a straightforward application of Theorem \ref{th:eig_bounds}, inequalities \eqref{eq: ineq1}, with $A=K$, $B=\tilde{M}$ and $C=M$ leads to
\begin{equation}
   \lambda_1(M,\tilde{M}) \leq \frac{\lambda_k(K,\tilde{M})}{\lambda_k(K,M)} \leq \lambda_n(M,\tilde{M}) \qquad 1 \leq k \leq n. \label{eq:eig_pert_bounds_mass}
\end{equation}
An immediate conclusion can be drawn from these bounds: if $\tilde{M}$ is very good preconditioner for $M$, then the entire spectrum is well approximated since both the lower and upper bounds are close to $1$. They reveal the close connection between the quality of a preconditioner and the approximation of the spectrum. However, not all preconditioners are appropriate for our problem. Indeed, one of our main objectives when designing a preconditioner $\tilde{M}$ is to ensure that $\lambda_k(K,\tilde{M}) \leq \lambda_k(K,M)$ for all $k=1,\dots,n$. This requirement is sometimes easily fulfilled by simply understanding the relation between $M$ and $\tilde{M}$ in the Loewner partial order.  

\begin{definition}[Loewner partial order]
For two symmetric matrices $A,B \in \mathbb{R}^{n \times n}$, we write $A \succeq B$ (respectively $A \succ B$) if $A-B$ is positive semidefinite (respectively positive definite).
\end{definition}

The next lemma gathers some equivalent sufficient conditions guaranteeing that $\lambda_k(K,\tilde{M}) \leq \lambda_k(K,M)$ for all $k=1,\dots,n$.

\begin{corollary}
\label{cor:equivalence_conditions}
Let $A, B,\tilde{B} \in \mathbb{R}^{n \times n}$ be symmetric positive definite matrices, and denote $E=\tilde{B}-B$. Then the statements
\begin{enumerate}[noitemsep]
    \item $E \succeq 0$, \label{eq: eq_cond_1}
    \item $\tilde{B} \succeq B$, \label{eq: eq_cond_2}
    \item $\Lambda(B,\tilde{B}) \subset (0,1]$, \label{eq: eq_cond_3}
\end{enumerate}
are all equivalent and imply that $\lambda_k(A,\tilde{B}) \leq \lambda_k(A,B)$ for all $k=1,\dots,n$.
\end{corollary}
\begin{proof}
We first prove that all statements are equivalent:
\newline \eqref{eq: eq_cond_1} $\iff$ \eqref{eq: eq_cond_2}: $\tilde{B} \succeq B \iff \tilde{B}-B = E \succeq 0$.
\newline \eqref{eq: eq_cond_2} $\iff$ \eqref{eq: eq_cond_3}: Since $B$ and $\tilde{B}$ are symmetric positive definite, $\lambda_k(B,\tilde{B})>0$ for all $k=1,\dots,n$. Moreover, $\tilde{B} \succeq B \iff \mathbf{u}^T \tilde{B} \mathbf{u} \geq \mathbf{u}^T B \mathbf{u} \ \forall \mathbf{u} \in \mathbb{R}^n \iff \lambda_n(B, \tilde{B}) \leq 1$ where the last equivalence follows from the Courant-Fischer theorem. 

We now show that $\lambda_k(A,\tilde{B}) \leq \lambda_k(A,B)$ for all $k=1,\dots,n$. Using for instance the third characterization combined with Theorem \ref{th:eig_bounds}, inequalities \eqref{eq: ineq1}, we obtain
\begin{equation*}
    \lambda_k(A,\tilde{B}) \leq \lambda_k(A,B)\lambda_n(B,\tilde{B}) \leq \lambda_k(A,B) \qquad 1 \leq k \leq n.
\end{equation*}
\end{proof}

As we will show, these conditions are actually satisfied for an appropriate definition of mass lumping.

\begin{definition}[Lumping operator]
\label{def: lumping}
Let $B \in \mathbb{R}^{n \times n}$. The lumping operator $\mathcal{L} \colon \mathbb{R}^{n \times n} \to \mathbb{R}^{n \times n}$ is defined as
\begin{equation*}
\mathcal{L}(B)=\diag(d_1,\dots,d_n)
\end{equation*}
where $d_i=\sum_{j=1}^n |b_{ij}|$ for $i=1,\dots,n$.
\end{definition}

\begin{remark}
Definition \ref{def: lumping} defines more specifically the row-sum lumping operator. Some readers might be concerned about the absolute values included in our definition. They are absolutely necessary for our results to hold as they guarantee positive definiteness of the lumped mass matrix provided $B$ does not contain a row of zeros. For nonnegative matrices (i.e. matrices with nonnegative entries), our definition obviously coincides with the usual row-sum mass lumping. Some typical examples include isogeometric analysis \cite{hughes2005isogeometric}, classical low order finite element methods and nonnegative partition of unity methods applicable to generalized or extended finite element methods \cite{schweitzer2013variational}. Even for more classical methods, nonnegative mass matrices can still be constructed for special sets of basis functions and finite elements \cite{yang2017rigorous}.
\end{remark}

The next lemma is a first step in analyzing the preconditioned spectrum $(M,\mathcal{L}(M))$.

\begin{lemma}
\label{lem:spectrum_row_sum_prec}
Let $B \in \mathbb{R}^{n \times n}$ be symmetric positive definite. Then
\begin{enumerate}[noitemsep]
    \item $\Lambda(B, \mathcal{L}(B)) \subset (0,1]$,
    \item If $B$ is nonnegative, then $\lambda_n(B, \mathcal{L}(B))=1$.
\end{enumerate}
\end{lemma}
\begin{proof}
We prove each point below.
\begin{enumerate}[noitemsep]
\item First note that, by construction, if $B$ is symmetric positive definite then $\mathcal{L}(B)$ is also symmetric positive definite and consequently all eigenvalues of the matrix pencil $(B, \mathcal{L}(B))$ are strictly positive. Moreover, $\Lambda(B, \mathcal{L}(B))=\Lambda(\mathcal{L}(B)^{-1}B)$. Denoting $\|B\|_{\infty}=\max_i \sum_{j=1}^n |b_{ij}|$ the infinity norm of a matrix $B$, we obtain
\begin{equation*}
    \lambda_k(B, \mathcal{L}(B))=\lambda_k(\mathcal{L}(B)^{-1}B) \leq \|\mathcal{L}(B)^{-1}B\|_{\infty} = \max_i \frac{1}{\mathcal{L}(B)_{ii}}\sum_{j=1}^n |b_{ij}|=\max_i \frac{\sum_{j=1}^n |b_{ij}|}{\sum_{j=1}^n |b_{ij}|}=1 \quad k=1,\dots,n,
\end{equation*}
showing that $1$ is an upper bound on the spectrum. 
\item Note that if $B$ is nonnegative, then $(1, \mathbf{e})$ is an eigenpair of $(B, \mathcal{L}(B))$, where $\mathbf{e}$ is the vector of all ones. Thus, the upper bound is always attained.
\end{enumerate}
\end{proof}
The next corollary is one of our main results.

\begin{corollary}
\label{cor:eig_interlacing}
Let $A,B \in \mathbb{R}^{n \times n}$ be symmetric positive definite matrices and let all eigenvalues be numbered in ascending algebraic order. Then
\begin{equation*}
    \lambda_k(A, \mathcal{L}(B)) \leq \lambda_k(A,B) \qquad 1 \leq k \leq n.
\end{equation*}
\end{corollary}
\begin{proof}
The result follows immediately from Lemma \ref{lem:spectrum_row_sum_prec} and Corollary \ref{cor:equivalence_conditions}.
\end{proof}

\begin{remark}
Combining Theorem \ref{th:eig_bounds} and Lemma \ref{lem:spectrum_row_sum_prec} also yields the upper bound
\begin{equation*}
    \lambda_k(A,B) \leq \lambda_k(\mathcal{L}(A),B) \qquad 1 \leq k \leq n.
\end{equation*}
In words, lumping the mass matrix leads to an underestimate of the eigenvalues whereas lumping the stiffness matrix leads to an overestimate of the eigenvalues. Clearly, the above inequality is not of any practical interest for our problem and it is only mentioned for completeness.
\end{remark}

In our quest to ensure an underestimate of the eigenvalues, we will heavily rely on the sufficient conditions stated in Corollary \ref{cor:equivalence_conditions}. In fact, the positive semi-definiteness of the error is also the main building block for ad hoc mass scaling strategies \citep{macek1995mass, olovsson2005selective, stoter2022variationally, tkachuk2013variational}. The objectives of mass lumping and mass scaling are similar, however, mass scaling primarily focuses on improving the stability. The issue of solving linear systems with the scaled mass matrix is usually disregarded and at best only addressed in a later design phase. On the contrary, mass lumping addresses both issues simultaneously.

Nevertheless, although the conditions stated in Corollary \ref{cor:equivalence_conditions} are sufficient, they are not necessary, as the following example shows.

\begin{example}
\label{ex: indefinite_error}
Consider the symmetric positive definite matrices
\begin{equation*}
    A=
    \begin{pmatrix}
    6 & 0 \\
    0 & 6
    \end{pmatrix},
    \qquad
    B=
    \begin{pmatrix}
    2 & 1 \\
    1 & 2
    \end{pmatrix},
    \qquad
    \tilde{B}=
    \begin{pmatrix}
    3 & 0 \\
    0 & 2
    \end{pmatrix}.
\end{equation*}
The eigenvalues of $(A,B)$, $(A,\tilde{B})$ and of the error $E=\tilde{B}-B$ are
\begin{align*}
    \lambda_1(A,B)=2, & &\lambda_1(A,\tilde{B})=2, & & \lambda_1(E)=\frac{1-\sqrt{5}}{2}<0, \\
    \lambda_2(A,B)=6, & &\lambda_2(A,\tilde{B})=3, & &\lambda_2(E)=\frac{1+\sqrt{5}}{2}>0.
\end{align*}
Thus, the error $E$ is indefinite and yet $\lambda_k(A,\tilde{B}) \leq \lambda_k(A,B)$ for $k=1,2$. 
\end{example}
The theoretical results are now illustrated on a numerical example.

\begin{example}[Preconditioner comparison]
\label{ex: prec_comp}
We consider an isogeometric discretization of our model problem \eqref{eq: wave_equation} on a stretched square domain already used by previous authors \citep{gao2014fast, loli2021easy}. Quadratic B-splines of maximal smoothness are used with 20 subdivisions in each parametric direction. All our experiments are done with GeoPDEs \cite{vazquez2016new}, an open source Matlab/Octave package for isogeometric analysis. The perturbed spectrum $\Lambda(K,\tilde{M})$ is compared for two different choices of $\tilde{M}$, namely the lumped mass matrix introduced in Definition \ref{def: lumping} and a very good preconditioner developed by Loli et al. \cite{loli2021easy}, which is provably optimal under mesh refinement. The interested reader is referred to the original article for its definition and construction. This experiment has two purposes: firstly, it illustrates the close connection between the quality of a preconditioner and the approximation of the eigenvalues. Secondly, it supports our first theoretical finding regarding mass lumping and pinpoints some characteristic features so far not captured by our bounds.

The eigenvalue ratios together with their lower and upper bounds from \eqref{eq:eig_pert_bounds_mass} are shown in Figure \ref{fig: 2D_Laplace_stretched_square_eig_ratio_LM_vs_Loli_p2_n20} for the lumped mass matrix and for the preconditioner of Loli et al. In Figure \ref{fig: 2D_Laplace_stretched_square_eig_LM_vs_Loli_p2_n20}, the exact discrete spectrum is compared to the approximate spectrum obtained with our two choices of preconditioners. The excellent preconditioner of Loli et al. leads to a nearly perfect approximation of the entire spectrum while mass lumping leads to an underestimate of the eigenvalues, in agreement with Corollary \ref{cor:eig_interlacing}.

\begin{center}
\begin{minipage}[t]{.48\linewidth}
\vspace{0pt}
\begin{figure}[H]
    \centering
    \includegraphics[width=\textwidth]{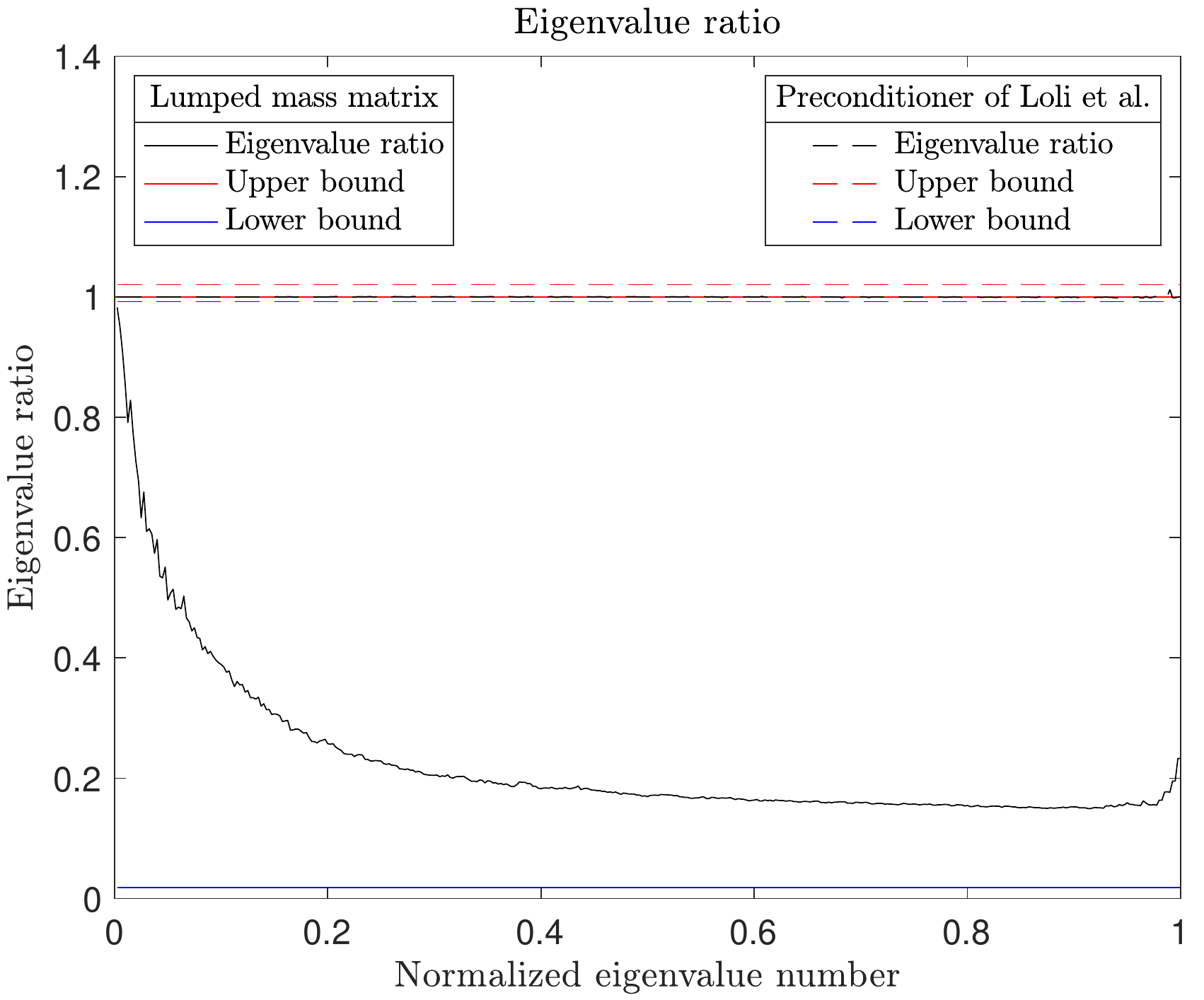}
    \caption{Eigenvalue ratios $\frac{\lambda_k(K,\tilde{M})}{\lambda_k(K,M)}$ when using the lumped mass matrix and the preconditioner of Loli et al. \cite{loli2021easy}}
    \label{fig: 2D_Laplace_stretched_square_eig_ratio_LM_vs_Loli_p2_n20}
\end{figure}
\end{minipage}
\hspace{2pt}
\begin{minipage}[t]{.48\linewidth}
\vspace{0pt}
\begin{figure}[H]
    \centering
    \includegraphics[width=\textwidth]{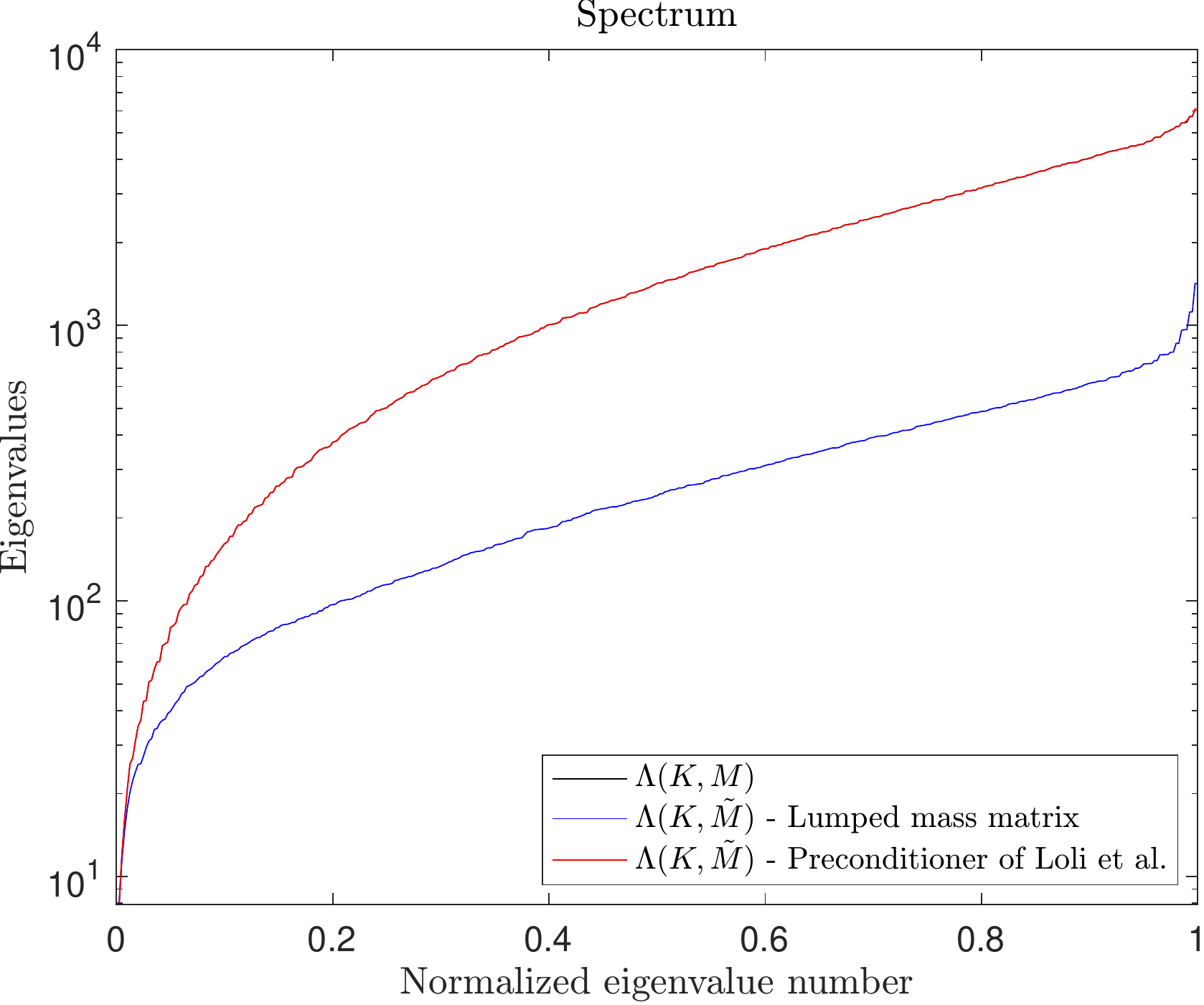}
    \caption{Comparison between $\Lambda(K,M)$ and $\Lambda(K,\tilde{M})$ obtained using the lumped mass matrix and the preconditioner of Loli et al. \cite{loli2021easy}}
    \label{fig: 2D_Laplace_stretched_square_eig_LM_vs_Loli_p2_n20}
\end{figure}
\end{minipage}
\end{center}
\end{example}

Example \ref{ex: prec_comp} also illustrates a typical and widely appreciated feature of mass lumping: In Figure \ref{fig: 2D_Laplace_stretched_square_eig_ratio_LM_vs_Loli_p2_n20}, the smallest eigenvalues are reasonably well approximated, whereas the larger eigenvalues are strongly underestimated. It explains why mass lumping does not lead to completely erroneous approximate solutions. This feature is explored more thoroughly in a second and simpler example.

\begin{example}[$h$-refinement for 1D Laplace]
We perform an $h$-refinement experiment for the isogeometric discretization of the 1D Laplace eigenvalue problem on the unit line with homogeneous Dirichlet boundary conditions. The relative eigenvalue error $\frac{\lambda_i-\tilde{\lambda}_i}{\lambda_i}$ for the first 20 discrete eigenvalues is shown in Figures \ref{fig: 1D_Laplace_relative_err_eig20_h_refinement_p6} for increasingly fine meshes. The spline order is fixed at $p=6$. As expected, for a given spline degree the relative error decreases for increasingly fine meshes but is not even across all eigenvalues: the smallest eigenvalues are always much better approximated.

\begin{figure}[H]
    \centering
    \includegraphics[scale=0.5]{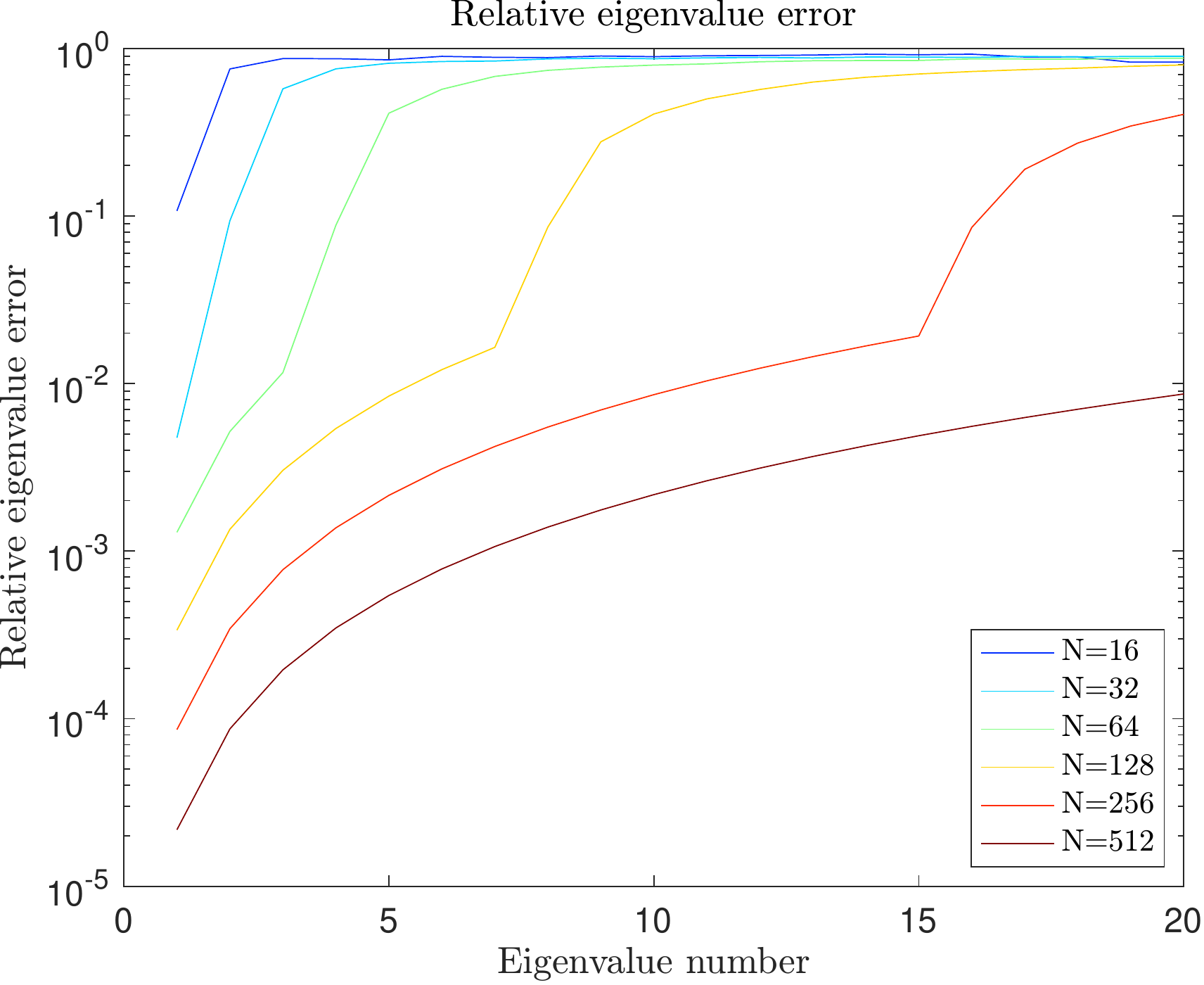}
    \caption{Relative eigenvalue error for the first 20 discrete eigenvalues and $p=6$}
    \label{fig: 1D_Laplace_relative_err_eig20_h_refinement_p6}
\end{figure}
\end{example}

The reasons why mass lumping gives a much better approximation of the smallest eigenvalues can be intuitively explained with the perturbation theory of generalized eigenvalue problems pioneered by Stewart in the early 1970s. For generalized eigenvalue problems, the eigenvalue error is often measured in the chordal metric \citep{stewart1990matrix, stewart1975gershgorin, stewart1979pertubation, sun1982note, crawford1976stable} because its ability to treat small and large eigenvalues uniformly. Error bounds are also alternatively formulated in terms of \textit{eigenangles}, which were first introduced by Stewart in 1979 \cite{stewart1979pertubation} and are related to eigenvalues by
\begin{equation*}
    \theta_i=\arccot(\lambda_i) \qquad 1 \leq i \leq n.
\end{equation*}
Similarly to eigenvalues, eigenangles satisfy a min-max characterization, which is the basis for deriving error bounds on eigenangles. The classical perturbation theory uniformly bounds the quantity $\sin(|\tilde{\theta}_i-\theta_i|)$ by the ratio of some error measure over the definiteness of the pencil. Due to the highly nonlinear nature of the cotangent function, large eigenvalues undergo large perturbations while small eigenvalues undergo small perturbations, which is consistent with the results in Figure \ref{fig: 2D_Laplace_stretched_square_eig_LM_vs_Loli_p2_n20}. Quantitatively speaking though, the perturbations bounds from \cite{stewart1979pertubation} seemed sharper for the smallest eigenvalues than for the largest ones in our numerical experiments. The classical perturbation theory may not be satisfactory for our problem for two main reasons. Firstly, the error $E$ is treated as a random symmetric perturbation whereas in our problem the lumped mass matrix is constructed from the consistent mass matrix and therefore the error inherits some structure, which is ignored in the classical perturbation theory. Secondly, all error bounds we are aware of are \textit{uniform}, meaning the error (in whichever measure is used) is bounded independently of the eigenvalue/eigenangle number similarly to Weyl’s theorem for standard eigenvalue problems \citep[][Equation (6.3.4.1)]{horn2012matrix}. Such bounds cannot explain differences in relative error as we are witnessing in Figure \ref{fig: 1D_Laplace_relative_err_eig20_h_refinement_p6}. Instead, focus should be directed at understanding how the error is interacting with the eigenspace. Surprisingly little attention has been paid to this problem and the few results we are aware of are limited to standard eigenvalue problems \cite{ipsen2009refined}. Such error bounds are not only valuable for mass lumping but also for selective mass scaling, where specific eigenvalues are selectively driven down \cite{olovsson2005selective}. Perturbation bounds that take into account eigenspaces are established in Theorem \ref{th:Bauer_Fike}.

\begin{theorem}
\label{th:Bauer_Fike}
Let $(A,B)$ be a symmetric matrix pencil with $B$ positive definite. Let $(\tilde{A},\tilde{B})=(A+E,B+F)$, where $E,F \in \mathbb{R}^{n \times n}$ are symmetric perturbations and $\tilde{B}$ is positive definite. Let $\{(\lambda_i, \mathbf{u}_i)\}_{i=1}^n$ and $\{(\tilde{\lambda}_i, \tilde{\mathbf{u}}_i)\}_{i=1}^n$ denote the eigenpairs of $(A,B)$ and $(\tilde{A}, \tilde{B})$, respectively, where the eigenvectors $\mathbf{u}_i$ and $\tilde{\mathbf{u}}_i$ are assumed $B$-orthonormal and $\tilde{B}$-orthonormal, respectively. Then
\begin{subequations}
\begin{align}
    &\min_j |\tilde{\lambda}_j-\lambda_i| \leq \frac{1}{\lambda_1(\tilde{B})}\left(\frac{\|E\mathbf{u}_i\|_2}{\|\mathbf{u}_i\|_2} +|\lambda_i| \frac{\|F\mathbf{u}_i\|_2}{\|\mathbf{u}_i\|_2}\right) \qquad 1 \leq i \leq n, \label{eq: Bauer_Fike_1} \\
    &\min_j |\lambda_j-\tilde{\lambda}_i| \leq \frac{1}{\lambda_1(B)}\left(\frac{\|E\tilde{\mathbf{u}}_i\|_2}{\|\tilde{\mathbf{u}}_i\|_2} +|\tilde{\lambda}_i| \frac{\|F\tilde{\mathbf{u}}_i\|_2}{\|\tilde{\mathbf{u}}_i\|_2}\right) \qquad 1 \leq i \leq n. \label{eq: Bauer_Fike_2}
\end{align}
\end{subequations}
\end{theorem}
\begin{proof}
We prove statement \eqref{eq: Bauer_Fike_1} by first noticing that for any given vector $\mathbf{x} \in \mathbb{R}^n$ and any given scalar $\sigma \in \mathbb{R} \setminus \Lambda(\tilde{A},\tilde{B})$, $\mathbf{x}=(\tilde{A}-\sigma \tilde{B})^{-1}(\tilde{A}-\sigma \tilde{B})\mathbf{x}$. Since $\tilde{B}$ is positive definite, there exists a matrix $\tilde{U} \in \mathbb{R}^{n \times n}$ of $\tilde{B}$-orthonormal eigenvectors that diagonalizes $(\tilde{A},\tilde{B})$ such that $\tilde{U}^T\tilde{A}\tilde{U}=\tilde{D}$ and $\tilde{U}^T\tilde{B}\tilde{U}=I$, where $\tilde{D}$ is the diagonal matrix of eigenvalues of $(\tilde{A},\tilde{B})$ (Lemma \ref{lem:classical_problem}). Then, by taking the norm, we obtain
\begin{equation*}
    \|\mathbf{x}\|_2 \leq \|(\tilde{A}-\sigma \tilde{B})^{-1}\|_2\|(\tilde{A}-\sigma \tilde{B})\mathbf{x}\|_2 = \|\tilde{U}(\tilde{D}-\sigma I)^{-1}\tilde{U}^T\|_2\|(\tilde{A}-\sigma \tilde{B})\mathbf{x}\|_2 \leq \|\tilde{U}\|_2^2 \frac{\|(\tilde{A}-\sigma \tilde{B})\mathbf{x}\|_2}{\min_j |\tilde{\lambda}_j-\sigma|}.
\end{equation*}
Noticing that $\|\tilde{U}\|_2^2=\|\tilde{B}^{-1}\|_2=\frac{1}{\lambda_1(\tilde{B})}$, we obtain
\begin{equation*}
    \min_j |\tilde{\lambda}_j-\sigma| \leq \frac{1}{\lambda_1(\tilde{B})}\frac{\|(\tilde{A}-\sigma \tilde{B})\mathbf{x}\|_2}{\|\mathbf{x}\|_2}.
\end{equation*}
In particular, the inequality trivially holds for any $\sigma \in \Lambda(\tilde{A},\tilde{B})$. After choosing $\sigma=\lambda_i$ and $\mathbf{x}=\mathbf{u}_i$, we obtain a residual bound. Since $A\mathbf{u}_i=\lambda_i B\mathbf{u}_i$, we finally obtain
\begin{equation*}
    \min_j |\tilde{\lambda}_j-\lambda_i| \leq \frac{1}{\lambda_1(\tilde{B})}\frac{\|(\tilde{A}-\lambda_i \tilde{B})\mathbf{u}_i\|_2}{\|\mathbf{u}_i\|_2} = \frac{1}{\lambda_1(\tilde{B})}\frac{\|(\tilde{A}-A+\lambda_i(B-\tilde{B}))\mathbf{u}_i\|_2}{\|\mathbf{u}_i\|_2} \leq \frac{1}{\lambda_1(\tilde{B})}\left(\frac{\|E\mathbf{u}_i\|_2}{\|\mathbf{u}_i\|_2} +|\lambda_i| \frac{\|F\mathbf{u}_i\|_2}{\|\mathbf{u}_i\|_2}\right).
\end{equation*}
Statement \eqref{eq: Bauer_Fike_2} follows by swapping the roles of $(A,B)$ and $(\tilde{A},\tilde{B})$.
\end{proof} 
\begin{remark}
Theorem \ref{th:Bauer_Fike}, statement \eqref{eq: Bauer_Fike_1} is a generalization of a uniform bound on eigenvalues \citep[][Theorem 2]{crawford1976stable}, which appeared earlier in \cite{crawford1970numerical}.
\begin{equation}
    |\tilde{\lambda}_i-\lambda_i| \leq \frac{1}{\lambda_1(\tilde{B})}\left(\|E\|_2 +|\lambda_i|\|F\|_2\right) \qquad 1 \leq i \leq n. \label{eq: Crawford_1976}
\end{equation}
We also note that the pre-multiplicative factors $\frac{1}{\lambda_1(\tilde{B})}$ and $\frac{1}{\lambda_1(B)}$ can be removed by considering the error in the norms induced by $\tilde{B}^{-1}$ and $B^{-1}$, respectively \cite[][Theorem 15.9.1]{parlett1998symmetric}.
\end{remark}
A direct comparison between \eqref{eq: Bauer_Fike_1} and \eqref{eq: Crawford_1976} is generally not possible since they do not bound the same quantities. Nevertheless, if the eigenvalue gap is large enough relative to the error, then the perturbed eigenvalue that attains the minimum in \eqref{eq: Bauer_Fike_1} is actually the $i$th one \cite{ipsen2009refined}. In this case, Theorem \ref{th:Bauer_Fike} is a clear improvement. In the context of mass lumping, the error component in the eigenspaces associated to the smallest eigenvalues tends to be quite small. Moreover, for sufficiently fine meshes, the discrete eigenvalues are sufficiently close to the exact eigenvalues, which are well separated for 1D problems. Thus, Theorem \ref{th:Bauer_Fike} might very well directly provide an upper bound on the relative error for the smallest eigenvalues. This assertion is tested in the next example.

\begin{example}
The quality of the upper bounds given in \eqref{eq: Bauer_Fike_1} and \eqref{eq: Crawford_1976} is assessed numerically for a first order isogeometric discretization of the 1D Laplace. Figure \ref{fig: 1D_Laplace_ML_P1_rel_error_p1_n100_upper_comp} indicates that for the first few eigenvalues, $\tilde{\lambda}_i$ is the perturbed eigenvalue closest to $\lambda_i$. However, it is no longer the case for the upper 82 \% of the spectrum. The error grows faster than the relative gap and the pairing is lost. From Figure \ref{fig: 1D_Laplace_ML_P1_rel_error_p1_n100_upper_comp}, the reader might misleadingly think that the upper bound in \eqref{eq: Bauer_Fike_1} is actually \textit{always} an upper bound on the relative error between $\lambda_i$ and $\tilde{\lambda}_i$, regardless of the eigenvalue gap. Unfortunately, other experiments showed the contrary. The results using the upper bound in \eqref{eq: Bauer_Fike_2} were quite similar and are omitted.

\begin{figure}[H]
    \centering
    \includegraphics[scale=0.5]{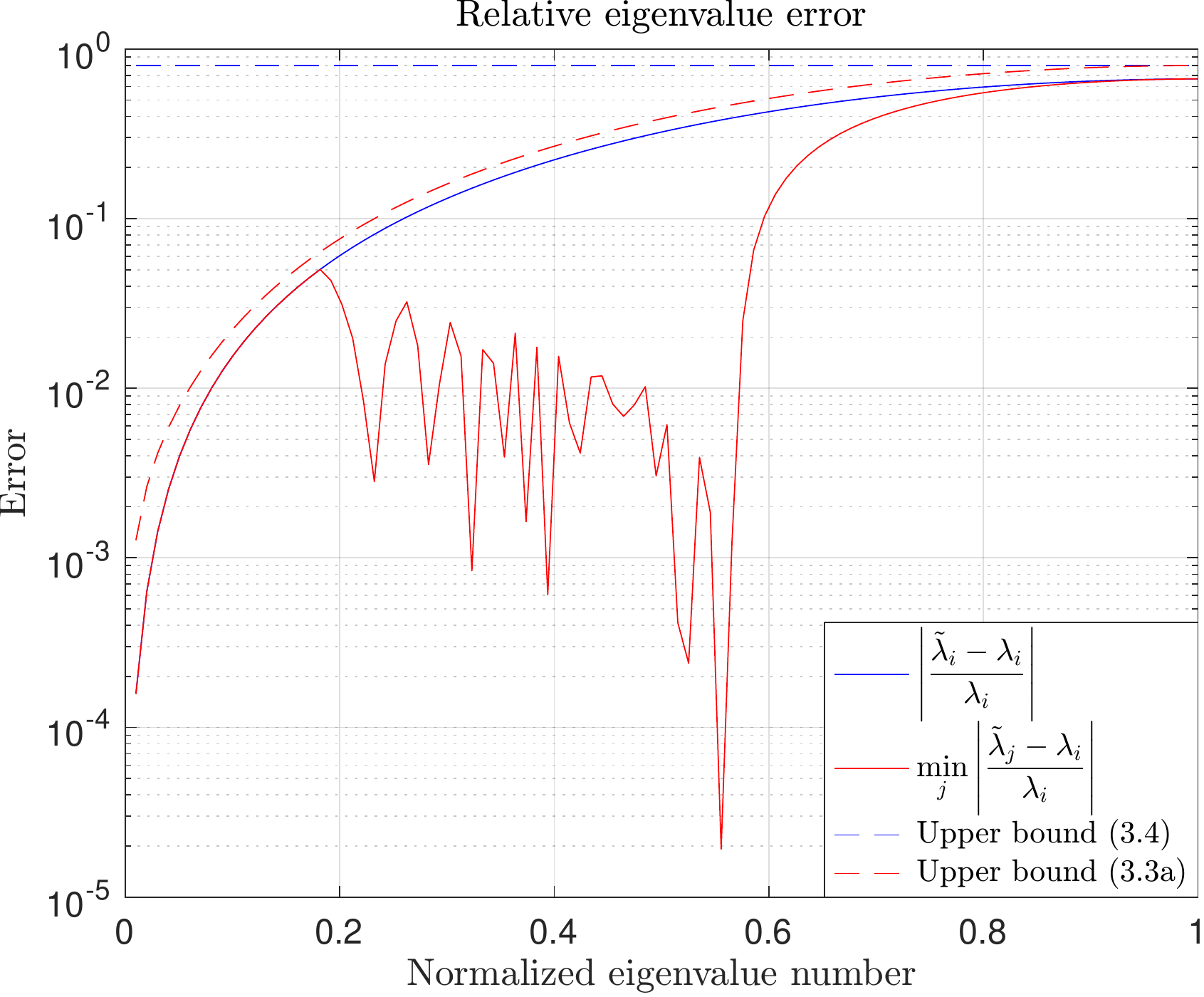}
    \caption{Comparison of \eqref{eq: Bauer_Fike_1} and \eqref{eq: Crawford_1976} for the row-sum lumped mass matrix and $p=1$}
    \label{fig: 1D_Laplace_ML_P1_rel_error_p1_n100_upper_comp}
\end{figure}
\end{example}

\subsection{A generalization of mass lumping}
For applications in explicit dynamics, the different approximation quality at different branches of the spectrum is an attractive feature of mass lumping. A good approximation of the smallest eigenvalues allows us to recover reasonably good solutions whereas a rather poor and gross underestimate of the largest eigenvalues, which usually do not contribute much to the solution for linear problems, improves the stability of the time integration scheme. We will now design an entire class of preconditioners preserving those nice properties and generalizing the simple row-sum lumped mass matrix to more complex structures. A definition and a small preparatory lemma are needed before describing this class.

\begin{definition}[Singular matrix pencil]
\label{def: singular_pencil}
A matrix pencil $(A,B)$ is called singular if $\det(A-\lambda B)=0$ for all $\lambda \in \mathbb{C}$.
\end{definition}

Clearly, matrix pencils formed by matrices $A$ and $B$ whose null space intersects nontrivially are necessarily singular.

\begin{lemma}
\label{lem:negative_semidefiniteness}
Let $B \in \mathbb{R}^{n \times n}$ be a symmetric matrix. Then 
\begin{equation*}
    -\mathcal{L}(B) \preceq B \preceq \mathcal{L}(B).
\end{equation*}
Moreover, 
\begin{enumerate}[noitemsep]
    \item If $B$ does not contain a row of zeros, then $\Lambda(B, \mathcal{L}(B)) \subset [-1,1]$,
    \item If $B$ is nonnegative, then $\dim \ker(B-\mathcal{L}(B)) \geq 1$,
\end{enumerate}
\end{lemma}
\begin{proof}
Since we do not assume $B$ to be positive definite, we run the risk of encountering singular matrix pencils (see Definition \ref{def: singular_pencil}). In general, $\ker(\mathcal{L}(B)) \subseteq \ker(B)$ and $\dim \ker(\mathcal{L}(B))>0$ if and only if $B$ contains a row of zeros. Thus, if $B$ contains a row of zeros, then the null spaces of $B$ and $\mathcal{L}(B)$ have a nontrivial intersection and the matrix pencil $(B, \mathcal{L}(B))$ is singular: generalized eigenvalues associated to an eigenvector $\mathbf{y} \in \ker(B) \cap \ker(\mathcal{L}(B))$ ($\mathbf{y} \neq \mathbf{0}$) are arbitrary. We first study the regular case by assuming $B$ does not contain a row of zeros and later extend the argument to the more general case. If $B$ does not contain a row of zeros, then the lumping operator still delivers a symmetric positive definite matrix $\mathcal{L}(B)$. Thus, $\Lambda(B, \mathcal{L}(B)) \subset \mathbb{R}$. Now, similarly to Lemma \ref{lem:spectrum_row_sum_prec}, we obtain
\begin{equation*}
    |\lambda_k(B, \mathcal{L}(B))| = |\lambda_k(\mathcal{L}(B)^{-1}B)| \leq \|\mathcal{L}(B)^{-1}B\|_{\infty}=1
\end{equation*}
and consequently $\Lambda(B, \mathcal{L}(B)) \subset [-1,1]$. Moreover, since $B$ is symmetric, the Courant-Fischer theorem still holds. Thus, $\forall \mathbf{u} \in \mathbb{R}^n, \mathbf{u} \neq \mathbf{0}$
\begin{equation*}
    -1 \leq \lambda_1(B, \mathcal{L}(B)) = \min_{\substack{ \mathbf{x} \in \mathbb{R}^n \\ \mathbf{x} \neq \mathbf{0}}} \frac{\mathbf{x}^T B\mathbf{x}}{\mathbf{x}^T \mathcal{L}(B) \mathbf{x}} \leq \frac{\mathbf{u}^T B\mathbf{u}}{\mathbf{u}^T \mathcal{L}(B) \mathbf{u}} \leq \max_{\substack{ \mathbf{x} \in \mathbb{R}^n \\ \mathbf{x} \neq \mathbf{0}}} \frac{\mathbf{x}^T B\mathbf{x}}{\mathbf{x}^T \mathcal{L}(B) \mathbf{x}} = \lambda_n(B, \mathcal{L}(B)) \leq 1.
\end{equation*}
Consequently, $-\mathcal{L}(B) \preceq B \preceq \mathcal{L}(B)$. In particular, if $B$ is nonnegative, $\dim \ker(B-\mathcal{L}(B)) \geq 1$ since $(B-\mathcal{L}(B))\mathbf{e}=\mathbf{0}$ where $\mathbf{e}$ is the vector of all ones.

We now extend the proof to the more general case. Let us assume that $B$ contains $m$ rows of zeros. Since $B$ is symmetric, if the $i$th row is zero, so is the $i$th column. Without loss of generality, we may assume that the rows and columns of $B$ are ordered such that
\begin{equation*}
B=
\begin{pmatrix}
C & 0 \\
0 & 0
\end{pmatrix}
\qquad
\mathcal{L}(B)=
\begin{pmatrix}
\mathcal{L}(C) & 0 \\
0 & 0
\end{pmatrix}
\end{equation*}
where $C$ is symmetric and does not contain any row of zeros. The previous proof argument can be used to show that $\Lambda(C, \mathcal{L}(C)) \subset [-1, 1]$ and $-\mathcal{L}(C) \preceq C \preceq \mathcal{L}(C)$. Consequently, $-\mathcal{L}(B) \preceq B \preceq \mathcal{L}(B)$ still holds and the zero rows only increase the dimension of $\ker(B-\mathcal{L}(B))$. If $B$ is the zero matrix, the result trivially holds.
\end{proof}

In some sense, if $B$ is nonnegative, Lemma \ref{lem:negative_semidefiniteness} states that the ``upper bound'' $B \preceq \mathcal{L}(B)$ is ``attained'' in the Loewner ordering. Namely, the error $\mathcal{L}(B)-B$ is not only positive semidefinite but also singular. For the sake of eigenvalue approximation, we are especially interested in singular errors. In fact, positive definite errors will necessarily effect the smallest generalized eigenvalues, which we of course want to avoid.

\begin{definition}[Preconditioners $P_i$]
\label{def:preconditioners_Pi}
Let $B \in \mathbb{R}^{n \times n}$ be a symmetric positive definite matrix and consider the matrix splitting $B=D_i+R_i$ where $D_i$ consists of all super and sub-diagonals smaller than $i$ and $R_i$ is the remainder. We define the sequence of preconditioners $P_i=D_i+\mathcal{L}(R_i)$ for $i=1,\dots,n$. In particular, we observe that $P_1=\mathcal{L}(B)$ and $P_n=B$.
\end{definition}

We are now ready to state an important theorem generalizing Lemma \ref{lem:spectrum_row_sum_prec}.

\begin{theorem}
\label{th:family_lumped_mass_prec}
Let $B \in \mathbb{R}^{n \times n}$ be a symmetric positive definite matrix. Then, the sequence of preconditioners $\{P_i\}_{i=1}^n$ constructed from $B$ according to Definition \ref{def:preconditioners_Pi} satisfies the following properties:
\begin{enumerate}[noitemsep]
    \item $\Lambda(B, P_i) \subset (0,1]$ for all $i=1,\dots,n$,
    \item $\lambda_k(B, P_i) \leq \lambda_k(B, P_{i+1})$ for all $k=1,\dots,n$ and any given $i=1,\dots, n-1$,
    \item If $B$ is nonnegative, then $\lambda_n(B, P_i)=1$ for all $i=1,\dots,n$.
\end{enumerate}
\end{theorem}
\begin{proof}
We prove all properties below.
\begin{enumerate}[noitemsep]
    \item We first prove that $P_i$ is symmetric positive definite for all $i=1,\dots,n$. Symmetry is a straightforward consequence of its definition and the symmetry of $B$. Positive definiteness immediately follows from Lemma \ref{lem:negative_semidefiniteness} and the positive definiteness of $B$:
    \begin{equation*}
        B-P_i = R_i-\mathcal{L}(R_i) \preceq 0 \implies P_i \succeq B \succ 0.
    \end{equation*}
    From the equivalent characterizations in Corollary \ref{cor:equivalence_conditions}, $\Lambda(B, P_i) \subset (0,1]$.
    \item For the second property, we note that
    \begin{equation*}
        P_{i+1}-P_i=D_{i+1}+\mathcal{L}(R_{i+1})-(D_i+\mathcal{L}(R_i))=D_{i+1}-D_i - \mathcal{L}(D_{i+1}-D_i)=\Delta D_i - \mathcal{L}(\Delta D_i) \preceq 0
    \end{equation*}
    where we have denoted $\Delta D_i=D_{i+1}-D_i$ and used the definition of the matrices $R_i$ and $R_{i+1}$ to deduce that
    \begin{equation*}
        \mathcal{L}(R_{i+1})-\mathcal{L}(R_i)=-\mathcal{L}(R_i-R_{i+1})=-\mathcal{L}(D_{i+1}-D_{i}).
    \end{equation*}
    We note that $\Delta D_i$ is obviously symmetric, but it may be singular. Nevertheless, Lemma \ref{lem:negative_semidefiniteness} covers for such a situation and states that $\Delta D_i-\mathcal{L}(\Delta D_i) \preceq 0$. Consequently, $P_i \succeq P_{i+1}$. Moreover, since $P_i$ and $P_{i+1}$ are symmetric positive definite, Corollary \ref{cor:equivalence_conditions} concludes that 
    \begin{equation*}
        \lambda_k(B,P_i) \leq \lambda_k(B,P_{i+1}) \qquad 1 \leq k \leq n.
    \end{equation*}
    The same argument shows that more generally, $P_i \succeq P_j$ for all $j \geq i$. Thus, $\lambda_n(P_j, P_i) \leq 1$ for all $j \geq i$.
    \item By the first property, $\lambda_n(B, P_i) \leq 1$. If $B$ is nonnegative, this upper bound is attained. Indeed, if $B$ is nonnegative, all preconditioners $P_i$ for $i=1,\dots,n$ are also nonnegative and all matrices have the same row-sum by construction. Thus, $B\mathbf{e}=(D_i+R_i)\mathbf{e}=(D_i+\mathcal{L}(R_i))\mathbf{e}=P_i\mathbf{e}$ where $\mathbf{e}$ is the vector of all ones. Therefore, $(1,\mathbf{e})$ is an eigenpair of $(B, P_i)$ and $(P_j, P_i)$ for all $i,j=1,\dots,n$. Consequently, $\lambda_n(B, P_i)=1$ and $\lambda_n(P_j, P_i)=1$ for all $j \geq i$.
\end{enumerate}
\end{proof}

Lemma \ref{lem:negative_semidefiniteness} states that $\mathcal{L}(B) \succeq B$. Increasing the bandwidth of the preconditioner allows us to find increasingly better preconditioners that are ``closer'' to $B$ in the Loewner ordering:
\begin{equation*}
    \mathcal{L}(B)=P_1 \succeq P_2 \succeq \dots \succeq P_{n-1} \succeq P_n=B.
\end{equation*}
Theorem \ref{th:family_lumped_mass_prec} then simply states that $\Lambda(B, P_i)$ converges monotonically to 1 for increasing values of $i$. In particular, the theorem can be used to construct a sequence of preconditioners for the mass matrix and the first element of this sequence coincides with the usual row-sum lumped mass matrix. Since $\Lambda(M, P_i)$ converges monotonically to $1$ for increasing values of $i$, $\Lambda(K, P_i)$ is expected to yield an increasingly good approximation of $\Lambda(K,M)$. Indeed, since $P_i \succeq P_{i+1}$, Corollary \ref{cor:equivalence_conditions} yields $\lambda_k(K, P_i) \leq \lambda_k(K, P_{i+1})$. By repeating the same argument for all $i=1,\dots,n-1$, we obtain
\begin{equation*}
    \lambda_k(K, \mathcal{L}(M))=\lambda_k(K, P_1) \leq \lambda_k(K, P_2) \leq \dots \leq \lambda_k(K, P_{n-1}) \leq \lambda_k(K, P_n)=\lambda_k(K, M) \qquad 1 \leq k \leq n.
\end{equation*}
In other words, $\Lambda(K, P_i)$ converges monotonically from below to $\Lambda(K, M)$ for increasing values of $i$. In practice, we are only interested in small values of $i$ such that the preconditioners $P_i$ have narrow bandwidths. Linear systems of size $n$ with banded symmetric positive definite matrices of bandwidth $b \ll n$ can be solved in $O(nb^2)$ floating point operations (flops), which represents a considerable saving over the $O(n^3)$ complexity for general matrices using standard Gaussian elimination \citep{golub2013matrix, quarteroni2010numerical}. The tridiagonal case is particularly appealing since linear systems can be solved in $O(n)$ flops using the Thomas algorithm. In general, choosing the right preconditioner is a compromise between accuracy and computational cost. Moreover, we must emphasize that all preconditioners devised in this work are not meant to be used as preconditioners but as direct replacement for the mass matrix in the Newmark method, as is customary in engineering practice.

\begin{example}[Lumped mass preconditioners - 1D]
\label{ex: preconditioners_Pi_1D}
We consider an isogeometric discretization of the 1D Laplace equation on $\Omega=(0,1)$ with B-splines of order $p=3$ and $p=5$ with $C^{p-1}$ continuity and $400$ subdivisions. For a $p$th order discretization with $C^{p-1}$ continuity, the mass and stiffness matrices are banded matrices of bandwidth $p$. Therefore, they have $2p+1$ bands. We construct the first three lumped mass preconditioners of the sequence ($P_1$, $P_2$ and $P_3$). The preconditioner $P_i$ is a banded matrix of bandwidth $i-1$ and has $2i-1$ bands. Thus, $P_1$, $P_2$ and $P_3$ are diagonal, tridiagonal and pentadiagonal matrices, respectively. None of these preconditioners are able to represent the mass matrices exactly. Already for $p=3$, the mass matrix has $7$ bands. However, the further the bands are from the diagonal, the smaller the magnitude of their entries. The spectrum of $(K, M)$ and $(K, P_i)$ for $i=1,2,3$ are compared in Figures \ref{fig: Laplace_lumped_mass_1D_p3_n400} and \ref{fig: Laplace_lumped_mass_1D_p5_n400} for $p=3$ and $p=5$, respectively. The approximation quality improves tremendously already for small values of $i$. For $p=3$, the preconditioner $P_3$ leads to a nearly perfect approximation. We notice that $\Lambda(K, P_i)$ converges monotonically from below to $\Lambda(K, M)$, in perfect agreement with our theoretical results. Alert readers will also notice the ``spikes'' appearing towards the end of the spectrum. The largest eigenvalues of $(K, M)$ grossly overestimate those of the continuous operator and for this reason were coined \textit{outliers} in \citep{cottrell2006isogeometric} where they were first discussed. Unfortunately, in our case, $\Lambda(K, P_i)$ approximates the entire spectrum of $(K, M)$, including the outliers. Fixing this issue will be the object of future work.

\begin{center}
\begin{minipage}[t]{.48\linewidth}
\vspace{0pt}
\begin{figure}[H]
    \centering
    \includegraphics[width=\textwidth]{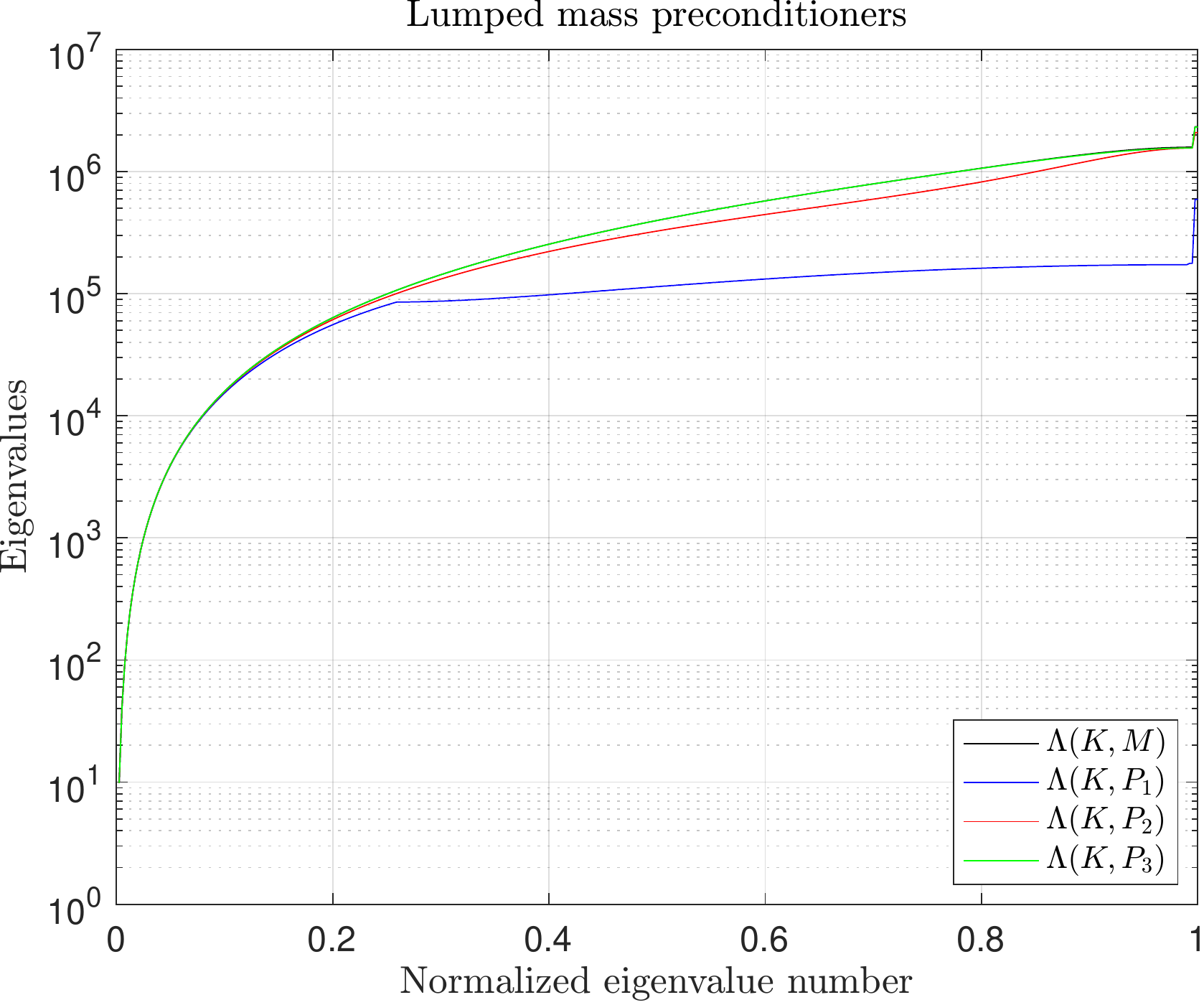}
    \caption{Comparison of $\Lambda(K, M)$ and $\Lambda(K, P_i)$ for $i=1,2,3$ and $p=3$}
    \label{fig: Laplace_lumped_mass_1D_p3_n400}
\end{figure}
\end{minipage}
\hspace{2pt}
\begin{minipage}[t]{.48\linewidth}
\vspace{0pt}
\begin{figure}[H]
    \centering
    \includegraphics[width=\textwidth]{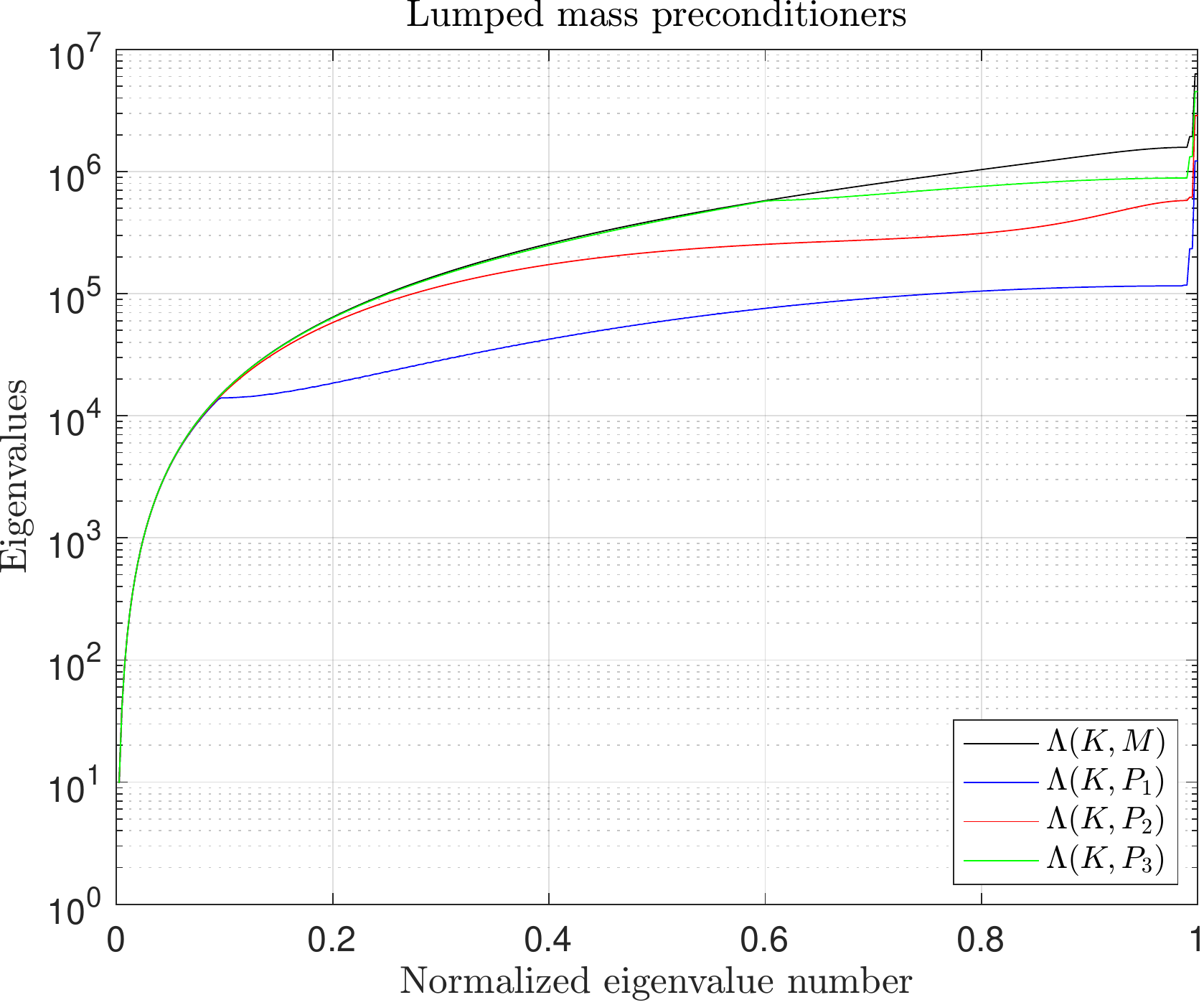}
    \caption{Comparison of $\Lambda(K, M)$ and $\Lambda(K, P_i)$ for $i=1,2,3$ and $p=5$}
    \label{fig: Laplace_lumped_mass_1D_p5_n400}
\end{figure}
\end{minipage}
\end{center}

The performance of these preconditioners in approximating the spectrum of $(K,M)$ is directly related to the clustering on the eigenvalues of $(M, P_i)$ around $1$. The eigenvalues of the matrix pencils $(M, P_i)$ for $i=1,2,3$ are shown in Figures \ref{fig: Laplace_spectrum_prec_1D_p3_n400} and \ref{fig: Laplace_spectrum_prec_1D_p5_n400} for $p=3$ and $p=5$, respectively. As predicted theoretically, the spectrum of the preconditioned matrices converges monotonically to $1$.

\begin{center}
\begin{minipage}[t]{.48\linewidth}
\vspace{0pt}
\begin{figure}[H]
    \centering
    \includegraphics[width=\textwidth]{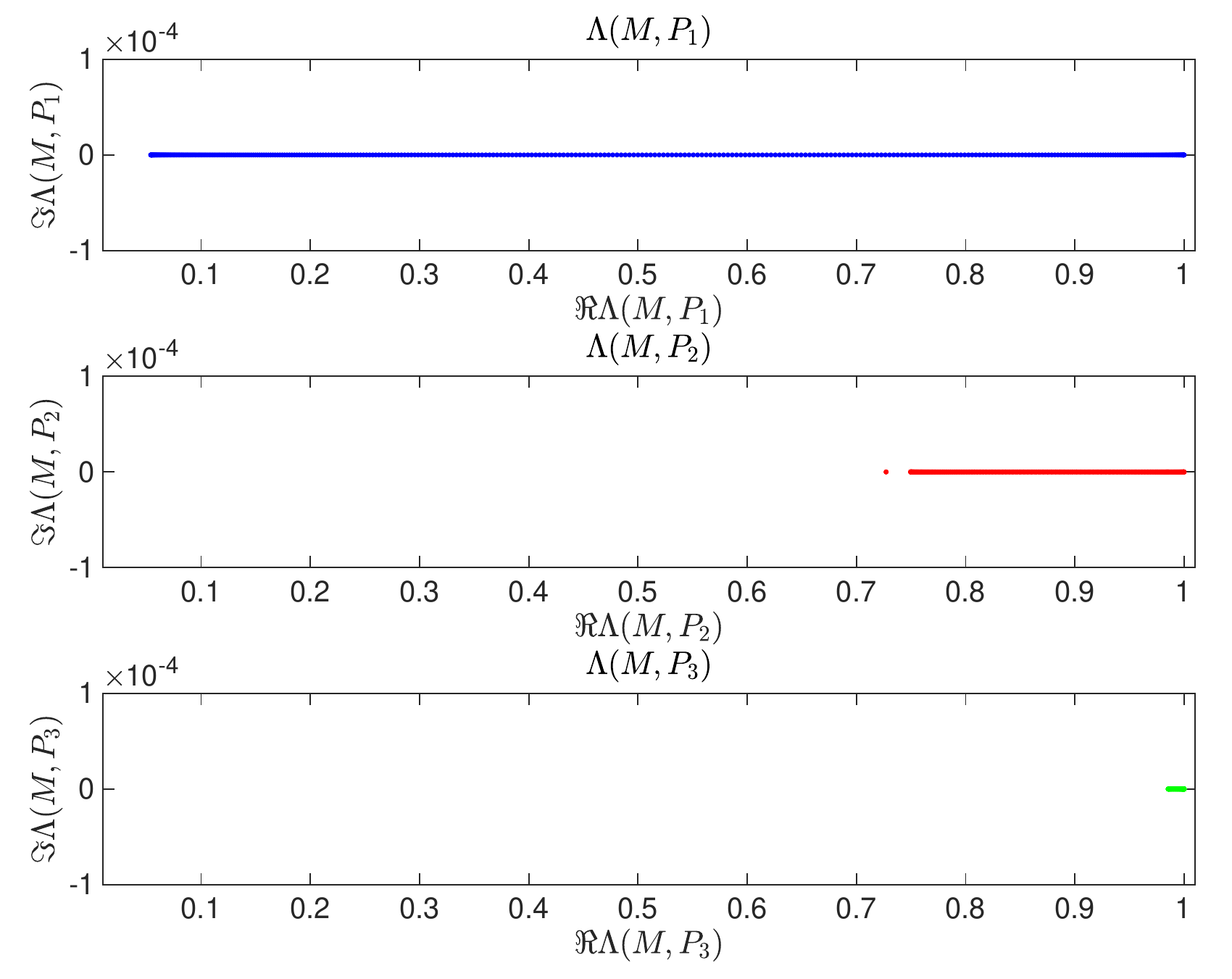}
    \caption{$\Lambda(M, P_i)$ for $i=1,2,3$ and $p=3$}
    \label{fig: Laplace_spectrum_prec_1D_p3_n400}
\end{figure}
\end{minipage}
\hspace{2pt}
\begin{minipage}[t]{.48\linewidth}
\vspace{0pt}
\begin{figure}[H]
    \centering
    \includegraphics[width=\textwidth]{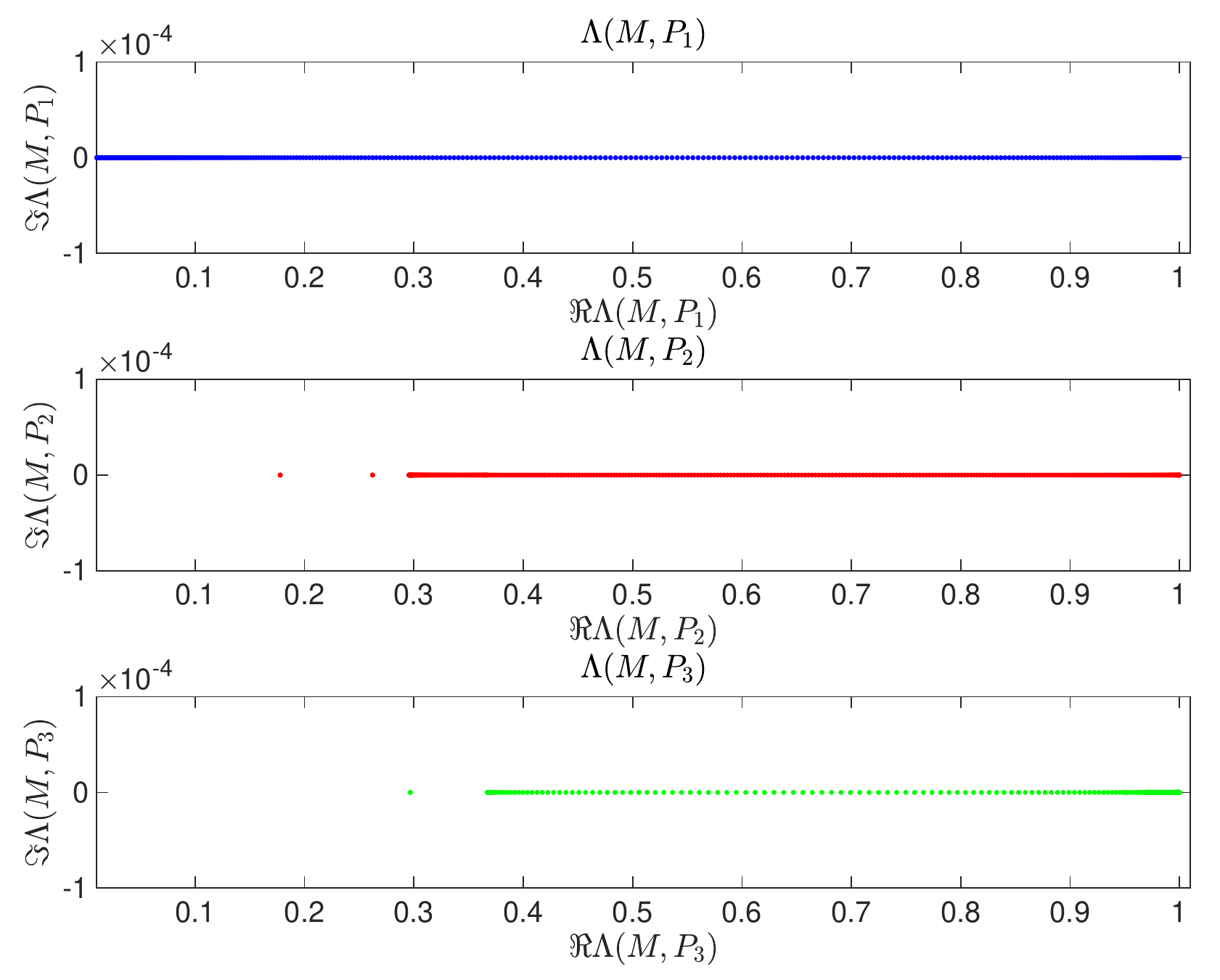}
    \caption{$\Lambda(M, P_i)$ for $i=1,2,3$ and $p=5$}
    \label{fig: Laplace_spectrum_prec_1D_p5_n400}
\end{figure}
\end{minipage}
\end{center}
\end{example}

\begin{remark}
Theorem \ref{th:family_lumped_mass_prec} can also be used to construct a sequence of preconditioners for the stiffness matrix $K$. It can be easily shown that in this case the approximate spectrum converges monotonically from above to $\Lambda(K, M)$ for increasing values of $i$. However, the main computational bottleneck in explicit dynamics consists in solving linear systems with the mass matrix, not the stiffness matrix. Therefore, we will not explore this option any further.
\end{remark}

\begin{example}[Elastodynamics - 1D]
\label{ex: elastodynamics_1D}
We will show in this example that the improved accuracy on the lower part of the spectrum translates into better approximate solutions. We consider a rod clamped at both sides and subjected to some initial displacement profile. This problem is a small variant of \citep[][Section 7.1]{evans2018explicit}. The initial boundary value problem is a particular case of \eqref{eq: wave_equation} and reads 
\begin{align}
 \partial_{tt} u(x,t)-\partial_{xx} u(x,t)&=0 & &\text{ in } \Omega \times (0,T], \label{eq: 1D_elastodynamics} \\
 u(x,t)&=0 & &\text{ on } \partial \Omega \times (0,T], \nonumber\\
 u(x,0)&=\sin(4 \pi x) & &\text{ in } \Omega,  \nonumber\\
 \partial_t u(x,0)&=0 & &\text{ in } \Omega, \nonumber
\end{align}
where $\Omega=(0,1)$ and $T=6$. The exact solution of \eqref{eq: 1D_elastodynamics} is $u(x,t)=\sin(4 \pi x)\cos(4 \pi t)$. Problem \eqref{eq: 1D_elastodynamics} is discretized in space using quartic B-splines with $50$ subdivisions and integrated in time using the central difference method described in Section \ref{se: explicit_dynamics}. The consistent mass matrix yields the most restrictive step size and is computed according to \eqref{eq: delta_t_crit_central_diff} after multiplying it by $0.85$ for improved accuracy. For the sake of comparison, the same step size is chosen for all test cases although it could be increased for the lumped mass matrices $P_i$. The numerical solutions for the consistent mass and lumped mass matrices $P_i$ for $i=1,2,3$ are shown at time $t=1$ and $t=5$ alongside the exact solution in Figures \ref{fig: 1D_Laplace_ML_explicit_disp_time_1_p4_n50} and \ref{fig: 1D_Laplace_ML_explicit_disp_time_5_p4_n50}, respectively. The experiment reveals that slightly expanding the bandwidth might significantly improve the solution, especially for large time. As expected, the larger the bandwidth, the better the solution.

\begin{center}
\begin{minipage}[t]{.48\linewidth}
\vspace{0pt}
\begin{figure}[H]
    \centering
    \includegraphics[width=\textwidth]{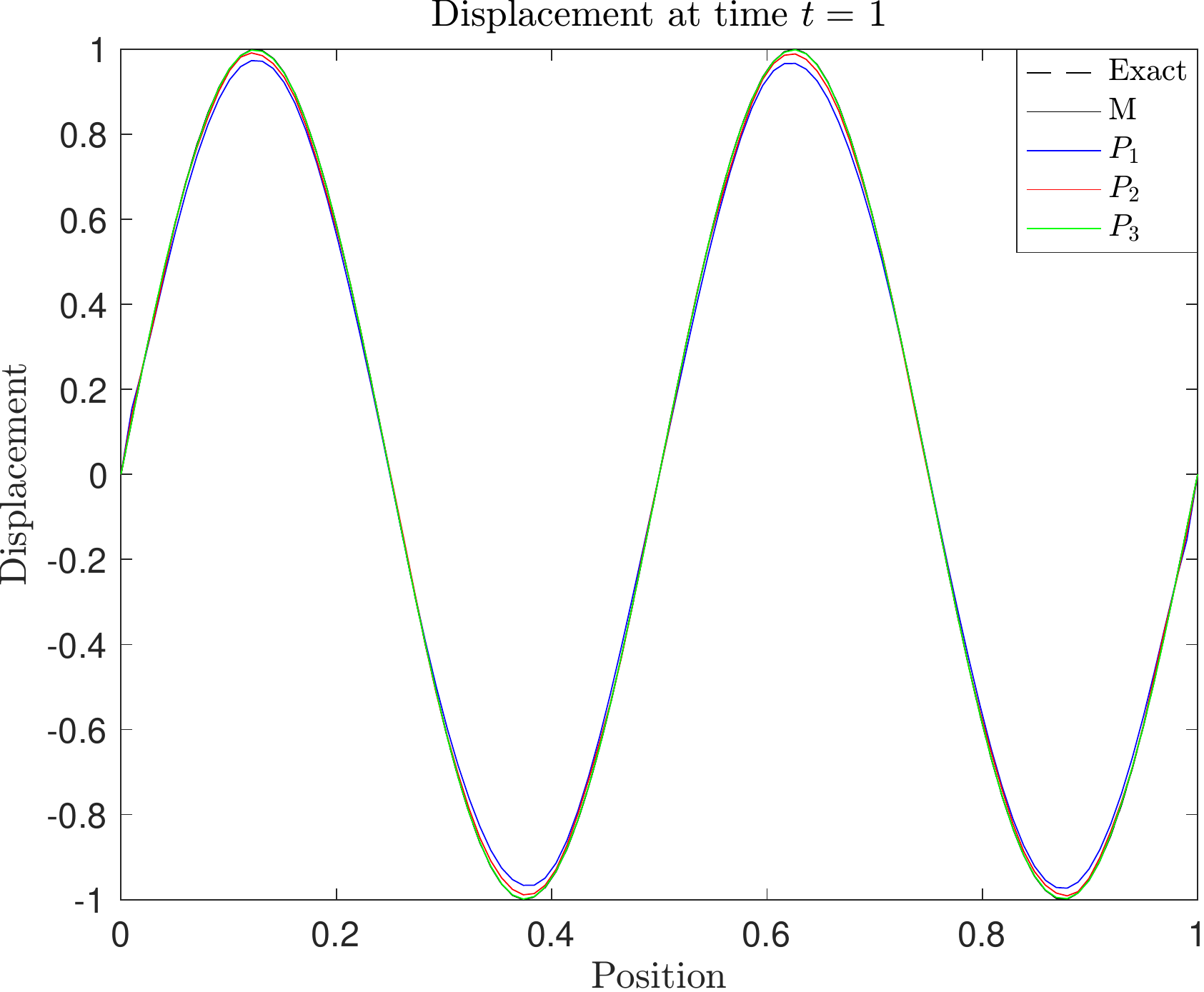}
    \caption{Displacement along the rod at time $t=1$}
    \label{fig: 1D_Laplace_ML_explicit_disp_time_1_p4_n50}
\end{figure}
\end{minipage}
\hspace{2pt}
\begin{minipage}[t]{.48\linewidth}
\vspace{0pt}
\begin{figure}[H]
    \centering
    \includegraphics[width=\textwidth]{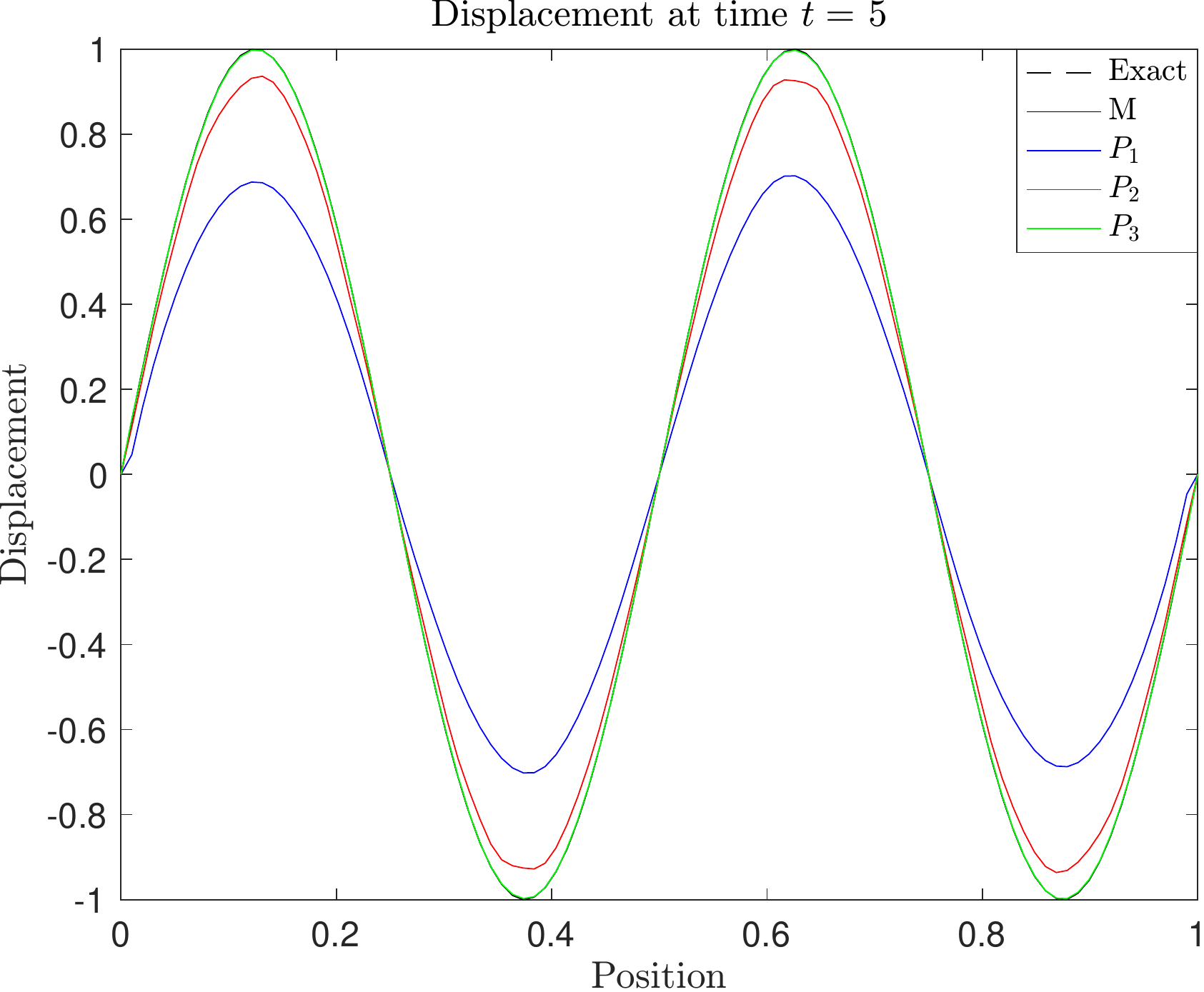}
    \caption{Displacement along the rod at time $t=5$}
    \label{fig: 1D_Laplace_ML_explicit_disp_time_5_p4_n50}
\end{figure}
\end{minipage}
\end{center}
\end{example}

Theorem \ref{th:family_lumped_mass_prec} is based on purely algebraic considerations on the entries of the matrices. Thus, it holds regardless of how complicated the physical domain is or how large the dimension of the problem becomes. However, for multidimensional problems the mass matrix no longer has a banded structure. Therefore, if banded preconditioners are constructed, the preconditioned spectrum will most likely converge very slowly to $1$. This foreboding was confirmed in numerical experiments. Indeed, it is well-known in preconditioning practice that good preconditioners must preserve the essential structure of the matrices \citep{wathen2015preconditioning, pearson2020preconditioners}. Therefore, it would be wise to further extend the family of lumped mass preconditioners to better account for the structure. For tensor product basis functions, there exists a natural generalization to multidimensional problems.

\subsection{Multidimensional problems}
\label{subse: multidimensional_pb}
We first discuss the case of a scalar PDE in 2D. The extension to 3D is straightforward and vector problems will be discussed at the end of this section. If the physical domain is a square, the density function is separable and tensor product basis functions are used, the mass matrix is expressed as $M=M_1 \otimes M_2$ where $M_1$ and $M_2$ are mass matrices for univariate problems (see for instance \cite{gao2014fast}). We then consider the family of lumped mass preconditioners given by $P_{ij}=P_{1,i} \otimes P_{2,j}$ where $P_{1,i}$ and $P_{2,j}$ are banded preconditioners defined in Definition \ref{def:preconditioners_Pi} for some indices $1 \leq i,j \leq n$. This definition is the natural extension to two-dimensional problems, as shown in the next theorem. 

\begin{definition}[Preconditioners $P_{ij}$]
\label{def:preconditioners_Pij}
Let $B=B_1 \otimes B_2$ where $B_1,B_2 \in \mathbb{R}^{n \times n}$ are symmetric positive definite matrices. We define the family of preconditioners $P_{ij}=P_{1,i} \otimes P_{2,j}$, where $P_{1,i}$ and $P_{2,j}$ are preconditioners constructed from $B_1$ and $B_2$, respectively, for some indices $1 \leq i,j \leq n$ according to Definition \ref{def:preconditioners_Pi}. In particular, we observe that $P_{nn}=B$.
\end{definition}

\begin{theorem}
\label{th:family_lumped_mass_prec_ext}
Let $B=B_1 \otimes B_2$ where $B_1,B_2 \in \mathbb{R}^{n \times n}$ are symmetric positive definite matrices. Then, the family of preconditioners $\{P_{ij}\}_{i,j=1}^n$ constructed from $B$ according to Definition \ref{def:preconditioners_Pij} satisfies the following properties:
\begin{enumerate}[noitemsep]
    \item $\Lambda(B, P_{ij}) \subset (0,1]$ for all $i,j=1,\dots,n$,
    \item $\lambda_k(B, P_{sq}) \leq \lambda_k(B, P_{ij})$ for all $k=1,\dots,n^2$ if $i \geq s$ and $j \geq q$,
    \item If $B_1$ and $B_2$ are nonnegative, then $\lambda_n(B, P_{ij})=1$ for all $i,j=1,\dots,n$.
\end{enumerate}
\end{theorem}
\begin{proof}
We prove all properties below using the results of Theorem \ref{th:family_lumped_mass_prec}. Let us first define the product of two sets $\mathcal{S}_1, \mathcal{S}_2 \subseteq \mathbb{C}$ as $\mathcal{S}_1\mathcal{S}_2=\{s_1s_2 \colon s_1 \in \mathcal{S}_1, s_2 \in \mathcal{S}_2\}$.
\begin{enumerate}[noitemsep]
    \item We will first prove that $B$ and $P_{ij}$ are symmetric positive definite. Our arguments are based on some basic properties of the Kronecker product. The reader may for instance refer to \citep[][Chapter 4]{horn1991topics} for their proof. For matrices $A,C \in \mathbb{R}^{n \times n}$
    \begin{itemize}
        \item $\Lambda(A \otimes C)=\Lambda(A)\Lambda(C)$,
        \item $(A \otimes C)^T=A^T \otimes C^T$,
        \item If $A$ and $C$ are invertible, $(A \otimes C)^{-1}=A^{-1} \otimes C^{-1}$.
    \end{itemize}
     By the first and second property, since $B_1$, $B_2$ and $P_{1,i}, P_{2,j}$ are symmetric positive definite for all $i,j=1,\dots,n$, then so are $B_1 \otimes B_2$ and $P_{1,i} \otimes P_{2,j}$. Therefore, $\lambda_k(B, P_{ij})> 0$ for all $k=1,\dots,n^2$. Then, using Lemma \ref{lem:equivalent_pencils} and the properties of the Kronecker product
    \begin{equation*}
        \Lambda(B, P_{ij})=\Lambda(B_1 \otimes B_2, P_{1,i} \otimes P_{2,j})=\Lambda(P_{1,i}^{-1}B_1 \otimes P_{2,j}^{-1}B_2)=\Lambda(B_1, P_{1,i})\Lambda(B_2, P_{2,j}).
    \end{equation*}
    Thus, using the first property of Theorem \ref{th:family_lumped_mass_prec}, $\lambda_n(B, P_{ij})=\lambda_n(B_1, P_{1,i})\lambda_n(B_2, P_{2,j}) \leq 1$ and $\Lambda(B, P_{ij}) \subset (0,1]$.
    \item Using Theorem \ref{th:eig_bounds}, inequality \eqref{eq: ineq1}, we obtain
    \begin{equation*}
        \lambda_k(B,P_{sq}) \leq \lambda_k(B,P_{ij})\lambda_n(P_{ij},P_{sq}) \qquad 1 \leq k \leq n^2.
    \end{equation*}
    Using again Lemma \ref{lem:equivalent_pencils} and the properties of the Kronecker product, $\lambda_n(P_{ij},P_{sq})=\lambda_n(P_i, P_s)\lambda_n(P_j,P_q)$. Finally, when proving the third property of Theorem \ref{th:family_lumped_mass_prec}, we showed that $\lambda_n(P_j, P_i) \leq 1$ for all $j \geq i$. Therefore, $\lambda_n(P_{ij},P_{sq}) \leq 1$ if $i \geq s$ and $j \geq q$ and the result follows.
    \item Moreover, if $B_1$ and $B_2$ are nonnegative, so are the preconditioners $P_{1,i}$ and $P_{2,j}$ for all $i,j=1,\dots,n$ constructed from $B_1$ and $B_2$ and Theorem \ref{th:family_lumped_mass_prec} states that $\lambda_n(B_1, P_{1,i})=1$ and $\lambda_n(B_2, P_{2,j})=1$ for all $i,j=1,\dots,n$. Therefore, $\lambda_n(B, P_{ij})=\lambda_n(B_1, P_{1,i})\lambda_n(B_2, P_{2,j})=1$.
\end{enumerate}
\end{proof}

\begin{remark}
The previous theorem is extended straightforwardly to matrices $B=B_1 \otimes B_2 \otimes B_3$ and the associated family of preconditioners $P_{ijk}=P_{1,i} \otimes P_{2,j} \otimes P_{3,k}$ where $P_{1,i}$, $P_{2,j}$ and $P_{3,k}$ are constructed from $B_1$, $B_2$ and $B_3$, respectively, for some indices $1 \leq i,j,k \leq n$ according to Definition \ref{def:preconditioners_Pi}. 
\end{remark}

Linear systems with a Kronecker product structure can be solved very efficiently and especially if the factor matrices themselves have good properties. It is the main motivation for designing Kronecker product preconditioners \citep{gao2014fast, loli2021easy, sangalli2016isogeometric}. In the 2D case, if all factor matrices have size $n$ and $c_1$ and $c_2$ denote the cost of solving a single linear system with $P_{1,i}$ and $P_{2,j}$, respectively, then the cost of solving a linear system with $P_{ij}$ is $O(n(c_1+c_2))$. Thus, the preconditioner of largest bandwidth dictates the overall cost. If $c_1 \approx c_2 \approx n$, this cost amounts to $O(n^2)$, which is linear in the size of the system. The analysis can be easily generalized to the 3D case.

\begin{remark}
Definition \ref{def:preconditioners_Pij} can be generalized to matrices $B=S^T(B_1 \otimes B_2)S$ where $B_1,B_2 \in \mathbb{R}^{n \times n}$ are symmetric positive matrices and $S$ is an invertible matrix. The preconditioners are then naturally defined as $P_{ij}=S^T(P_{1,i} \otimes P_{2,j})S$. However, such preconditioners are only useful if linear systems with $S$ can be solved easily. 
\end{remark}

Theorem \ref{th:family_lumped_mass_prec_ext} significantly increases the possibilities for preconditioning the mass matrix. It is worthwhile noting that different preconditioners may have the same sparsity pattern. 

\begin{example}[Lumped mass preconditioners - 2D]
\label{ex: preconditioners_Pij_2D}
We consider the 2D Laplace model problem on the unit square $\Omega=(0,1)^2$ discretized with maximally smooth B-splines of order $p=3$ and $p=5$ and $20$ subdivisions in each parametric direction. For this simple model problem, the mass matrix is expressed as a Kronecker product $M=M_1 \otimes M_2$, with $M_1=M_2$ because of the same discretization parameters in each direction. We now compare the preconditioners $P_i$ for $i=1,2,3$ constructed following Definition \ref{def:preconditioners_Pi} and ignoring the structure of $M$ with the preconditioners $P_{ii}$ constructed following Definition \ref{def:preconditioners_Pij}. For $i=1$, the two preconditioning strategies coincide. Hence, Figures \ref{fig: 2D_Laplace_unit_square_LM_P2_P22_p3_n20} and \ref{fig: 2D_Laplace_unit_square_LM_P3_P33_p3_n20} show the comparison for $i=2$ and $i=3$, respectively and $p=3$. The same comparison for $p=5$ is shown in Figures \ref{fig: 2D_Laplace_unit_square_LM_P2_P22_p5_n20} and \ref{fig: 2D_Laplace_unit_square_LM_P3_P33_p5_n20}. In both cases, the preconditioners $P_{ii}$ lead to a significant improvement. However, it is not a fair comparison since they also usually have many more nonzero entries. Thus, let us select two different preconditioners with exactly the same sparsity pattern and the same number of nonzeros. For instance, $P_2$ and $P_{12}$ are both tridiagonal. Yet, Figure \ref{fig: 2D_Laplace_unit_square_LM_comp_P2_P12_p3_n20} reveals that the preconditioner $P_{12}$ performs much better. 

\begin{center}
\begin{minipage}[t]{.48\linewidth}
\vspace{0pt}
\begin{figure}[H]
    \centering
    \includegraphics[width=\textwidth]{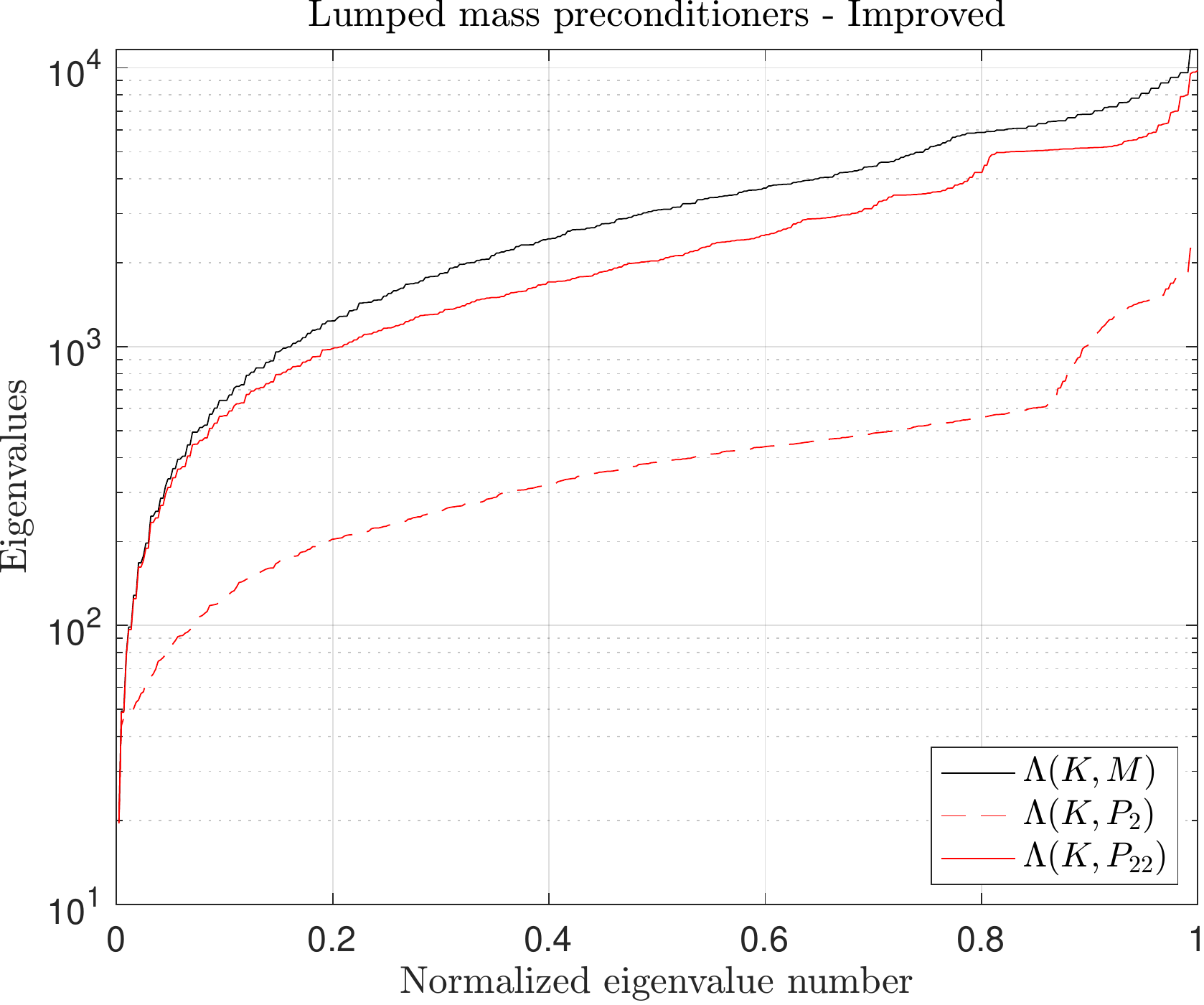}
    \caption{Comparison of $\Lambda(K, P_2)$ and $\Lambda(K, P_{22})$ for $p=3$}
    \label{fig: 2D_Laplace_unit_square_LM_P2_P22_p3_n20}
\end{figure}
\end{minipage}
\hspace{2pt}
\begin{minipage}[t]{.48\linewidth}
\vspace{0pt}
\begin{figure}[H]
    \centering
    \includegraphics[width=\textwidth]{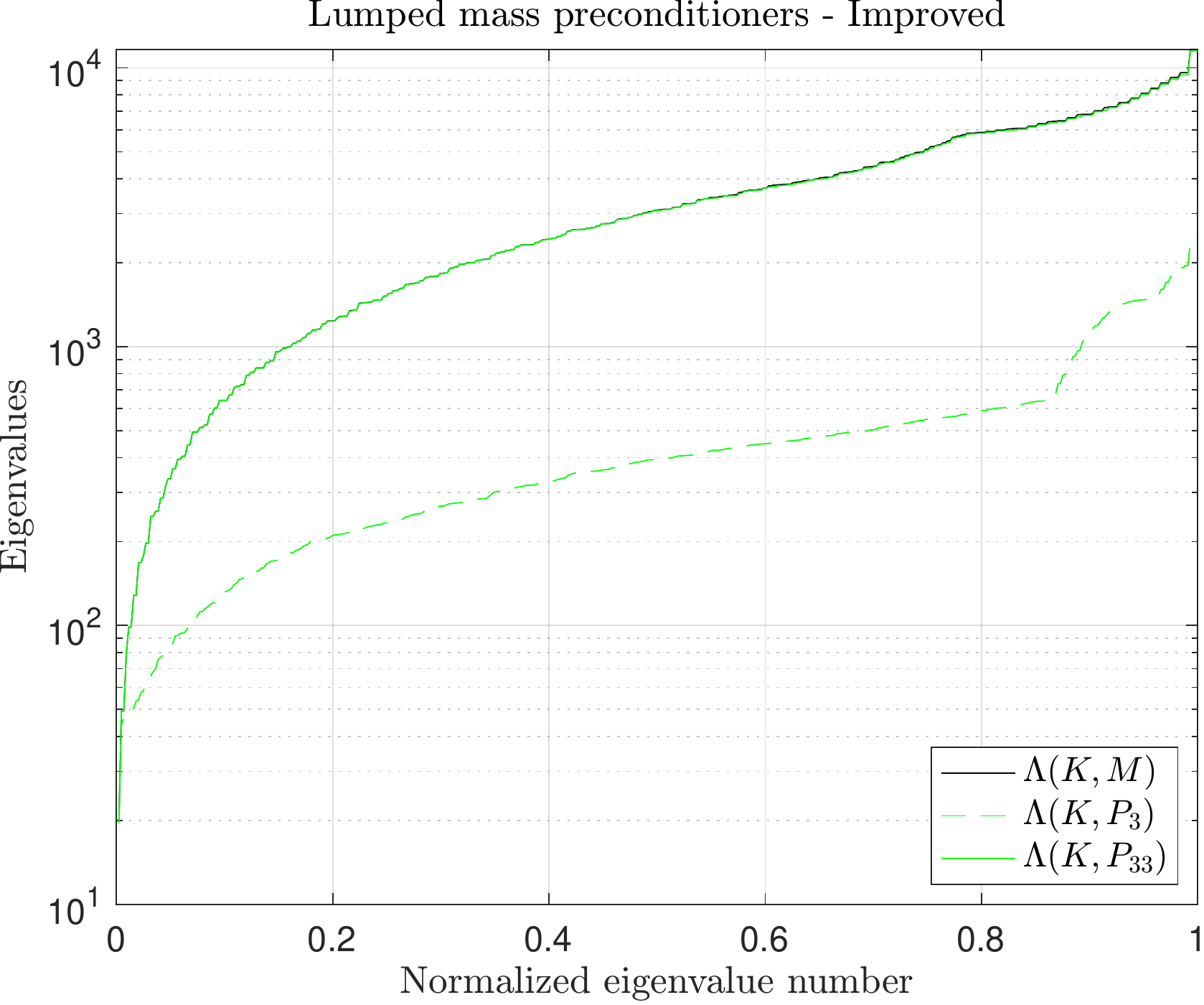}
    \caption{Comparison of $\Lambda(K, P_3)$ and $\Lambda(K, P_{33})$ for $p=3$}
    \label{fig: 2D_Laplace_unit_square_LM_P3_P33_p3_n20}
\end{figure}
\end{minipage}
\end{center}

\begin{center}
\begin{minipage}[t]{.48\linewidth}
\vspace{0pt}
\begin{figure}[H]
    \centering
    \includegraphics[width=\textwidth]{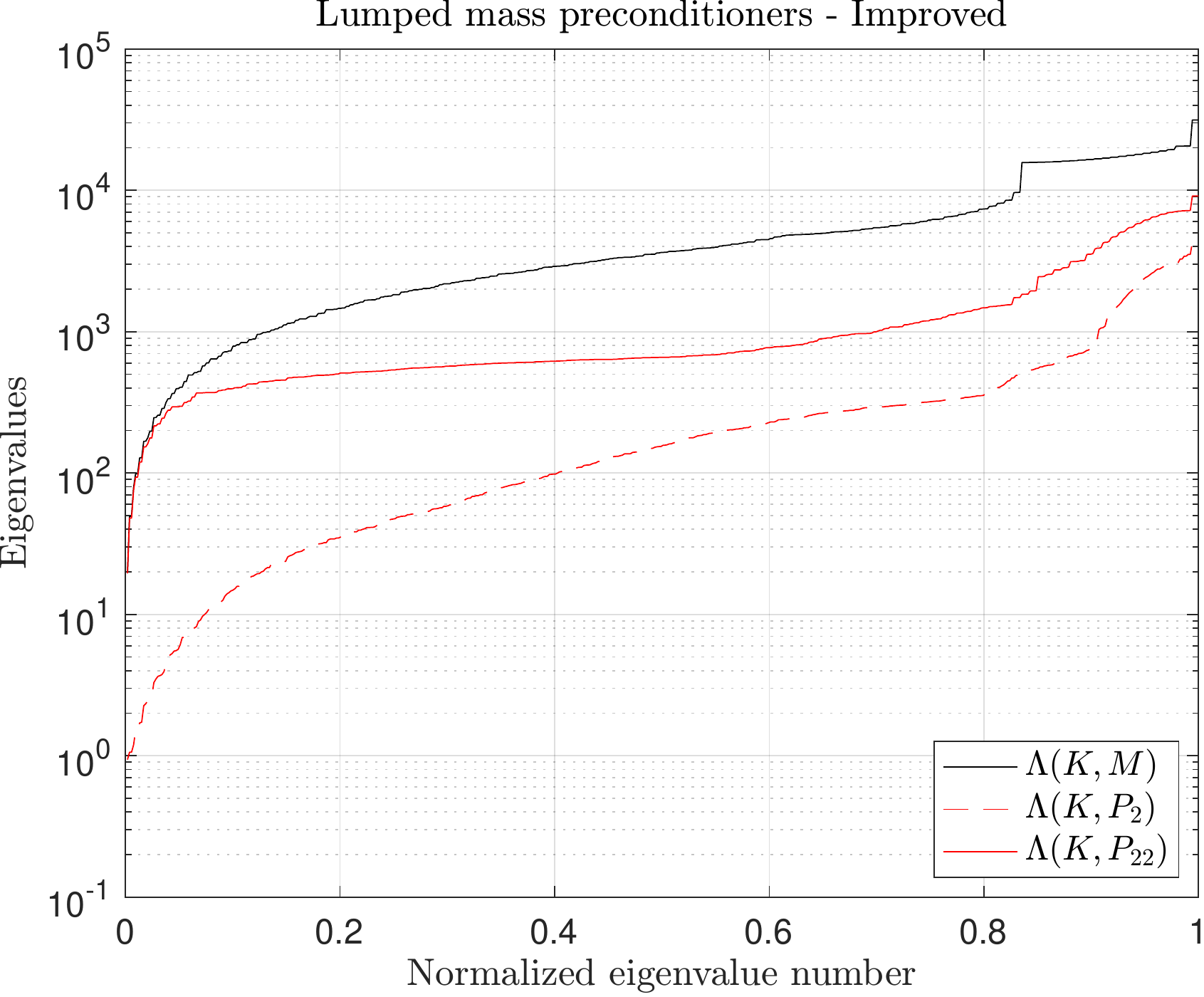}
    \caption{Comparison of $\Lambda(K, P_2)$ and $\Lambda(K, P_{22})$ for $p=5$}
    \label{fig: 2D_Laplace_unit_square_LM_P2_P22_p5_n20}
\end{figure}
\end{minipage}
\hspace{2pt}
\begin{minipage}[t]{.48\linewidth}
\vspace{0pt}
\begin{figure}[H]
    \centering
    \includegraphics[width=\textwidth]{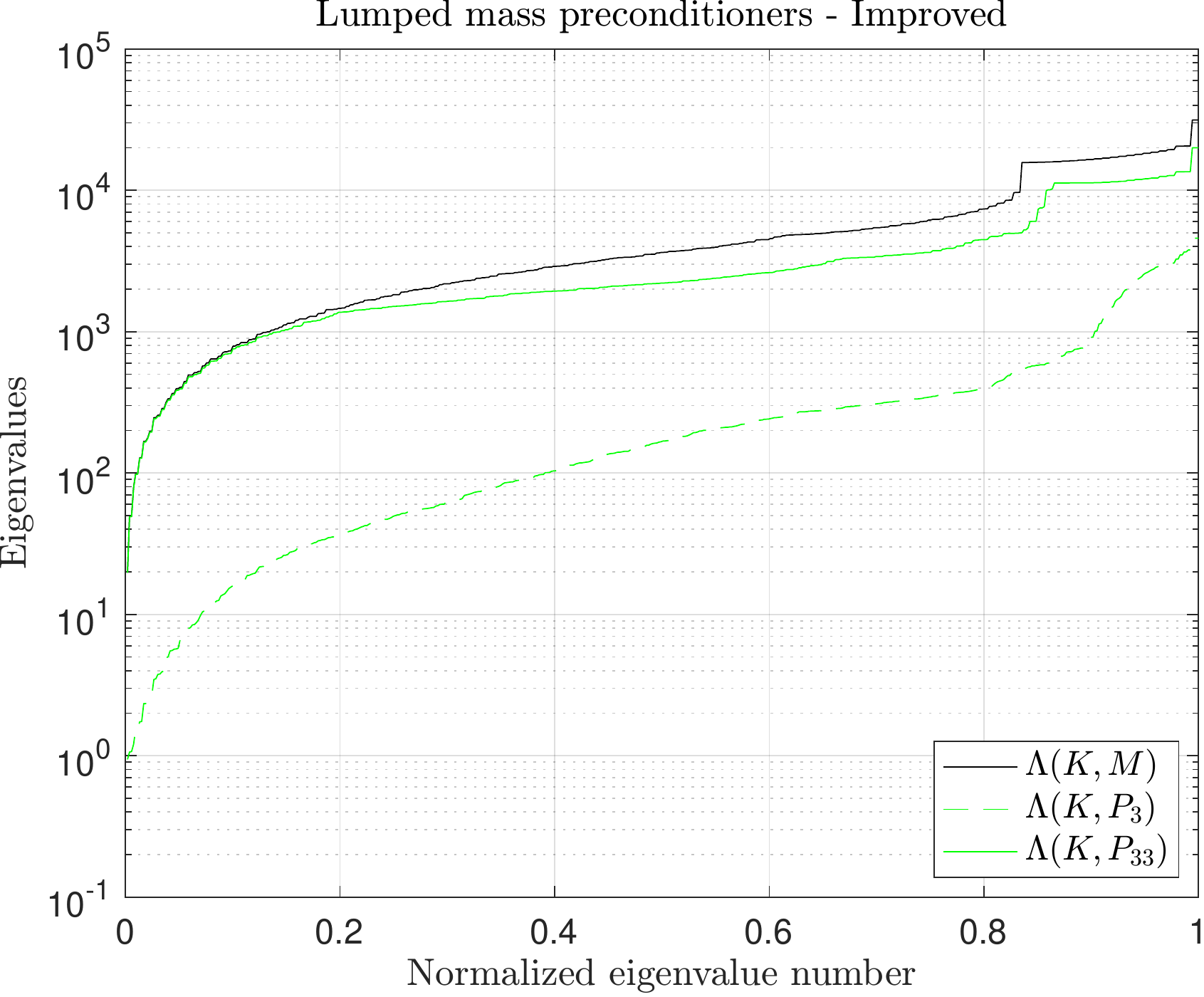}
    \caption{Comparison of $\Lambda(K, P_3)$ and $\Lambda(K, P_{33})$ for $p=5$}
    \label{fig: 2D_Laplace_unit_square_LM_P3_P33_p5_n20}
\end{figure}
\end{minipage}
\end{center}

\begin{figure}[H]
    \centering
    \includegraphics[scale=0.5]{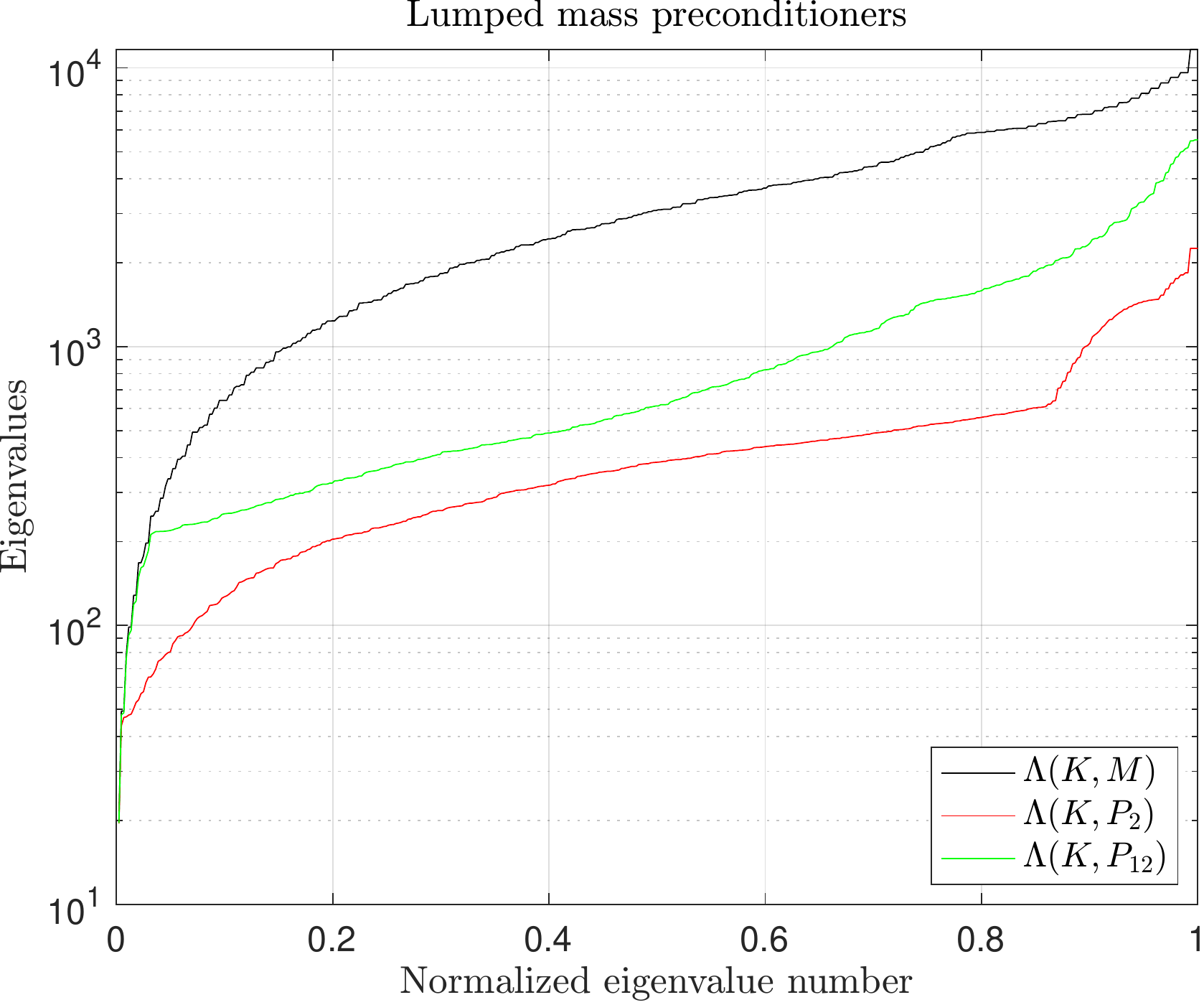}
    \caption{Comparison of $\Lambda(K, M)$, $\Lambda(K, P_2)$ and $\Lambda(K, P_{12})$ for $p=3$}
    \label{fig: 2D_Laplace_unit_square_LM_comp_P2_P12_p3_n20}
\end{figure}
\end{example}

\begin{example}[Elastodynamics - 2D]
As an example for a two-dimensional initial boundary value problem, we consider the wave equation on a quarter annulus with inner radius $1$ and outer radius $2$. This example is again inspired from \citep[][Section 7.3]{evans2018explicit}. The forcing term and initial conditions are chosen such that the problem admits the exact solution $u(x,y,t)=(x^2+y^2-1)(x^2+y^2-2^2)\sin(x)\sin(y)\sin(2\pi t)$. The problem is discretized using cubic B-splines and 20 subdivisions in each parametric direction. Although the domain has curved boundaries, the geometry mapping is a separable function and consequently the mass matrix is exactly expressed as a Kronecker product \cite{gao2014fast}. Similarly to Example \ref{ex: elastodynamics_1D}, the solution of the semi-discrete problem is approximated using the central difference method, resulting in $322$, $139$, $289$ and $320$ time steps for $M$, $P_{11}$, $P_{22}$ and $P_{33}$, respectively. The $L^2$ norm of the error $u_h(x,t)-u(x,t)$ over the time span $[0, 6]$ is shown in Figure \ref{fig: 2D_Laplace_ring_ML_L2_error_p3_n20}. As expected, it increases over time although non-monotonically. A snapshot of the solutions at time $t=0.64$ and $t=2.55$ is displayed in Figures \ref{fig: 2D_Laplace_ring_ML_disp_time_0_64_p3_n20} and \ref{fig: 2D_Laplace_ring_ML_disp_time_2_55_p3_n20}, respectively. At small times, the solutions are visually nearly identical while differences start appearing at larger times.

\begin{figure}[H]
    \centering
    \includegraphics[scale=0.50]{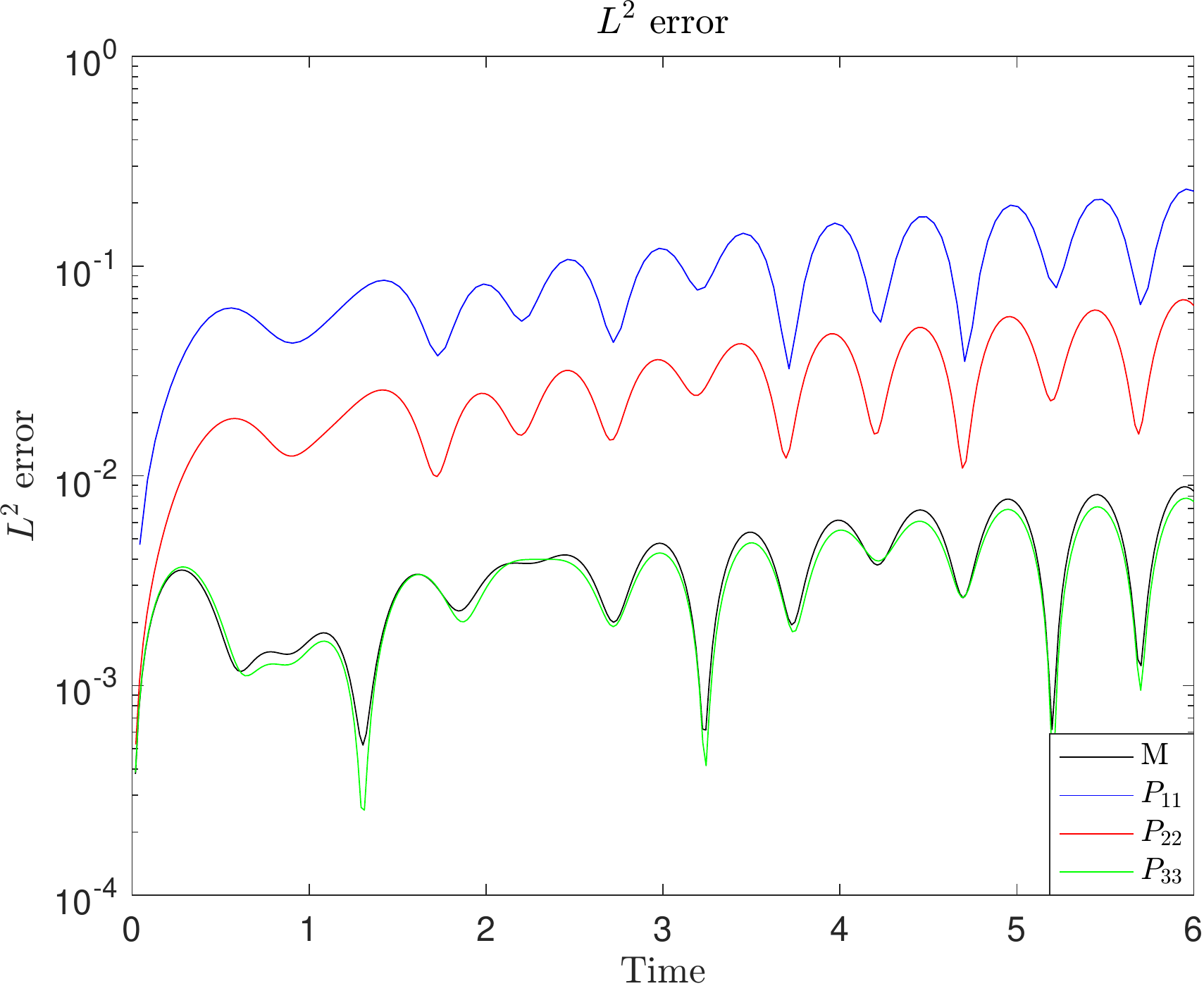}
    \caption{$L^2$ error over the time span $[0, 6]$}
    \label{fig: 2D_Laplace_ring_ML_L2_error_p3_n20}
\end{figure}

\begin{figure}[H]
    \centering
    \includegraphics[scale=0.39]{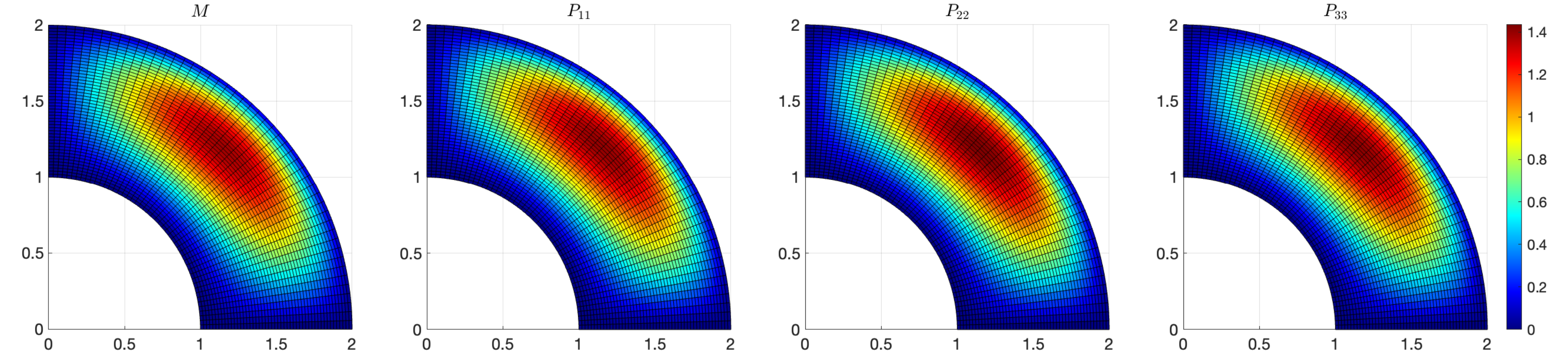}
    \caption{Displacement at time $t=0.64$}
    \label{fig: 2D_Laplace_ring_ML_disp_time_0_64_p3_n20}
\end{figure}

\begin{figure}[H]
    \centering
    \includegraphics[scale=0.39]{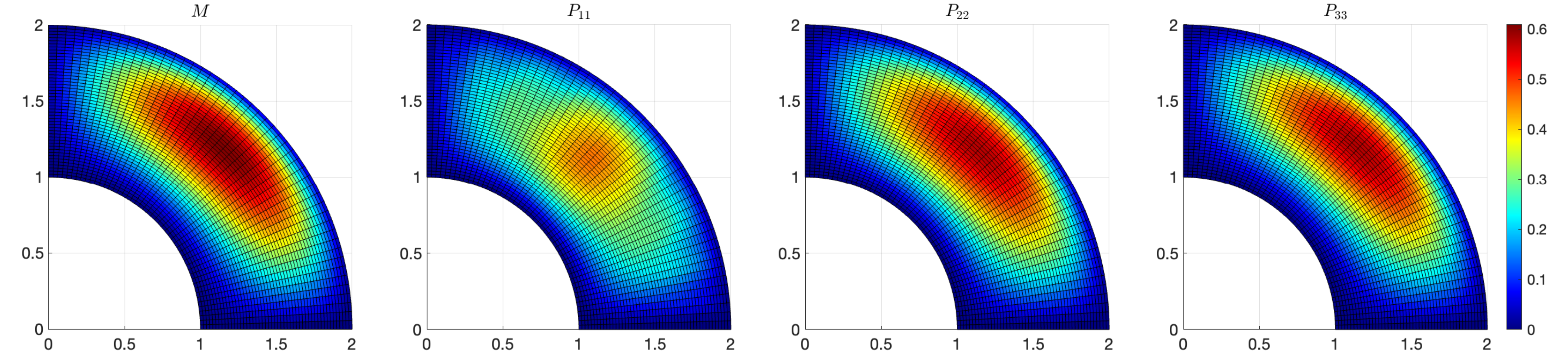}
    \caption{Displacement at time $t=2.55$}
    \label{fig: 2D_Laplace_ring_ML_disp_time_2_55_p3_n20}
\end{figure}
\end{example}

\begin{example}[Accuracy test - 1D and 2D]
\label{ex: accuracy_test}
As we have discussed in Section \ref{se: introduction}, the row-sum technique leads to a second order convergence rate for the smallest eigenvalues (at least for 1D problems and for quadratic and cubic B-spline discretizations \cite{cottrell2006isogeometric}). This property is now investigated numerically for our new class of lumped mass preconditioners for both 1D and 2D model problems. We consider a cubic isogeometric discretization of the 1D and 2D Laplace on the unit line and unit square, respectively, with mixed boundary conditions. The smallest eigenfrequency $\omega_1=\sqrt{\lambda_1}$ for these model problems is given by
\begin{equation*}
    \omega_1=\frac{\pi}{2} \text{ (1D)}, \qquad \omega_1=\frac{\pi}{\sqrt{2}} \text{ (2D)}.
\end{equation*}
We report the relative error $\frac{\omega_1-\omega_{h,1}}{\omega_1}$ in Figures \ref{fig: ML_1D_relative_error_standard_p3} and \ref{fig: ML_2D_relative_error_standard_p3} for our 1D and 2D problems, respectively. Unfortunately, our class of lumped mass preconditioners does not seem to improve the convergence rate but may improve the constant by several orders of magnitude. This fact is interesting given how close $P_3$ (or $P_{33}$) is to the consistent mass matrix $M$. Small differences are enough to ruin the convergence rate, which we expect is due to the consistency error of the method. The same trend was also observed for other types of boundary conditions. 

\begin{center}
\begin{minipage}[t]{.48\linewidth}
\vspace{0pt}
\begin{figure}[H]
    \centering
    \includegraphics[width=\textwidth]{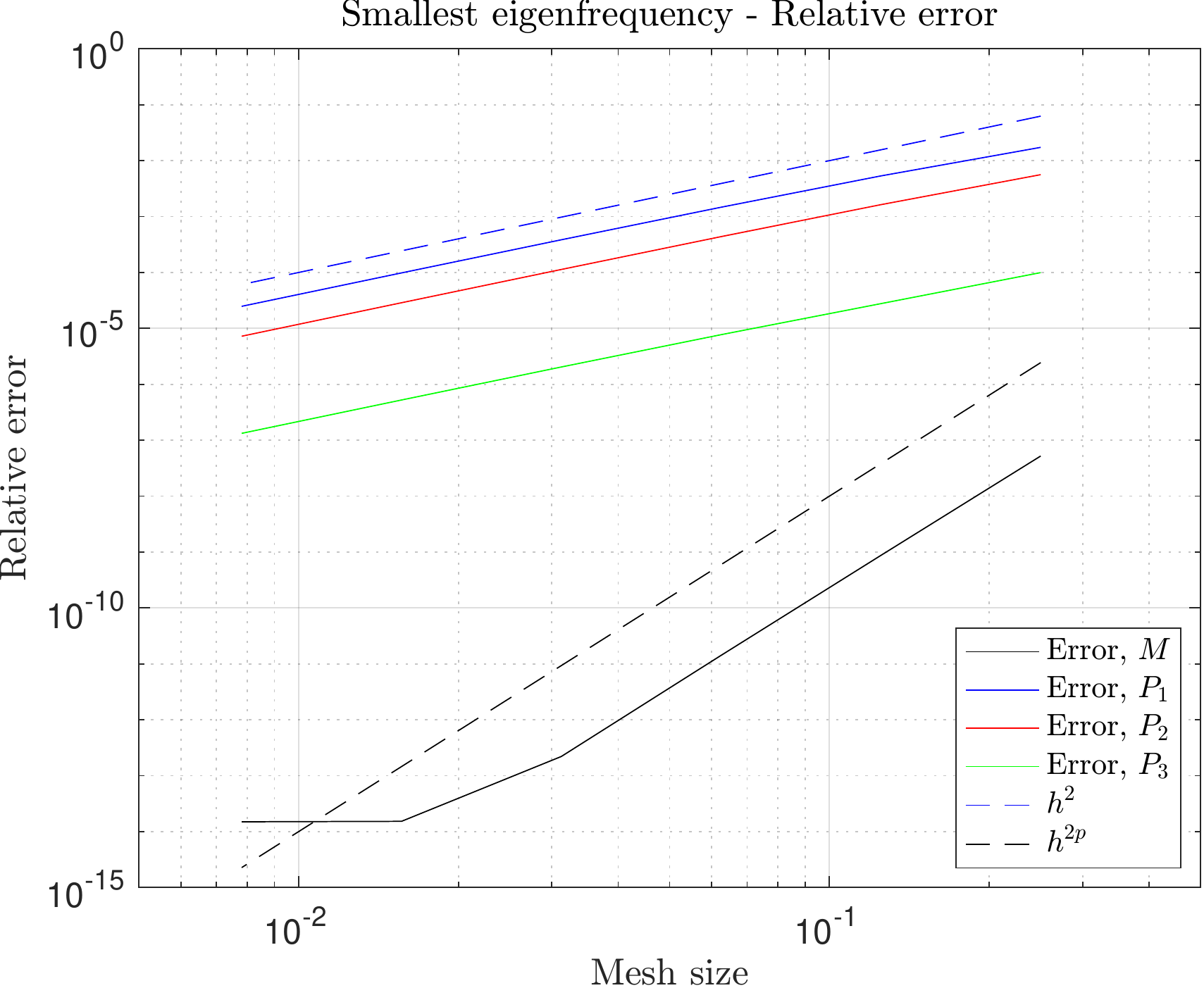}
    \caption{Relative error in the smallest eigenfrequency for a cubic ($p=3$) discretization of the 1D Laplace on the unit line with mixed boundary conditions}
    \label{fig: ML_1D_relative_error_standard_p3}
\end{figure}
\end{minipage}
\hspace{2pt}
\begin{minipage}[t]{.48\linewidth}
\vspace{0pt}
\begin{figure}[H]
    \centering
    \includegraphics[width=\textwidth]{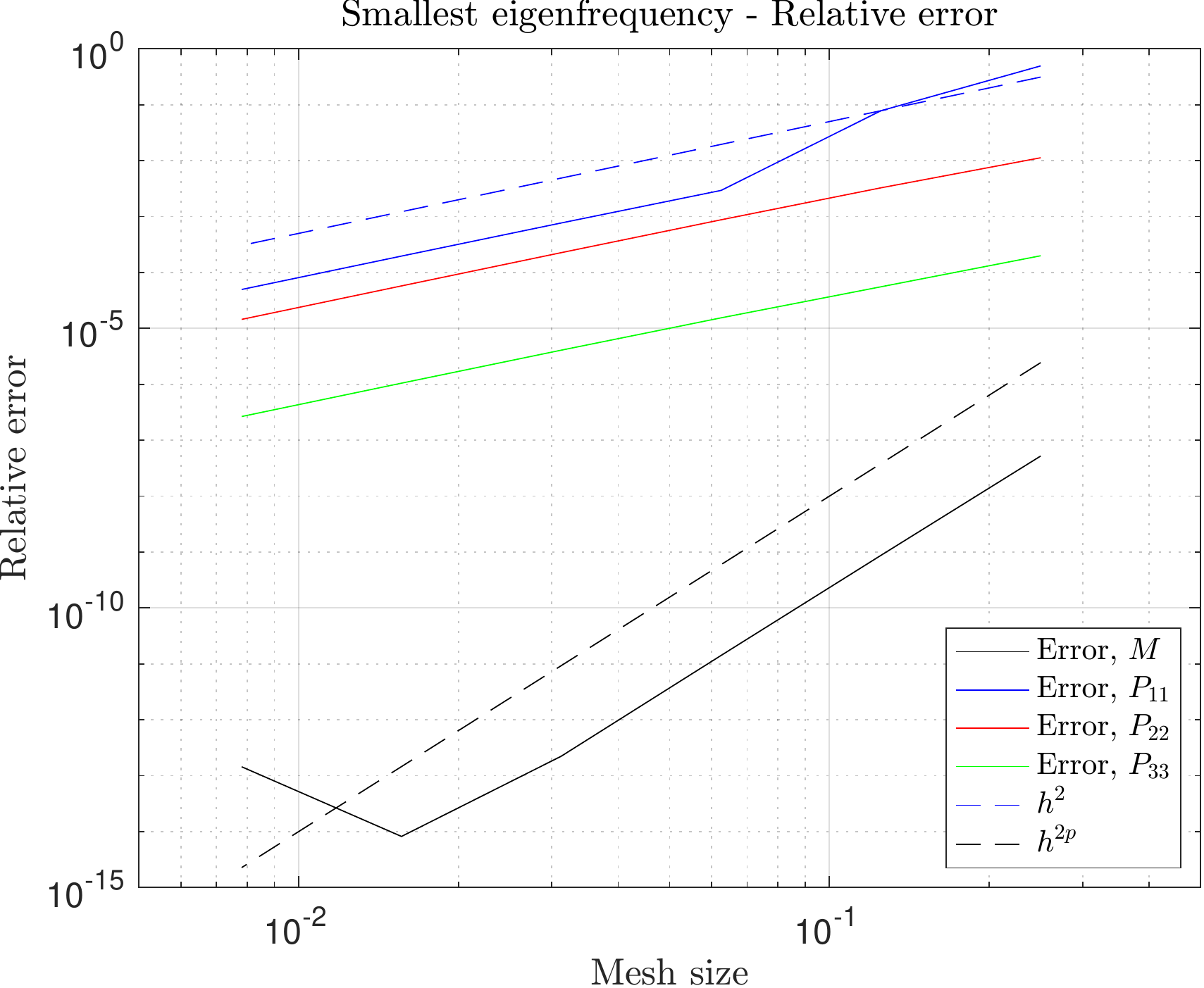}
    \caption{Relative error in the smallest eigenfrequency for a cubic ($p=3$) discretization of the 2D Laplace on the unit square with mixed boundary conditions}
    \label{fig: ML_2D_relative_error_standard_p3}
\end{figure}
\end{minipage}
\end{center}
\end{example}

\begin{remark}[Vector PDEs]
Vector PDEs, such as those modeling the elastic deformation of structures, again require a slight adjustment to the definition of the preconditioner due to the specific structure of the mass matrix. If $M_S$ denotes the mass matrix of a multidimensional scalar PDE, and if the same approximation spaces are used for each component of the approximate solution of the vector PDE, then the mass matrix of a vector problem $M_V$ is either $I_d \otimes M_S$ or $M_S \otimes I_d$ depending on the ordering of the basis functions \cite{voet2020nonlinear}. With the first definition $M_V$ is a block diagonal matrix. This ordering is used in GeoPDEs for instance \cite{vazquez2016new}. In order to preserve the structure, the preconditioner is naturally defined as $P_V=I_d \otimes P_S$ or $P_V=P_S \otimes I_d$. Again, due to the Kronecker product structure, all the properties of $P_S$ carry over to $P_V$. Numerical testing for elasticity problems produced figures similar to those reported in Example \ref{ex: preconditioners_Pij_2D} and are omitted.
\end{remark}

\section{Nearest Kronecker product approximation}
\label{se: NKP}
As we have seen in Section \ref{subse: multidimensional_pb}, the performance of the lumped mass preconditioners can be drastically improved for tensor product basis functions if the mass matrix is expressed as a Kronecker product. Unfortunately, this structure cannot be guaranteed for most problems of practical interest. In several cases, the entries of the mass matrix are computed from
\begin{equation}
    M_{ij}=\int_{\hat{\Omega}} \hat{\rho}(\hat{\mathbf{x}}) \hat{B}_i(\hat{\mathbf{x}}) \hat{B}_j(\hat{\mathbf{x}}) |\det(J_F(\hat{\mathbf{x}}))| \ \mathrm{d}\hat{\Omega} \quad i,j=1,\dots,n^2. \label{eq: mass_mat_entries}
\end{equation}
The integration is done in the parametric domain $\hat{\Omega}$ where $\hat{B}_i(\hat{\mathbf{x}})$ are tensor product basis functions defined in the parametric domain (for instance B-splines in isogeometric analysis), $\hat{\rho}(\hat{\mathbf{x}})=\rho(F(\hat{\mathbf{x}}))$ is the density function, $F \colon \hat{\Omega} \to \Omega$ is the mapping from parametric to physical domain and $J_F(\hat{\mathbf{x}})$ is its Jacobian matrix. The mass matrix looses its Kronecker product structure when the integrand is no longer a separable function. A nontrivial density function, geometry or the use of NURBS basis functions are some of the causes. Nevertheless, for tensor product basis functions, the mass (and stiffness) matrix can be generally very well approximated by a sum of Kronecker products, a fact which has already been exploited for designing fast assemblage algorithms \citep{gao2014fast,mantzaflaris2017low, scholz2018partial, hofreither2018black,chan2018multi}. For 2D problems, we are interested in finding the best factor matrices $B$ and $C$ such that $M \approx \Tilde{M} = B \otimes C$. Several different approaches have been proposed in the literature. The one proposed in \cite{mantzaflaris2017low} and later simplified in \cite{scholz2018partial}, is tailored to PDE problems and attempts to directly approximate all non-separable functions in the integral \eqref{eq: mass_mat_entries} by a sum of separable functions. In other words, it finds a low-rank approximation for the integrand itself. Another approach, used in \cite{hofreither2018black} and based on earlier work from \citep{van1993approximation, pitsianis1997kronecker}, is more general and does not require any knowledge of the origin of the matrix. Of course, the first approach seems better suited to our problem. However, its implementation, although simplified, remains technical. The implementation of the second approach, in comparison, is rather straightforward. The method relies entirely on standard matrix operations. For this reason, we will assess if the easiest method can deliver useful results. We will initially restrict the discussion to 2D problems by closely following the presentation in \cite{van1993approximation}. The extension to 3D problems will be discussed subsequently.

\subsection{Approximation for 2D problems}
\label{se: KP_approx_2D}
Given a matrix $M \in \mathbb{R}^{n_1n_2 \times n_1n_1}$, our goal is to find the best factor matrices $B \in \mathbb{R}^{n_1 \times n_1}$ and $C \in \mathbb{R}^{n_2 \times n_2}$ such that $\phi_M(B,C)=\|M-B \otimes C\|_F$ is minimized. This problem was first investigated by Van Loan and Pitsianis \cite{van1993approximation}. They observed that since both the Kronecker product $B \otimes C$ and $\vectorization(B)\vectorization(C)^T$ form all the products $b_{ij}c_{kl}$ for $i,j=1,\dots,n_1$ and $k,l=1,\dots,n_2$ but at different locations, there must clearly exist a linear mapping $\mathcal{R} \colon \mathbb{R}^{n_1n_2 \times n_1n_2} \to \mathbb{R}^{n_1^2 \times n_2^2}$ such that $\mathcal{R}(B \otimes C)=\vectorization(B)\vectorization(C)^T$. This mapping can be defined explicitly. Consider the following block matrix where $A_{ij} \in \mathbb{R}^{n_2 \times n_2}$ for $i,j=1,\dots,n_1$ and define the mapping as
\begin{equation*}
    A =
    \begin{pmatrix}
    A_{1,1} & \cdots & A_{1, n_1} \\
    \vdots & \ddots & \vdots \\
    A_{n_1,1} & \cdots & A_{n_1, n_1}
    \end{pmatrix}
    \qquad
    \mathcal{R}(A)=
    \begin{pmatrix}
    \vectorization(A_{1,1})^T \\
    \vectorization(A_{2,1})^T \\
    \vdots \\
    \vectorization(A_{n_1, n_1})^T
    \end{pmatrix}.
\end{equation*}
Then, by construction, for a matrix $A=B \otimes C$,
\begin{equation*}
    B \otimes C =
    \begin{pmatrix}
    b_{1,1}C & \cdots & b_{1,n_1}C \\
    \vdots & \ddots & \vdots \\
    b_{n_1,1}C & \cdots & b_{n_1, n_1}C
    \end{pmatrix}
    \qquad \mathcal{R}(B \otimes C)=     
    \begin{pmatrix}
    b_{1,1}\vectorization(C)^T \\
    b_{2,1}\vectorization(C)^T \\
    \vdots \\
    b_{n_1, n_1}\vectorization(C)^T
    \end{pmatrix}
    =\vectorization(B)\vectorization(C)^T.
\end{equation*}
More generally, since the mapping $\mathcal{R}$ inherits the linearity of the vectorization operator, it transforms a sum of $r$ Kronecker products into a rank $r$ matrix. A matrix is said to have Kronecker rank $r$ if it can be expressed as the sum of $r$ Kronecker products. Thus, $A$ has Kronecker rank $r$ if and only if $\mathcal{R}(A)$ has (matrix) rank $r$. The mapping $\mathcal{R}$ was originally called \textit{rearrangement} in \cite{van1993approximation} but later authors have also called it \textit{reordering}. In any case, we must stress that $\mathcal{R}(A)$ is generally not a permutation of the rows and columns of $A$ or any change of basis. 

The minimization of $\|M-B \otimes C\|_F$ in the Frobenius norm allows to rewrite the problem in a more familiar way. The Frobenius norm is defined as the square root of the sum of squares of the matrix entries and does not depend on the location of these entries. Therefore, the mapping $\mathcal{R}$ can be applied to a matrix without affecting its Frobenius norm and the minimization problem becomes
\begin{equation*}
    \min \phi_M(B,C)=\min \|M-B \otimes C\|_F=\min \|\mathcal{R}(M)-\mathcal{R}(B \otimes C)\|_F =\min \|\mathcal{R}(M)-\vectorization(B)\vectorization(C)^T\|_F.
\end{equation*}
Thus, finding the best factor matrices $B$ and $C$ is equivalent to finding the best rank-$1$ approximation of $\mathcal{R}(M)$. The transformation from $M$ to $\mathcal{R}(M)$ requires knowing the block-partitioning of the matrix $M$, which in the context of isogeometric analysis is always clear from the underlying discretization. The next theorem summarizes a very useful result in this context.

\begin{theorem}[{\citep[][Theorems 5.1, 5.3 and 5.8]{van1993approximation}}]
Let $M \in \mathbb{R}^{n_1n_2 \times n_1n_2}$ be a block-banded, nonnegative and symmetric positive definite matrix. Then, there exists banded, nonnegative and symmetric positive definite factor matrices $B$ and $C$ such that $\phi_M(B,C)=\|M-B \otimes C\|_F$ is minimized.
\end{theorem}

Thus, the properties of $M$ are inherited by its approximation $B \otimes C$. Once the Kronecker product approximation has been computed, the lumped mass preconditioners can be constructed from the factor matrices $B$ and $C$ following Definition \ref{def:preconditioners_Pi}. The overall construction consists in a two-level approximation of the mass matrix. Algorithm \ref{algo: Low_rank_approx} is a practical algorithm for computing the preconditioner. For simplicity, we will assume that the preconditioners of the factor matrices have equal bandwidth.

\begin{algorithm}[ht]
\begin{algorithmic}[1]
\caption{Lumped mass and Kronecker product approximation}
\label{algo: Low_rank_approx}
\Statex \textbf{Input}: Mass matrix $M$
\Statex \textbf{Output}: Preconditioner $\tilde{P}_{ii}=\tilde{P}_{1,i} \otimes \tilde{P}_{2,i}$
\State Compute $\mathcal{R}(M)$ 
\State Compute a rank-1 approximation of $\mathcal{R}(M)$ such that $\mathcal{R}(M) \approx \mathbf{b}\mathbf{c}^T$
\State Reshape the vectors $\mathbf{b}$ and $\mathbf{c}$ to matrices $B$ and $C$, respectively, and set $\Tilde{M}=B \otimes C$
\State Construct the preconditioners $\tilde{P}_{1,i}$ and $\tilde{P}_{2,i}$ from $B$ and $C$, respectively, following Definition \ref{def:preconditioners_Pi}
\State Set $\tilde{P}_{ii}=\tilde{P}_{1,i} \otimes \tilde{P}_{2,i}$
\end{algorithmic}
\end{algorithm}

More specifically, a practical implementation in Line 2 would extract a dense submatrix from $\mathcal{R}(M)$ containing all its nonzero entries and then compute a low-rank approximation of that submatrix only. The truncated singular value decomposition (SVD) is an obvious choice for computing a low-rank approximation. However, there exists a host of other black-box algorithms for computing good but not optimal low-rank approximations. Some of them do not require the matrix to be explicitly available. It was the main motivation in \cite{hofreither2018black} for using Adaptive Cross Approximation (ACA). In our context, it means we could potentially compute a good Kronecker product approximation of the mass matrix without ever assembling it. In Line 5, the preconditioner $\tilde{P}_{ii}$ is not necessarily formed explicitly. Only the factors $\tilde{P}_{1,i}$ and $\tilde{P}_{2,i}$ are stored. \\

\textbf{Complexity analysis}: 
The cost of Algorithm \ref{algo: Low_rank_approx} depends on whether it is intended as an addition to the assembly of the mass matrix or as a substitution for it. In the first case, the cost incurred is the sum of the assembly cost and the cost for computing a rank-1 approximation of $\tilde{\mathcal{R}}(M)$, the dense submatrix of $\mathcal{R}(M)$ containing all its nonzero entries, while in the second case some of the existing assembly algorithms can be straightforwardly used to compute a Kronecker rank-1 approximation of the mass matrix \citep{hofreither2018black, mantzaflaris2017low, scholz2018partial}. Both cases are briefly discussed. In the first case, assuming that the mass matrix has already been assembled in $O(a)$, the cost of Algorithm \ref{algo: Low_rank_approx} depends on the algorithm used for computing a rank-1 approximation of $\tilde{\mathcal{R}}(M)$. This matrix has $O(n_1p_1)$ rows and $O(n_2p_2)$ columns, where $n_i$ and $p_i$ are the basis dimension and spline order, respectively, in the $i$th parametric direction \cite[][Section 3.2]{hofreither2018black}. Since only a rank-1 approximation of $\tilde{\mathcal{R}}(M)$ is sought, the SVD Lanczos process is a viable option, as suggested in \cite{van1993approximation}. Assuming that $k$ iterations of SVD Lanczos are needed, then, for uniform discretization parameters ($n_1=n_2=n$ and $p_1=p_2=p$), the overall cost is $O(a+kn^2p^2)$. The additional cost incurred compares favorably with existing assembly algorithms, where the assembly cost typically ranges from $O(rn^2p^2)$ to $O(n^2p^6)$ for 2D discretizations, where $r$ is an integer; see e.g. \citep{mantzaflaris2017low, calabro2017fast, antolin2015efficient, pan2022efficient}. Instead of forming the mass matrix explicitly, one may directly turn to low-rank matrix and tensor decompositions, which are at the core of several existing assembly algorithms \citep{hofreither2018black, mantzaflaris2017low, scholz2018partial}. For instance, ACA with partial pivoting is used in \cite{hofreither2018black} for computing a low-rank approximation of $\tilde{\mathcal{R}}(M)$. For a rank-1 approximation, the cost amounts to evaluating a single row and column of $M$, which is $O(np^5)$ if Gauss quadrature is used for evaluating its entries \cite{hofreither2018black}. However, the approximation may be quite poor in comparison to the best rank-1 approximation obtained by truncating the SVD of $\tilde{\mathcal{R}}(M)$. \\

\textbf{Error analysis}:
Clearly, we can expect the approximation error of $M$ by $B \otimes C$ to be related to the low-rank approximation error of $\mathcal{R}(M)$ by $\vectorization(B)\vectorization(C)^T$. The next argument attempts to formalize this intuition. Assuming that $\mathcal{R}(M)$ has rank $r$, its reduced singular value decomposition (SVD) is
\begin{equation*}
    \mathcal{R}(M)=U_r\Sigma_{rr} V_r^T = \sigma_1 \mathbf{u}_1\mathbf{v}_1^T+U_{r-1}\Sigma_{r-1,r-1}V_{r-1}^T
\end{equation*}
where $U_r=[\mathbf{u}_1, \dots, \mathbf{u}_r]$ and $V_r=[\mathbf{v}_1, \dots, \mathbf{v}_r]$ have orthonormal columns and are formed by the left and right singular vectors, respectively, and $\Sigma_{rr}=\diag(\sigma_1, \dots, \sigma_r)$ is the diagonal matrix of decreasingly ordered singular values. In Algorithm \ref{algo: Low_rank_approx}, Line 2, we may set for instance $\mathbf{b}=\sqrt{\sigma_1}\mathbf{u}_1$ and $\mathbf{c}=\sqrt{\sigma_1}\mathbf{v}_1$. Then, by construction, the Frobenius norm of the error $E=\Tilde{M}-M=B \otimes C-M$ is given by
\begin{equation*}
    \|E\|_F=\|\mathcal{R}(E)\|_F=\|\mathcal{R}(M)-\sigma_1 \mathbf{u}_1\mathbf{v}_1^T\|_F=\|U_{r-1}\Sigma_{r-1,r-1}V_{r-1}^T\|_F=\sqrt{\sum_{i=2}^r \sigma_i^2}.
\end{equation*}
Our ability to approximate the spectrum of $M$ also depends on the low-rank approximation error of $\mathcal{R}(M)$. This fact is a direct consequence of \citep[][Corollary 6.3.8]{horn2012matrix}, applied to the symmetric matrices $M$ and $\tilde{M}$
\begin{equation*}
    \sum_{i=1}^n (\lambda_i(M)-\lambda_i(\Tilde{M}))^2 \leq \|E\|_F^2 = \sum_{i=2}^r \sigma_i^2.
\end{equation*}
The error we are committing will surely affect all eigenvalues. However, our main concern when approximating $M$ by $\tilde{M}$ is to guarantee that $\lambda_k(K, \tilde{M})$ remains close to $\lambda_k(K, M)$ for the smallest eigenvalues while underestimating the largest eigenvalues. We have already noted in Corollary \ref{cor:equivalence_conditions} that if the error $E=\tilde{M}-M$ is symmetric positive semidefinite, then $\lambda_k(K, \tilde{M}) \leq \lambda_k(K, M)$. Unfortunately, this property is not satisfied by the nearest Kronecker product approximation, the error matrix $E$ being in general indefinite, and the eigenvalues of $(K,\tilde{M})$ might be greater than those of $(K,M)$. However, as we have seen in Example \ref{ex: indefinite_error}, the positive semidefiniteness of the error is not a necessary condition. Moreover, our strategy relies on a two-level approximation of $M$. It is first approximated by a Kronecker structured matrix $\tilde{M}$ and then $\tilde{M}$ is approximated by a lumped mass preconditioner $\tilde{P}_{ii}$. The relation between $\lambda_k(K,\tilde{M})$ and $\lambda_k(K,\tilde{P}_{ii})$ is well understood thanks to Theorem \ref{th:family_lumped_mass_prec_ext}, guaranteeing that $\lambda_k(K,\tilde{P}_{ii}) \leq \lambda_k(K,\tilde{M})$. The problem of understanding the relation between $\lambda_k(K,M)$ and $\lambda_k(K,\tilde{M})$ once again amounts to assessing the quality of the preconditioner $\tilde{M}$. From \eqref{eq:eig_pert_bounds_mass},
\begin{equation*}
    \lambda_1(M, \Tilde{M}) \leq \frac{\lambda_k(K, \Tilde{M})}{\lambda_k(K, M)} \leq \lambda_n(M, \Tilde{M}).
\end{equation*}
Although $\lambda_n(M, \tilde{M})$ may be greater than 1 due to the indefiniteness of the error, we will show numerically that it remains very close to 1. The use of Kronecker product approximations as preconditioners is not a new idea. It has already been explored for various applications including image processing \cite{nagy2006kronecker}, Markov chains \cite{langville2004akronecker}, stochastic Galerkin finite element discretizations \cite{ullmann2010kronecker} and isogeometric analysis \citep{gao2014fast, loli2021easy, sangalli2016isogeometric}. However, only few authors have undertaken the challenge of obtaining bounds for the preconditioned spectrum and most attempts are application specific. The following theorem is a variation of a general result presented in \citep[][Lemma 5.3]{ullmann2010kronecker} obtained by exhibiting the singular values and giving a more explicit bound on the condition number of the preconditioned matrix. 

\begin{theorem}
\label{th:cond_NKP_prec}
Let $M=\sum_{i=1}^r B_i \otimes C_i$ be a symmetric positive definite matrix of Kronecker rank $r$ where each factor matrices $B_i, C_i$ are symmetric for $i=1,\dots,r$ and let $\mathcal{R}(M)=\sum_{i=1}^r \sigma_i \mathbf{u}_i\mathbf{v}_i^T$ be the reduced singular value decomposition of the reordered matrix. Let $\tilde{M}$ be the nearest Kronecker product approximation. Then, assuming that $\delta=\sum_{i=2}^r \frac{\sigma_i}{\sigma_1}\max_k |\lambda_k(U_i,U_1)| \max_k |\lambda_k(V_i,V_1)| <1$ where $\mathbf{u}_i=\vectorization(U_i)$ and $\mathbf{v}_i=\vectorization(V_i)$, it holds
\begin{equation*}
    \kappa(\tilde{M}^{-1/2}M\tilde{M}^{-1/2}) \leq \frac{1+\delta}{1-\delta}
\end{equation*}
where $\kappa$ denotes the spectral condition number.
\end{theorem}
\begin{proof}
By reverting the ordering, $M$ is expressed as a sum of $r$ Kronecker products
\begin{equation*}
    \mathcal{R}(M)=\sum_{i=1}^r \sigma_i \mathbf{u}_i\mathbf{v}_i^T \iff M=\sum_{i=1}^r \sigma_i (U_i \otimes V_i)
\end{equation*}
where $\mathbf{u}_i=\vectorization(U_i)$ and $\mathbf{v}_i=\vectorization(V_i)$ for $i=1,\dots,r$. Similarly, by construction, $\mathcal{R}(\tilde{M})=\sigma_1 \mathbf{u}_1\mathbf{v}_1^T$ and thus $\tilde{M}=\sigma_1 (U_1 \otimes V_1)$. The symmetry and positive definiteness of $U_1$ and $V_1$ has already been established in \citep[][Theorem 5.8]{van1993approximation}. We now establish symmetry of all $U_i$ and $V_i$ for $i=1,\dots,r$ using proof arguments similar to those presented in \citep[][Theorem 4.1]{langville2004akronecker}. By the singular value decomposition, $\mathcal{R}(M)\mathbf{v}_k=\sigma_k\mathbf{u}_k$ and $\mathcal{R}(M)^T\mathbf{u}_k=\sigma_k\mathbf{v}_k$ for $k=1,\dots,r$. Since $\mathcal{R}(M)=\sum_{i=1}^r \sigma_i \mathbf{u}_i\mathbf{v}_i^T=\sum_{i=1}^r \mathbf{b}_i\mathbf{c}_i^T$ with $\mathbf{b}_i=\vectorization(B_i)$ and $\mathbf{c}_i=\vectorization(C_i)$, we obtain
\begin{align*}
    \mathbf{u}_k&=\frac{1}{\sigma_k}\mathcal{R}(M)\mathbf{v}_k=\frac{1}{\sigma_k}\sum_{k=1}^r\mathbf{b}_i\mathbf{c}_i^T\mathbf{v}_k=\sum_{i=1}^r \alpha_{ik} \mathbf{b}_i, \\
    \mathbf{v}_k&=\frac{1}{\sigma_k}\mathcal{R}(M)^T\mathbf{u}_k=\frac{1}{\sigma_k}\sum_{k=1}^r\mathbf{c}_i\mathbf{b}_i^T\mathbf{u}_k=\sum_{i=1}^r \beta_{ik} \mathbf{c}_i,
\end{align*}
with $\alpha_{ik}=\sigma_k^{-1}\mathbf{c}_i^T\mathbf{v}_k$ and $\beta_{ik}=\sigma_k^{-1}\mathbf{b}_i^T\mathbf{u}_k$. By reshaping $\mathbf{u}_k$ and $\mathbf{v}_k$ to matrices, we obtain
\begin{equation*}
    U_k=\sum_{i=1}^r\alpha_{ik}B_i, \qquad V_k=\sum_{i=1}^r\beta_{ik}C_i.
\end{equation*}
Since all $B_i$ and $C_i$ are symmetric by assumption, so are $U_k$ and $V_k$ for all $k=1,\dots,r$. We now establish the bound on the condition number by noting that
\begin{align*}
    |\lambda_k(M,\tilde{M})-1| \leq \|\tilde{M}^{-1/2}M\tilde{M}^{-1/2}-I\|_2&=\bigg\|\sum_{i=2}^r \frac{\sigma_i}{\sigma_1}(U_1^{-1/2}U_iU_1^{-1/2} \otimes V_1^{-1/2}V_iV_1^{-1/2})\bigg\|_2 \\
    &\leq \sum_{i=2}^r \frac{\sigma_i}{\sigma_1}\|U_1^{-1/2}U_iU_1^{-1/2}\|_2 \|V_1^{-1/2}V_iV_1^{-1/2}\|_2 \\
    &=\sum_{i=2}^r\frac{\sigma_i}{\sigma_1}\max_k |\lambda_k(U_i,U_1)| \max_k |\lambda_k(V_i,V_1)|=\delta.
\end{align*}
Assuming that $\delta<1$, since $\tilde{M}^{-1/2}M\tilde{M}^{-1/2}$ is symmetric positive definite,
\begin{equation*}
    \kappa(\tilde{M}^{-1/2}M\tilde{M}^{-1/2}) =\frac{\lambda_n(M,\tilde{M})}{\lambda_1(M,\tilde{M})} \leq \frac{1+\delta}{1-\delta}.
\end{equation*}
\end{proof}
The previous theorem is not entirely satisfactory since $\delta$ still depends on the unknown quantities $\max_k |\lambda_k(U_i,U_1)|$ and $\max_k |\lambda_k(V_i,V_1)|$, which we were unable to bound in any convenient way. Nevertheless, our numerical experiments indicated that these quantities were $O(1)$. Most importantly, $\delta$ depends on the singular value ratios $\frac{\sigma_i}{\sigma_1}$ for $i=2,\dots,r$. Therefore, a large singular value gap ($\sigma_2 \ll \sigma_1$) is important for constructing a good preconditioner and the smallest singular values will not contribute much to the sum. The Appendix further expands upon the error committed by a Kronecker product approximation. The quality of the approximation is now assessed numerically on a few examples. 

\begin{example}[Low-rank approximation]
\label{ex: low_rank_2D}
We consider two 2D examples from isogeometric analysis on nontrivial geometries. A brief description is given below.
\begin{itemize}[noitemsep]
    \item \textbf{Example 1}: The domain is a stretched square already used by several other authors in \citep{gao2014fast, loli2021easy} (Figure \ref{fig: kite}).
    \item \textbf{Example 2}: The domain is a perturbed triangle featuring a re-entrant corner (Figure \ref{fig: reentrant_corner}).
\end{itemize}
We consider the same discretization parameters in all examples, namely cubic B-splines of maximal smoothness and $20$ subdivisions in each parametric direction. The factor matrices $B$ and $C$ are obtained by truncating the singular value decomposition of a dense submatrix extracted from $\mathcal{R}(M)$ and containing all its nonzero entries. The rank of $\mathcal{R}(M)$ is exactly $2$ for Example~1 since $|\det(J_F(\hat{\mathbf{x}})|$ is the sum of two separable functions for the stretched square domain \cite{gao2014fast}. The second example is more challenging with $\mathcal{R}(M)$ having numerical rank $22$. For each case, the discrete eigenvalues using the consistent mass matrix are compared to the approximate eigenvalues using the nearest Kronecker product (NKP) preconditioner with and without mass lumping. The curve using $\tilde{P}_{33}$ nearly overlaps with the one for the NKP preconditioner and is omitted.

\begin{center}
\begin{minipage}[t]{.48\linewidth}
\vspace{0pt}
\begin{figure}[H]
    \centering
    \includegraphics[width=\textwidth]{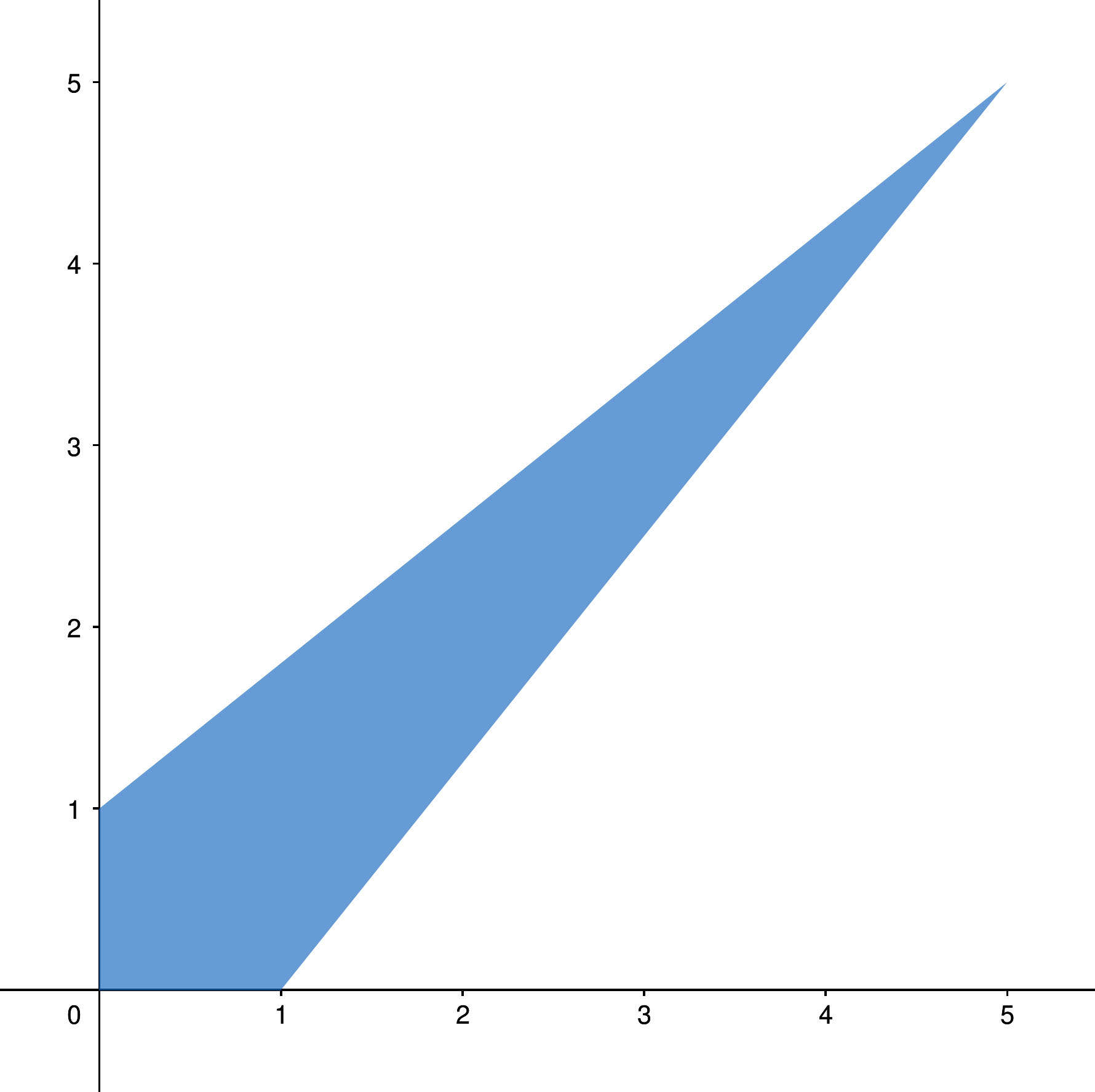}
    \caption{Stretched square domain}
    \label{fig: kite}
\end{figure}
\end{minipage}
\hspace{2pt}
\begin{minipage}[t]{.48\linewidth}
\vspace{0pt}
\begin{figure}[H]
    \centering
    \includegraphics[width=\textwidth]{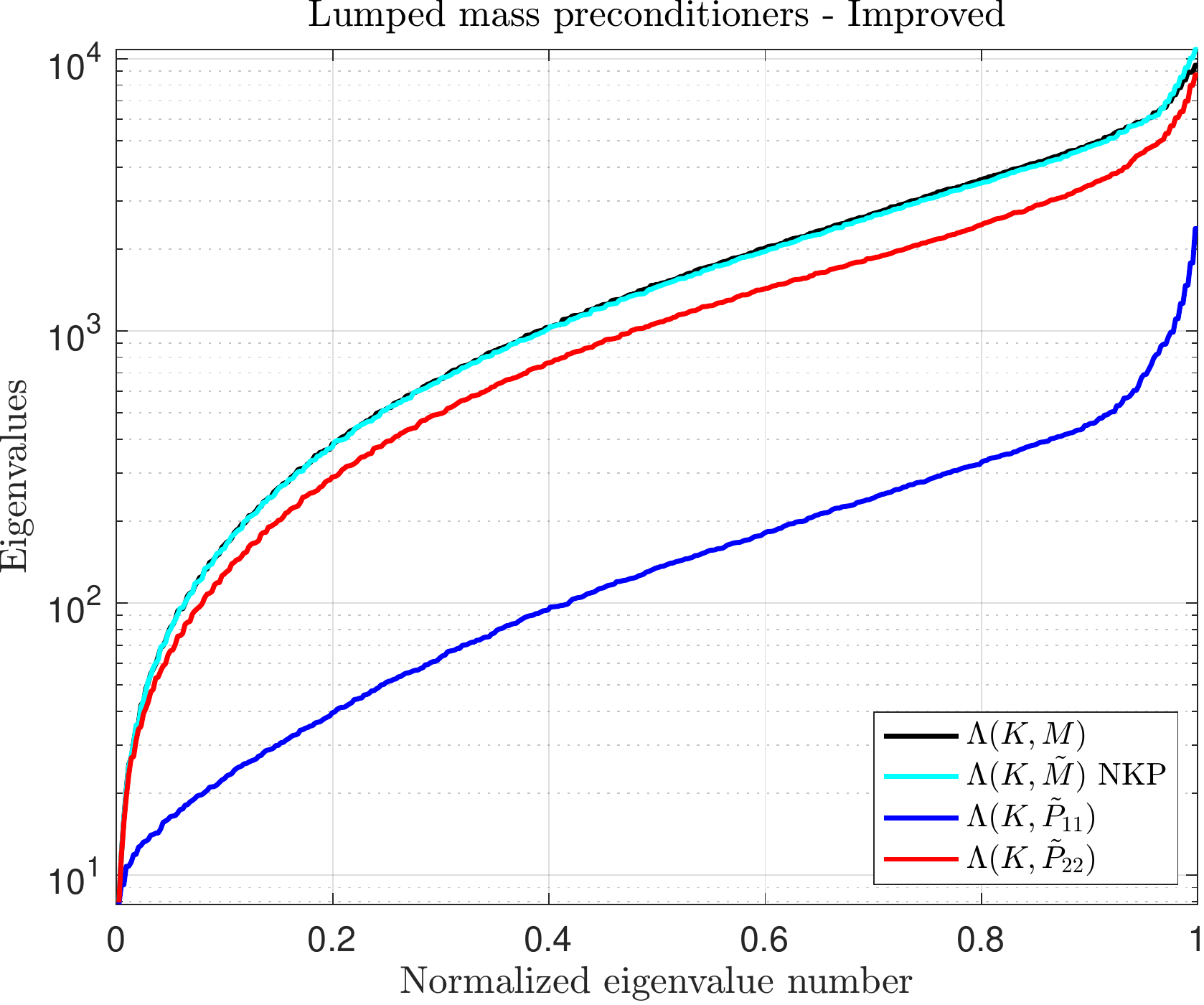}
    \caption{Comparison of $\Lambda(K,M)$ and $\Lambda(K,\tilde{M})$ for the nearest Kronecker product (NKP) preconditioner and $\Lambda(K,\tilde{P}_{ii})$ for $i=1,2,3$}
    \label{fig: 2D_Laplace_stretched_square_LM_+4.5_p3_n20}
\end{figure}
\end{minipage}
\end{center} 

\begin{center}
\begin{minipage}[t]{.48\linewidth}
\vspace{0pt}
\begin{figure}[H]
    \centering
    \includegraphics[width=\textwidth]{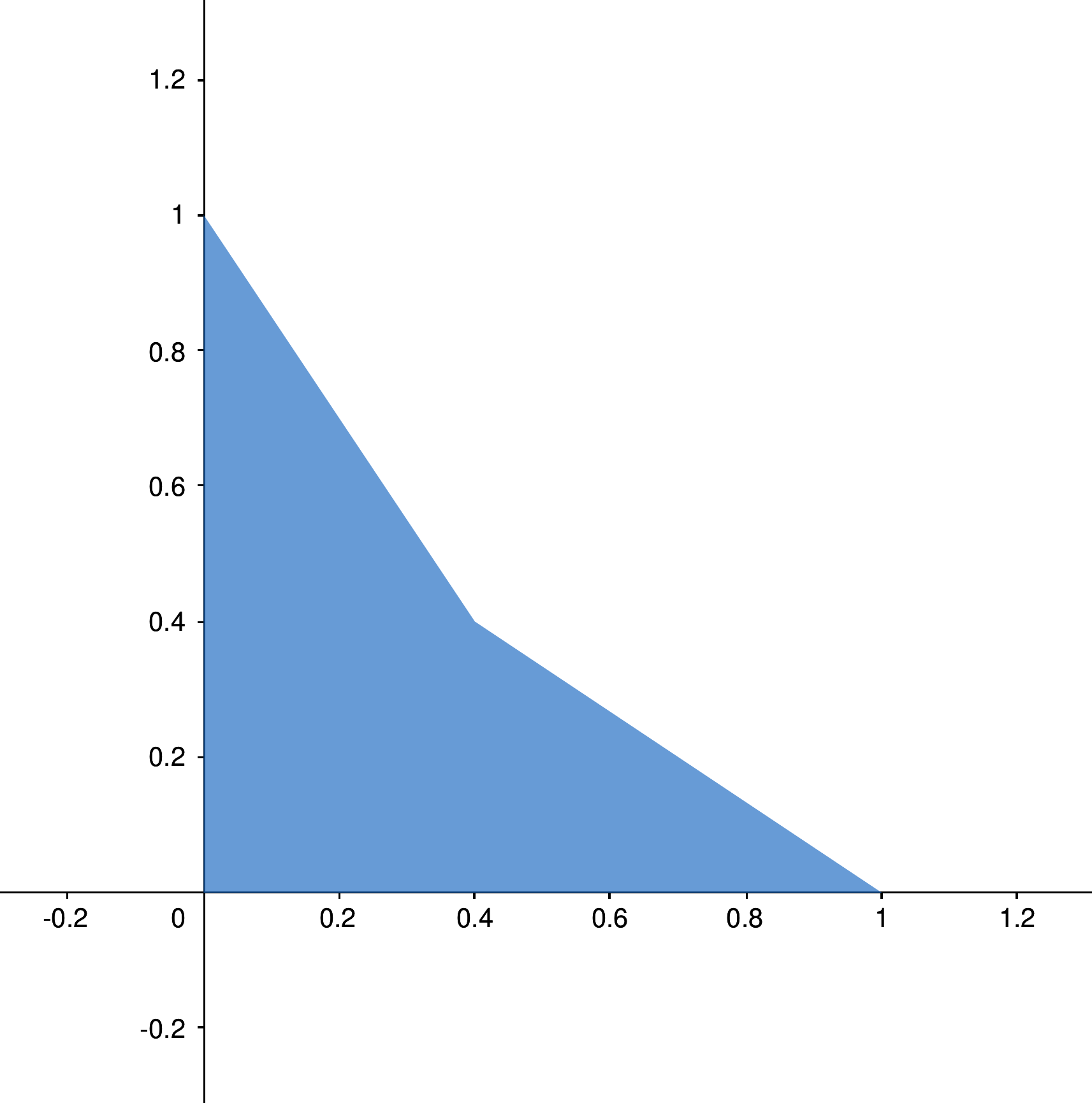}
    \caption{Domain with a re-entrant corner}
    \label{fig: reentrant_corner}
\end{figure}
\end{minipage}
\hspace{2pt}
\begin{minipage}[t]{.48\linewidth}
\vspace{0pt}
\begin{figure}[H]
    \centering
    \includegraphics[width=\textwidth]{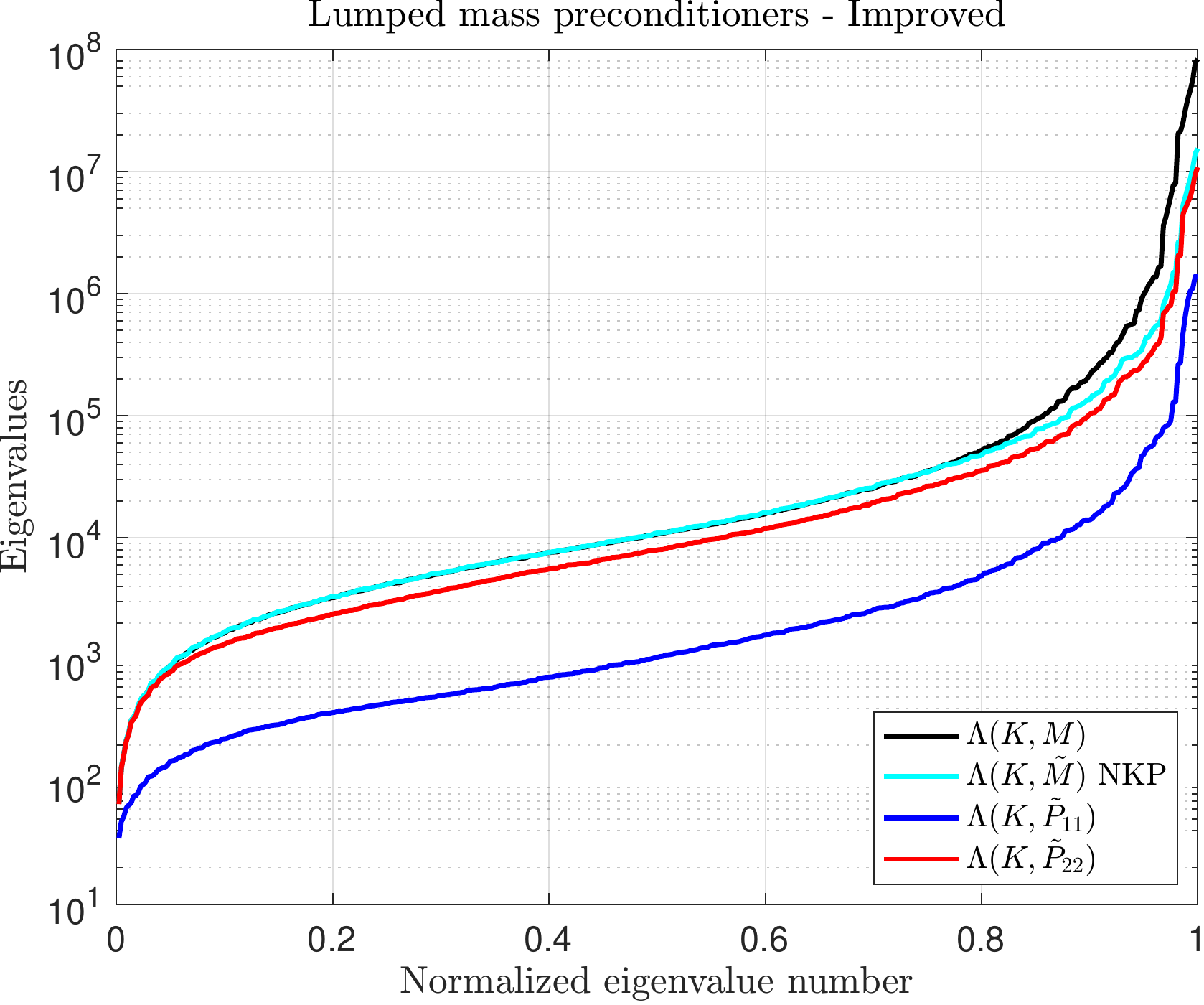}
    \caption{Comparison of $\Lambda(K,M)$ and $\Lambda(K,\tilde{M})$ for the nearest Kronecker product (NKP) preconditioner and $\Lambda(K,\tilde{P}_{ii})$ for $i=1,2,3$}
    \label{fig: 2D_Laplace_stretched_square_LM_-1e-1_p3_n20}
\end{figure}
\end{minipage}
\end{center} 

The condition number of the preconditioned mass matrix with the NKP preconditioner is $2.18$ and $39.1$ for Examples 1 and 2, respectively. The assumptions of Theorem \ref{th:cond_NKP_prec} are satisfied for the first example and the upper bound on the condition number is $2.73$. The assumptions are not satisfied for the second example due to the much slower decay of the singular values. In Figure \ref{fig: 2D_Laplace_stretched_square_LM_+4.5_p3_n20}, we notice that the largest eigenvalues of $(K,\tilde{M})$ are actually slightly greater than those of $(K,M)$. This situation cannot be prevented but our second level approximation involving mass lumping usually fixes it. On the contrary, in Figure \ref{fig: 2D_Laplace_stretched_square_LM_-1e-1_p3_n20}, the largest eigenvalues of $(K,M)$ are greatly underestimated when using the NKP preconditioner. This situation is very favorable and was generally experienced in our experiments with complicated problems. Any mass lumping done subsequently will bring down the eigenvalues further. However, the bandwidth must still be large enough for the smallest eigenvalues to be well approximated. The selective damping of the higher modes will be addressed in future work.
\end{example}

\subsection{Approximation for 3D problems}
\label{se: KP_approx_3D}
The natural analog of the minimization problem for 3D discretizations is
\begin{equation*}
    \min \phi_M(B,C,D)=\min \|M-B \otimes C \otimes D\|_F
\end{equation*}
where $M \in \mathbb{R}^{N \times N}$ with $N=n_1n_2n_3$ and the factor matrices $B \in \mathbb{R}^{n_1 \times n_1}$, $C \in \mathbb{R}^{n_2 \times n_2}$ and $D \in \mathbb{R}^{n_3 \times n_3}$. This extension was not originally investigated by Van Loan and Pitsianis \cite{van1993approximation} but was later analyzed in \cite{langville2004akronecker}. The definition of the mapping $\mathcal{R}$ must be adjusted such that $\mathcal{R}(B \otimes C \otimes D)=\vectorization(B) \circ \vectorization(C) \circ \vectorization(D)$. More generally, it transforms a Kronecker product of $p$ factor matrices to a $p$th order rank-1 tensor obtained by taking the outer product of the vectorization of the factor matrices. The symbol $\circ$ is used to denote the outer product. Note that this definition is consistent with the one previously introduced for the 2D case since $\mathbf{b} \circ \mathbf{c}=\mathbf{b}\mathbf{c}^T$. The minimization problem then translates to finding the best rank-1 tensor approximation of $\mathcal{R}(M)$. Unfortunately, contrary to the matrix case, we cannot hope to compute this \textit{best} rank-1 approximation and we must content ourselves with \textit{good} rank-1 approximations such as those computed with the high order SVD (HOSVD) \cite{de2000multilinear}. Another major shortcoming of the extension to higher dimensions is the lack of theoretical results related to the computed factor matrices. Some proof arguments used in \cite{van1993approximation} for the two-dimensional case cannot be straightforwardly extended to higher dimensions. On the other hand, Algorithm \ref{algo: Low_rank_approx} can be straightforwardly adapted to higher dimensional problems. The computational results were quite similar to those reported in Example \ref{ex: low_rank_2D} and are omitted. \\

\textbf{Complexity analysis}: 
In the 3D case, $\tilde{\mathcal{R}}(M)$ has roughly size $n_1p_1 \times n_2p_2 \times n_3p_3$. There exists various extensions of ACA to tensors. Following the discussion in \citep[][Section 4.3]{hofreither2018black} and in analogy to the 2D case, computing a rank-1 tensor approximation of $\tilde{\mathcal{R}}(M)$ using ACA with partial pivoting requires evaluating three of its fibers. If Gauss quadrature is used for this purpose and assuming uniform discretization parameters ($n_1=n_2=n_3=n$ and $p_1=p_2=p_3=p$), then the cost amounts to $O(np^7)$. However, similarly to the 2D case, the approximation may be quite poor depending on the pivot choice in the ACA algorithm. Alternatively, the strategies proposed in \citep{mantzaflaris2017low, scholz2018partial} are also very appealing, especially because the mass matrix is only needed in Kronecker format (i.e. only the factor matrices of the Kronecker product are needed). Using the partial tensor decomposition proposed in \cite{scholz2018partial}, the expected complexity is $O(n^2p^6)$.

\section{Conclusion}
\label{se: conclusion}
Mass lumping is used on a daily basis for industrial applications in structural dynamics and consists in replacing the consistent mass matrix with some kind of easily invertible (typically diagonal) approximation. In this article, we have unraveled some attractive mathematical features of mass lumping partly responsible for its success. Our results hold under rather broad assumptions, indicating that some properties are not specific to any discretization method. More generally, we have highlighted the close connection between the quality of a preconditioner and the approximation of the eigenvalues. Based on this insight, we have generalized the concept of mass lumping and its properties to matrices with more complex structures including Kronecker products, whose structure can be used to solve linear systems very efficiently. We have then proposed a two-level approximation of the mass matrix in case it is not already expressed as a Kronecker product. Numerical experiments have highlighted the practical usefulness of this approximation in the context of isogeometric analysis by increasing the critical time step of explicit time integration schemes. However, our generalization of mass lumping improves the approximation of all discrete eigenvalues, including the outliers. Outlier eigenvalues are a distinctive feature of maximally smooth isogeometric discretizations. Apart from being completely meaningless, they significantly constrain the critical time step of explicit time integration schemes. Understanding and removing these outliers is therefore of high interest and several follow-up papers have been written on this topic; see e.g., \cite{hughes2014finite,chan2018multi,deng2021boundary, hiemstra2021removal, gallistl2017stability,sande2019sharp,manni2022application}. Future work will attempt to selectively approximate the accurate low-frequency part of the spectrum while undermining the outliers. One possible research direction consists in combining the mass lumping techniques presented here with the optimal spline spaces in \cite{floater2019optimal} which were shown to be outlier-free in \cite{sande2019sharp} and \cite{manni2022application}.

\section*{Acknowledgments}
The second author kindly acknowledges the support of the SNSF through the project ``Smoothness in Higher Order Numerical Methods’’ n. TMPFP2\_209868 (SNSF Swiss Postdoctoral Fellowship 2021). 
\newline The third author kindly acknowledges the support of SNSF through the project ``Design-through-Analysis (of PDEs): the litmus test’’ n. 40B2-0 187094 (BRIDGE Discovery 2019).

\appendix
\section{Appendix}
The method proposed in Section \ref{se: NKP} consists in approximating the mass matrix by a Kronecker product as a preliminary step to applying mass lumping. The error committed in this step depends on the singular values of the reordered matrix $\mathcal{R}(M)$, which are independent of the discretization parameters. This truncation error limits the accuracy of all quantities computed with the approximate mass matrix $\tilde{M}$, including the generalized eigenvalues of $(K,\tilde{M})$. In order to fix ideas, we consider a cubic discretization of the Laplace eigenvalue problem on the unit square with a nontrivial density function given by $\rho(x,y)=|\sin(xy)|+x+y+1$. The singular values of $\mathcal{R}(M)$ are shown in Figure \ref{fig: 2D_Laplace_unit_square_coeff_singv_RM_p3} and reveal that the mass matrix can be approximated up to machine precision by a Kronecker rank $6$ matrix. The convergence test carried out in Figure \ref{fig: 2D_Laplace_unit_square_coeff_rel_error_smallest_eig_p3_r1-4} shows that the discrete eigenfrequencies of $(K,M)$ and $(K,\tilde{M})$ converge at the same rate provided the truncation error is small enough (i.e. the Kronecker rank $r$ of the approximation is sufficiently large). Comparing Figures \ref{fig: 2D_Laplace_unit_square_coeff_singv_RM_p3} and \ref{fig: 2D_Laplace_unit_square_coeff_rel_error_smallest_eig_p3_r1-4} also reveals a close connection between the truncation error, which only depends on the singular values of $\mathcal{R}(M)$, and the relative eigenfrequency error. The behavior shown in Figure \ref{fig: 2D_Laplace_unit_square_coeff_rel_error_smallest_eig_p3_r1-4} indicates that the error committed by approximating the consistent mass by the nearest Kronecker product may dominate the eigenfrequency error. Indeed, denoting $\omega_1$ the smallest exact eigenfrequency and $\omega_{h,1}$ and $\tilde{\omega}_{h,1}$ the smallest discrete eigenfrequencies for the consistent mass and nearest Kronecker product, respectively, then
\begin{equation*}
    \left|\frac{\omega_1-\tilde{\omega}_{h,1}}{\omega_1}\right| = \left|1-\frac{\tilde{\omega}_{h,1}}{\omega_{h,1}}\frac{\omega_{h,1}}{\omega_1}\right| = \left| 1- \frac{\tilde{\omega}_{h,1}}{\omega_{h,1}} + \underbrace{1-\frac{\omega_{h,1}}{\omega_1}}_{O(h^{2p})} -\left(1- \frac{\tilde{\omega}_{h,1}}{\omega_{h,1}}\right)\underbrace{\left(1-\frac{\omega_{h,1}}{\omega_1}\right)}_{O(h^{2p})}\right|
\end{equation*}
While the error $1-\frac{\omega_{h,1}}{\omega_1}$ committed by the consistent mass is known to converge at a high order rate ($2p$) under mild assumptions, the error $1- \frac{\tilde{\omega}_{h,1}}{\omega_{h,1}}$ committed by approximating the consistent mass by the nearest Kronecker product seems bounded independently of the mesh size and explains why the convergence curve for $(K, \tilde{M})$ flattens out when this error term dominates. This statement could be proved if one showed that the consistent mass and its nearest Kronecker product approximation are spectrally equivalent; i.e. there exist constants $c_1,c_2>0$ independent of $h$ such that
\begin{equation*}
    c_1 \tilde{M} \preceq M \preceq c_2 \tilde{M}.
\end{equation*}
To our knowledge, this is an open research problem.

\begin{figure}[H]
    \centering
    \begin{subfigure}[t]{0.48\textwidth}
    \centering
    \includegraphics[width=\textwidth]{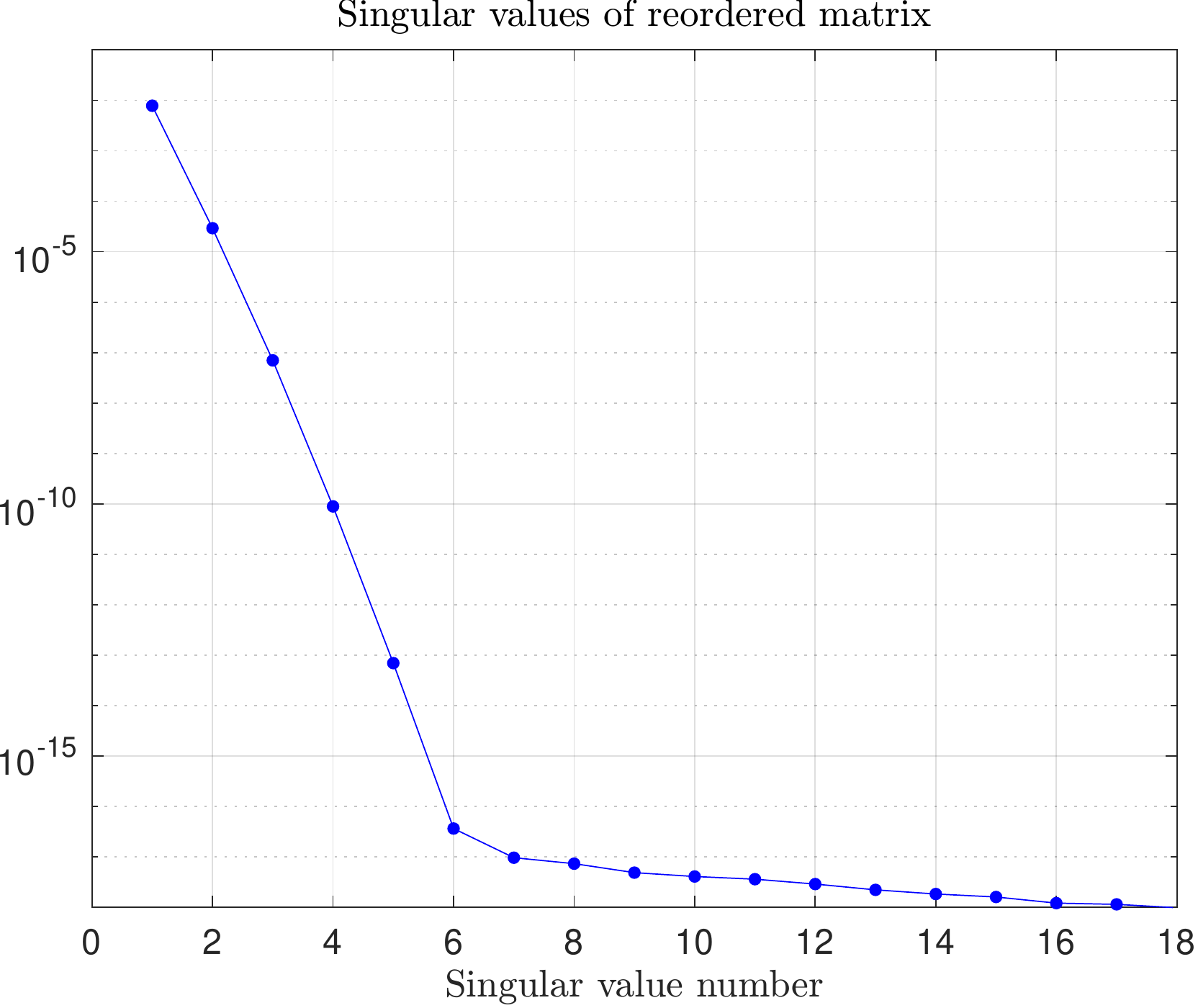}
    \caption{Singular values of $\mathcal{R}(M)$}
    \label{fig: 2D_Laplace_unit_square_coeff_singv_RM_p3}
    \end{subfigure}
    \hfill
    \begin{subfigure}[t]{0.48\textwidth}
    \centering
    \includegraphics[width=\textwidth]{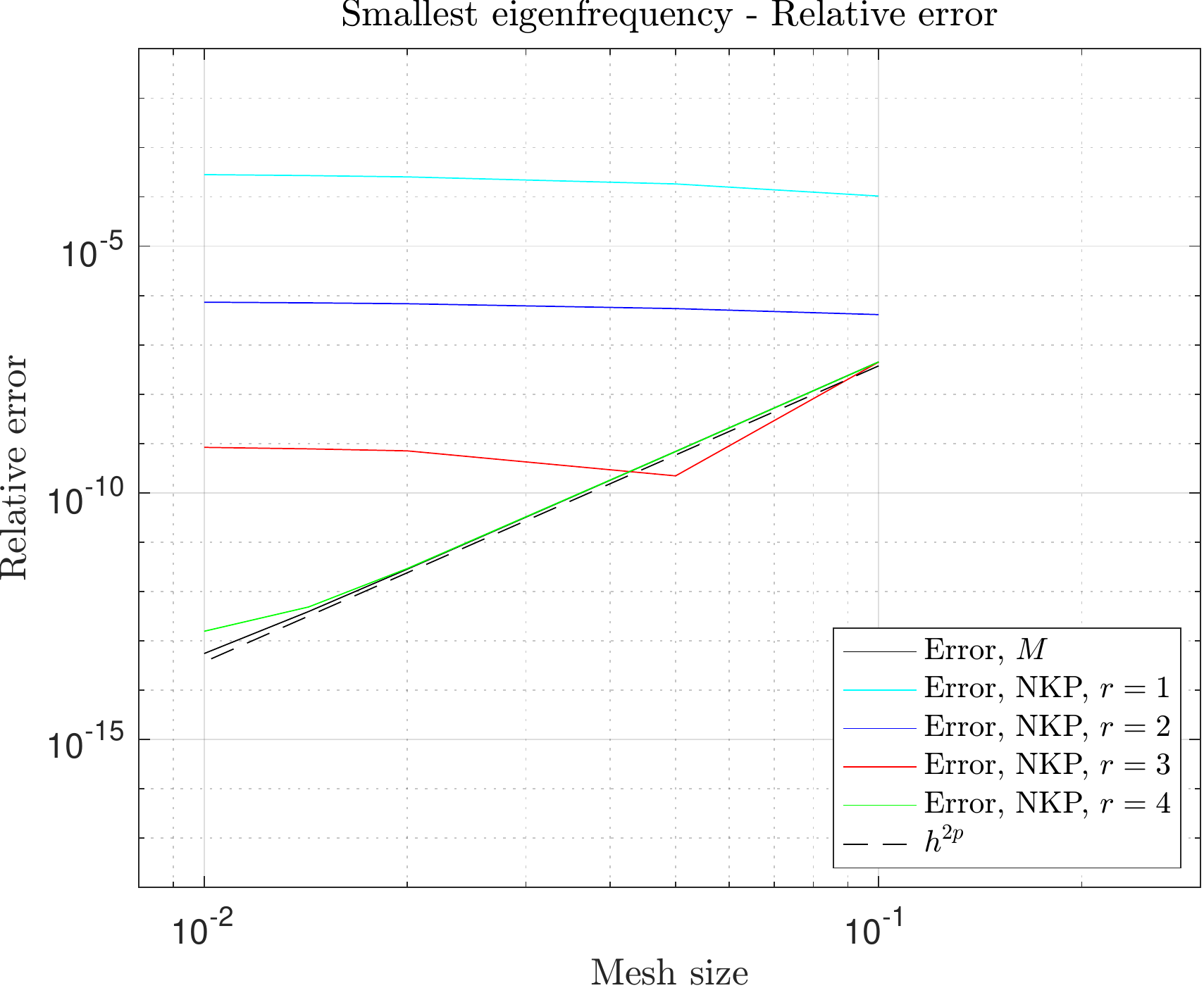}
    \caption{Relative eigenfrequency error computed with a reference solution}
    \label{fig: 2D_Laplace_unit_square_coeff_rel_error_smallest_eig_p3_r1-4}
    \end{subfigure}
    \hfill
    \caption{Nearest Kronecker product approximation}
    \label{fig: 2D_Laplace_unit_square_coeff_NKP}
\end{figure}


\begin{thebibliography}{10}
\expandafter\ifx\csname url\endcsname\relax
  \def\url#1{\texttt{#1}}\fi
\expandafter\ifx\csname urlprefix\endcsname\relax\def\urlprefix{URL }\fi
\expandafter\ifx\csname href\endcsname\relax
  \def\href#1#2{#2} \def\path#1{#1}\fi

\bibitem{zienkiewicz2005finite}
O.~C. Zienkiewicz, R.~L. Taylor, J.~Z. Zhu, {The finite element method: its
  basis and fundamentals}, Elsevier, 2005.

\bibitem{hughes2012finite}
T.~J. Hughes, {The finite element method: linear static and dynamic finite
  element analysis}, Courier Corporation, 2012.

\bibitem{duczek2019mass}
S.~Duczek, H.~Gravenkamp, {Mass lumping techniques in the spectral element
  method: On the equivalence of the row-sum, nodal quadrature, and diagonal
  scaling methods}, Computer Methods in Applied Mechanics and Engineering 353
  (2019) 516--569.

\bibitem{hughes2005isogeometric}
T.~J. Hughes, J.~A. Cottrell, Y.~Bazilevs, {Isogeometric analysis: CAD, finite
  elements, NURBS, exact geometry and mesh refinement}, Computer methods in
  applied mechanics and engineering 194~(39-41) (2005) 4135--4195.

\bibitem{cottrell2009isogeometric}
J.~A. Cottrell, T.~J. Hughes, Y.~Bazilevs, {Isogeometric analysis: toward
  integration of CAD and FEA}, John Wiley \& Sons, 2009.

\bibitem{cottrell2006isogeometric}
J.~A. Cottrell, A.~Reali, Y.~Bazilevs, T.~J. Hughes, {Isogeometric analysis of
  structural vibrations}, Computer methods in applied mechanics and engineering
  195~(41-43) (2006) 5257--5296.

\bibitem{leidinger2020explicit}
L.~Leidinger, {Explicit Isogeometric B-Rep Analysis for Nonlinear Dynamic Crash
  Simulations}, Ph.D. thesis, Technische Universit{\"a}t M{\"u}nchen (2020).

\bibitem{anitescu2019isogeometric}
C.~Anitescu, C.~Nguyen, T.~Rabczuk, X.~Zhuang, {Isogeometric analysis for
  explicit elastodynamics using a dual-basis diagonal mass formulation},
  Computer Methods in Applied Mechanics and Engineering 346 (2019) 574--591.

\bibitem{fried1975finite}
I.~Fried, D.~S. Malkus, {Finite element mass matrix lumping by numerical
  integration with no convergence rate loss}, International Journal of Solids
  and Structures 11~(4) (1975) 461--466.

\bibitem{cohen1994higher}
G.~Cohen, P.~Joly, N.~Tordjman, {Higher-order finite elements with mass-lumping
  for the 1D wave equation}, Finite elements in analysis and design 16~(3-4)
  (1994) 329--336.

\bibitem{durufle2009influence}
M.~Durufl{\'e}, P.~Grob, P.~Joly, {Influence of Gauss and Gauss-Lobatto
  quadrature rules on the accuracy of a quadrilateral finite element method in
  the time domain}, Numerical Methods for Partial Differential Equations: An
  International Journal 25~(3) (2009) 526--551.

\bibitem{cohen2001higher}
G.~Cohen, P.~Joly, J.~E. Roberts, N.~Tordjman, {Higher order triangular finite
  elements with mass lumping for the wave equation}, SIAM Journal on Numerical
  Analysis 38~(6) (2001) 2047--2078.

\bibitem{malkus1986zero}
D.~S. Malkus, M.~E. Plesha, {Zero and negative masses in finite element
  vibration and transient analysis}, Computer methods in applied mechanics and
  engineering 59~(3) (1986) 281--306.

\bibitem{hinton1976note}
E.~Hinton, T.~Rock, O.~Zienkiewicz, {A note on mass lumping and related
  processes in the finite element method}, Earthquake Engineering \& Structural
  Dynamics 4~(3) (1976) 245--249.

\bibitem{macek1995mass}
R.~W. Macek, B.~H. Aubert, {A mass penalty technique to control the critical
  time increment in explicit dynamic finite element analyses}, Earthquake
  engineering \& structural dynamics 24~(10) (1995) 1315--1331.

\bibitem{olovsson2005selective}
L.~Olovsson, K.~Simonsson, M.~Unosson, {Selective mass scaling for explicit
  finite element analyses}, International Journal for Numerical Methods in
  Engineering 63~(10) (2005) 1436--1445.

\bibitem{stoter2022variationally}
S.~K. Stoter, T.-H. Nguyen, R.~R. Hiemstra, D.~Schillinger, {Variationally
  consistent mass scaling for explicit time-integration schemes of lower-and
  higher-order finite element methods}, arXiv preprint arXiv:2201.10475 (2022).

\bibitem{tkachuk2013variational}
A.~Tkachuk, M.~Bischoff, {Variational methods for selective mass scaling},
  Computational Mechanics 52~(3) (2013) 563--570.

\bibitem{olovsson2006iterative}
L.~Olovsson, K.~Simonsson, {Iterative solution technique in selective mass
  scaling}, Communications in numerical methods in engineering 22~(1) (2006)
  77--82.

\bibitem{quarteroni2009numerical}
A.~Quarteroni, {Numerical models for differential problems}, Vol.~2, Springer,
  2009.

\bibitem{newmark1959method}
N.~M. Newmark, {A method of computation for structural dynamics}, Journal of
  the engineering mechanics division 85~(3) (1959) 67--94.

\bibitem{hughes2014finite}
T.~J. Hughes, J.~A. Evans, A.~Reali, {Finite element and NURBS approximations
  of eigenvalue, boundary-value, and initial-value problems}, Computer Methods
  in Applied Mechanics and Engineering 272 (2014) 290--320.

\bibitem{stewart1990matrix}
G.~Stewart, J.~Sun, {Matrix Perturbation Theory}, Computer Science and
  Scientific Computing, ACADEMIC Press, INC, 1990.

\bibitem{parlett1998symmetric}
B.~N. Parlett, {The symmetric eigenvalue problem}, SIAM, 1998.

\bibitem{horn2012matrix}
R.~A. Horn, C.~R. Johnson, {Matrix analysis}, Cambridge university press, 2012.

\bibitem{schweitzer2013variational}
M.~A. Schweitzer, {Variational mass lumping in the partition of unity method},
  SIAM Journal on Scientific Computing 35~(2) (2013) A1073--A1097.

\bibitem{yang2017rigorous}
Y.~Yang, H.~Zheng, M.~Sivaselvan, {A rigorous and unified mass lumping scheme
  for higher-order elements}, Computer Methods in Applied Mechanics and
  Engineering 319 (2017) 491--514.

\bibitem{gao2014fast}
L.~Gao, V.~M. Calo, {Fast isogeometric solvers for explicit dynamics}, Computer
  Methods in Applied Mechanics and Engineering 274 (2014) 19--41.

\bibitem{loli2021easy}
G.~Loli, G.~Sangalli, M.~Tani, {Easy and efficient preconditioning of the
  isogeometric mass matrix}, Computers \& Mathematics with Applications (2021).

\bibitem{vazquez2016new}
R.~V{\'a}zquez, {A new design for the implementation of isogeometric analysis
  in Octave and Matlab: GeoPDEs 3.0}, Computers \& Mathematics with
  Applications 72~(3) (2016) 523--554.

\bibitem{stewart1975gershgorin}
G.~W. Stewart, {Gershgorin theory for the generalized eigenvalue problem
  Ax=$\lambda$Bx}, Mathematics of Computation (1975) 600--606.

\bibitem{stewart1979pertubation}
G.~W. Stewart, {Pertubation bounds for the definite generalized eigenvalue
  problem}, Linear algebra and its applications 23 (1979) 69--85.

\bibitem{sun1982note}
J.-G. Sun, {A note on Stewart's theorem for definite matrix pairs}, Linear
  Algebra and its Applications 48 (1982) 331--339.

\bibitem{crawford1976stable}
C.~R. Crawford, {A stable generalized eigenvalue problem}, SIAM Journal on
  Numerical Analysis 13~(6) (1976) 854--860.

\bibitem{ipsen2009refined}
I.~C. Ipsen, B.~Nadler, {Refined perturbation bounds for eigenvalues of
  Hermitian and non-Hermitian matrices}, SIAM Journal on Matrix Analysis and
  Applications 31~(1) (2009) 40--53.

\bibitem{crawford1970numerical}
C.~R. Crawford, {The numerical solution of the generalized eigenvalue problem},
  Ph.D. thesis, University of Michigan (1970).

\bibitem{golub2013matrix}
G.~H. Golub, C.~F. Van~Loan, {Matrix computations}, JHU press, 2013.

\bibitem{quarteroni2010numerical}
A.~Quarteroni, R.~Sacco, F.~Saleri, {Numerical mathematics}, Vol.~37, Springer
  Science \& Business Media, 2010.

\bibitem{evans2018explicit}
J.~A. Evans, R.~R. Hiemstra, T.~J. Hughes, A.~Reali, {Explicit higher-order
  accurate isogeometric collocation methods for structural dynamics}, Computer
  Methods in Applied Mechanics and Engineering 338 (2018) 208--240.

\bibitem{wathen2015preconditioning}
A.~J. Wathen, {Preconditioning}, Acta Numerica 24 (2015) 329--376.

\bibitem{pearson2020preconditioners}
J.~W. Pearson, J.~Pestana, {Preconditioners for Krylov subspace methods: An
  overview}, GAMM-Mitteilungen 43~(4) (2020) e202000015.

\bibitem{horn1991topics}
R.~A. Horn, C.~R. Johnson, {Topics in Matrix Analysis}, Cambridge University
  Press, 1991.

\bibitem{sangalli2016isogeometric}
G.~Sangalli, M.~Tani, {Isogeometric preconditioners based on fast solvers for
  the Sylvester equation}, SIAM Journal on Scientific Computing 38~(6) (2016)
  A3644--A3671.

\bibitem{voet2020nonlinear}
Y.~Voet, {Nonlinear Finite Elements in Dynamics}, Tech. rep., {\'E}cole
  polytechnique f{\'e}d{\'e}rale de Lausanne (2020).

\bibitem{mantzaflaris2017low}
A.~Mantzaflaris, B.~J{\"u}ttler, B.~N. Khoromskij, U.~Langer, {Low rank tensor
  methods in Galerkin-based isogeometric analysis}, Computer Methods in Applied
  Mechanics and Engineering 316 (2017) 1062--1085.

\bibitem{scholz2018partial}
F.~Scholz, A.~Mantzaflaris, B.~J{\"u}ttler, {Partial tensor decomposition for
  decoupling isogeometric Galerkin discretizations}, Computer Methods in
  Applied Mechanics and Engineering 336 (2018) 485--506.

\bibitem{hofreither2018black}
C.~Hofreither, {A black-box low-rank approximation algorithm for fast matrix
  assembly in isogeometric analysis}, Computer Methods in Applied Mechanics and
  Engineering 333 (2018) 311--330.

\bibitem{chan2018multi}
J.~Chan, J.~A. Evans, {Multi-patch discontinuous Galerkin isogeometric analysis
  for wave propagation: Explicit time-stepping and efficient mass matrix
  inversion}, Computer Methods in Applied Mechanics and Engineering 333 (2018)
  22--54.

\bibitem{van1993approximation}
C.~F. Van~Loan, N.~Pitsianis, {Approximation with Kronecker products}, in:
  Linear algebra for large scale and real-time applications, Springer, 1993,
  pp. 293--314.

\bibitem{pitsianis1997kronecker}
N.~P. Pitsianis, {The Kronecker product in approximation and fast transform
  generation}, Ph.D. thesis, Cornell University (1997).

\bibitem{calabro2017fast}
F.~Calabro, G.~Sangalli, M.~Tani, {Fast formation of isogeometric Galerkin
  matrices by weighted quadrature}, Computer Methods in Applied Mechanics and
  Engineering 316 (2017) 606--622.

\bibitem{antolin2015efficient}
P.~Antolin, A.~Buffa, F.~Calabro, M.~Martinelli, G.~Sangalli, {Efficient matrix
  computation for tensor-product isogeometric analysis: The use of sum
  factorization}, Computer Methods in Applied Mechanics and Engineering 285
  (2015) 817--828.

\bibitem{pan2022efficient}
M.~Pan, B.~J{\"u}ttler, F.~Scholz, {Efficient matrix computation for
  isogeometric discretizations with hierarchical B-splines in any dimension},
  Computer Methods in Applied Mechanics and Engineering 388 (2022) 114210.

\bibitem{nagy2006kronecker}
J.~G. Nagy, M.~E. Kilmer, {Kronecker product approximation for preconditioning
  in three-dimensional imaging applications}, IEEE Transactions on Image
  Processing 15~(3) (2006) 604--613.

\bibitem{langville2004akronecker}
A.~N. Langville, W.~J. Stewart, {A Kronecker product approximate preconditioner
  for SANs}, Numerical Linear Algebra with Applications 11~(8-9) (2004)
  723--752.

\bibitem{ullmann2010kronecker}
E.~Ullmann, {A Kronecker product preconditioner for stochastic Galerkin finite
  element discretizations}, SIAM Journal on Scientific Computing 32~(2) (2010)
  923--946.

\bibitem{de2000multilinear}
L.~De~Lathauwer, B.~De~Moor, J.~Vandewalle, {A multilinear singular value
  decomposition}, SIAM journal on Matrix Analysis and Applications 21~(4)
  (2000) 1253--1278.

\bibitem{deng2021boundary}
Q.~Deng, V.~M. Calo, {A boundary penalization technique to remove outliers from
  isogeometric analysis on tensor-product meshes}, Computer Methods in Applied
  Mechanics and Engineering 383 (2021) 113907.

\bibitem{hiemstra2021removal}
R.~R. Hiemstra, T.~J. Hughes, A.~Reali, D.~Schillinger, {Removal of spurious
  outlier frequencies and modes from isogeometric discretizations of second-and
  fourth-order problems in one, two, and three dimensions}, Computer Methods in
  Applied Mechanics and Engineering 387 (2021) 114115.

\bibitem{gallistl2017stability}
D.~Gallistl, P.~Huber, D.~Peterseim, {On the stability of the Rayleigh--Ritz
  method for eigenvalues}, Numerische Mathematik 137~(2) (2017) 339--351.

\bibitem{sande2019sharp}
E.~Sande, C.~Manni, H.~Speleers, {Sharp error estimates for spline
  approximation: Explicit constants, n-widths, and eigenfunction convergence},
  Mathematical Models and Methods in Applied Sciences 29~(06) (2019)
  1175--1205.

\bibitem{manni2022application}
C.~Manni, E.~Sande, H.~Speleers, {Application of optimal spline subspaces for
  the removal of spurious outliers in isogeometric discretizations}, Computer
  Methods in Applied Mechanics and Engineering 389 (2022) 114260.

\bibitem{floater2019optimal}
M.~S. Floater, E.~Sande, {Optimal Spline Spaces for L2 n-Width Problems with
  Boundary Conditions}, Constructive Approximation 50~(1) (2019) 1--18.

\end{thebibliography}
\end{document}